\documentclass[11pt]{amsart}
\usepackage{cite}
\usepackage{graphicx}
\usepackage{latexsym,amsmath,amsfonts,amscd, amsthm}
\usepackage{color}
\usepackage[colorlinks=true]{hyperref}
\hypersetup{urlcolor=blue, citecolor=red}
\usepackage{tikz}
\usetikzlibrary{arrows,backgrounds,snakes}
\usepackage{epsfig}
\usepackage{dsfont}
\usepackage{amssymb}
\usepackage{xcolor}
\usepackage{comment}
\usepackage{subfigure}
\usepackage{multirow}
\usepackage{mathrsfs}
\usepackage{diagbox}
\usepackage{cancel}
\usepackage{bm}
\usepackage{subcaption}
\newtheoremstyle{plainNoItalics}{}{}{\normalfont}{}{\bfseries}{.}{ }{}
\theoremstyle{plain}
\newtheorem{thm}{Theorem}[section]

\newtheorem{lem}[thm]{Lemma}
\newtheorem{defn}[thm]{Definition}
\newtheorem{rem}[thm]{Remark}
\newtheorem{prop}[thm]{Proposition}
\newtheorem{exa}[thm]{Example}

\allowdisplaybreaks[3]

\newcommand{\beq}{\begin{equation}}
\newcommand{\eeq}{\end{equation}}
\newcommand{\beqa}{\begin{eqnarray}}
\newcommand{\eeqa}{\end{eqnarray}}
\newcommand{\bit}{\begin{itemize}}
\newcommand{\eit}{\end{itemize}}
\newcommand{\bedef}{\begin{defn}}
\newcommand{\edefn}{\end{defn}}
\newcommand{\bpro}{\begin{prop}}
\newcommand{\epro}{\end{prop}}

\newcommand{\mfg}{{\mathbf{g}}}

\newcommand{\mfk}{{\mathbf{k}}}

\newcommand{\mfn}{{\mathbf{n}}}

\newcommand{\mfq}{{\mathbf{q}}}

\newcommand{\mfu}{{\mathbf{u}}}
\newcommand{\mfv}{{\mathbf{v}}}
\newcommand{\mfw}{{\mathbf{w}}}
\newcommand{\mfx}{{\mathbf{x}}}
\newcommand{\mfy}{{\mathbf{y}}}

\newcommand{\mfC}{{\mathbf{C}}}

\newcommand{\mfF}{{\mathbf{F}}}
\newcommand{\mfG}{{\mathbf{G}}}

\newcommand{\mfI}{{\mathbf{I}}}

\newcommand{\mfL}{{\mathbf{L}}}

\newcommand{\mfP}{{\mathbf{P}}}

\newcommand{\mfS}{{\mathbf{S}}}

\newcommand{\mfU}{{\mathbf{U}}}
\newcommand{\mfV}{{\mathbf{V}}}
\newcommand{\mfW}{{\mathbf{W}}}

\newcommand{\dd}{{\text{d}}}

\newcommand{\bdp}[1]{\left(#1\right)}

\newcommand{\bdc}[1]{\left\{#1\right\}}

\newcommand{\restr}[2]{\left.\kern-\nulldelimiterspace #1 \vphantom{\big|} \right|_{#2}}

\newcommand{\av}[1]{\{\!\!\{#1\}\!\!\}}
\newcommand{\jump}[1]{[\![#1]\!]}

\newcommand{\pdrv}[2]{\dfrac{\partial{#1}}{\partial{#2}}}
\newcommand{\pdrvn}[3]{\dfrac{\partial^{#3}{#1}}{\partial{#2}^{#3}}}

\newcommand{\mD}{{\mathcal D}}

\newcommand{\mE}{{\mathcal E}}

\newcommand{\mT}{{\mathcal T}}

\newcommand{\bv}{{\bf v}}

\newcommand{\bu}{{\bf u}}

\title[High order WB TEC LDG scheme for compressible SG Euler equations]{High order well-balanced and total-energy-conserving local discontinuous Galerkin methods for compressible self-gravitating Euler equations}


\begin{document}

\maketitle
\medskip
\centerline{\scshape Liang Pan}
\medskip
{\footnotesize
	\centerline{School of Mathematical Sciences, Xiamen University}
	\centerline{Xiamen, Fujian, 361005, P.R. China}
	\centerline{Email: lightpl@stu.xmu.edu.cn}
}

\medskip
\centerline{\scshape Wei Chen}
\medskip
{\footnotesize
	\centerline{School of Mathematical Sciences, Xiamen University}
	\centerline{Xiamen, Fujian, 361005, P.R. China}
	\centerline{Email: weichenmath@stu.xmu.edu.cn}
}

\medskip
\centerline{\scshape Jianxian Qiu}
\medskip
{\footnotesize
   % please put the address of the author
	\centerline{School of Mathematical Sciences, Xiamen University}
	\centerline{Fujian Provincial Key Laboratory of Mathematical Modeling and High-Performance Scientific Computing}
	\centerline{Xiamen, Fujian, 361005, P.R. China}
	\centerline{Email: jxqiu@xmu.edu.cn}
}

\medskip
\centerline{\scshape Tao Xiong\footnote{Corresponding author. The work is partially supported by National Key R\&D Program of China No.
		2022YFA1004500, NSFC grant No. 92270112, NSF of Fujian Province No. 2023J02003.}}
\medskip
{\footnotesize
	% please put the address of the author
	\centerline{School of Mathematical Sciences, University of Science and Technology of China}
	\centerline{Hefei, Anhui, 230026, P.R. China}
	\centerline{Email: taoxiong@ustc.edu.cn}
}

\bigskip

\begin{abstract}
	In this paper, we develop a high order structure-preserving local discontinuous Galerkin (DG) scheme for the compressible self-gravitating Euler equations, which pose great challenges due to the presence of time-dependent gravitational potential. The designed scheme is well-balanced for general polytropic equilibrium state and total energy conserving for multiple spatial dimensions without an assumption of spherical symmetry. The well-balanced property is achieved by decomposing the gravitational potential into equilibrium and perturbation parts, employing a modified Harten–Lax–van Leer–contact flux and a modification of the discretization for the source term. Conservation of total energy is particularly challenging in the presence of self-gravity, especially when aiming for high order accuracy. To address this, we rewrite the energy equation into a conservative form, and carefully design an energy flux with the aid of weak formulation from the DG method to maintain conservation as well as high order accuracy. The resulting scheme can be extended to high order in time discretizations. Numerical examples for two and three dimensional problems are provided to verify the desired properties of our proposed scheme, including shock-capturing, high order accuracy, well balance, and total energy conservation.
\end{abstract}

\vspace{0.1cm}

\section{Introduction}
\label{sec1}
\setcounter{equation}{0}
\setcounter{figure}{0}
\setcounter{table}{0}
%Over the past few decades, numerous publicly available simulation codes have been developed and widely adopted, including notable ones such as Zeus-2D \cite{stone1992zeus}, Athena \cite{stone2008athena}, and the more recent Athena++ \cite{2020Athena++, mullen2021extension}. They concentrated on the standard finite volume (FV) methods and adaptive mesh refinement framework.

Numerical simulations of self-gravitating (SG) gas dynamics have become a crucial tool for exploring a wide range of astrophysical phenomena. The SG effect is particularly important in modeling phenomena such as Jeans instability \cite{PurelyDG_SG2021,JIANG201348,mullen2021extension}, Sedov explosions \cite{PurelyDG_SG2021,maciel2015introduction,Zhang_2022,KAPPELI2014WB}, and star formation \cite{Simon_2016}. Specifically, the compressible SG Euler equations in the $d-$dimensional case can be written as:
\begin{subequations}
	\label{S2:eq:0}
	\begin{equation}
		\frac{\partial }{\partial t}
		\begin{bmatrix}
			\rho  \\
			\rho \mathbf{u}\\
			E     \\
		\end{bmatrix}
		+
		\nabla \cdot
		\begin{bmatrix}
			\rho \mathbf{u}\\
			\rho \mathbf{u} \otimes \mfu + p \mfI_d\\
			(E + p) \mfu
		\end{bmatrix}
		= -
		\begin{bmatrix}
			0\\
			\rho \nabla \phi\\
			\rho \mfu \cdot \nabla \phi
		\end{bmatrix},
		\label{S2:eq:1}
	\end{equation}
	\begin{equation}
		\Delta\phi = 4\pi G\rho. \label{S2:eq:2}
	\end{equation}
\end{subequations}
Here $\rho$, $\mfu = (u_1, \cdots, u_d)^T$, $E$, $p$, $\phi$ are the fluid mass density, velocity, total non-gravitational energy, gas pressure, and gravitational potential, respectively. $\mfI_d$ denotes the $d-$dimensional identity matrix. The total non-gravitational energy is given by
\begin{equation}
	E = \rho e + \frac{1}{2}\rho {\Vert \mfu \Vert}^2,
	\label{S2:eq:3}
\end{equation}
where $e = \mathcal{E}(\rho,p)$ is the specific internal energy related to $\rho,p$ from equation of state (EOS), and  $\Vert \cdot \Vert$ denotes the typical Euclidean vector norm. For the ideal gas law, the pressure $p$ is given by $p = (\gamma - 1)\rho e$, where $\gamma$ is the specific heat ratio. $G$ in equation 
\eqref{S2:eq:2} is the gravitational constant.

% Compared to the Euler equations with static gravitational potential, the gravitational potential is modeled by the Poisson equation \eqref{S2:eq:2}, which complicates the formulation and numerical implementation. The main objective of this paper is to develop high order local discontinuous Galerkin (LDG) schemes, which are well-balanced (WB) and total-energy-conserving (TEC) for the compressible SG Euler equations.

The major difference between the SG Euler equations and the Euler equations with static gravitational potential is the time-dependent gravitational source term, i.e. the gravitational potential is modeled by the Poisson equation \eqref{S2:eq:2}, which complicates the formulation and numerical implementation. One approach is to rewrite the Poisson equation as a first-order hyperbolic system with a relaxation time, therefore the solver designed for Euler equations can also be applied to the Poisson equation \cite{PurelyDG_SG2021}. This strategy reduces the computational effort compared to solving a large linear system from the Poisson equation, particularly by improving the efficiency on adaptive meshes. Another is to directly solve the Poisson equation. Fast Fourier Transform (FFT) algorithms \cite{JIANG201348} and multipole expansions \cite{2013MultipoleExpansion,1995MullerSGEuler} are widely used. However, this approach would impose strict constraints on boundary conditions. In this paper, we construct our scheme under the framework of local discontinuous Galerkin (LDG) methods, which are specifically designed to handle partial differential equations (PDEs) involving high order derivatives and are flexible for various types of boundary conditions in complicate geometries. The LDG methods have been widely used for problems involving the Poisson equation, such as Vlasov-Poisson equations \cite{Dios3dimECDGforVP,Dios2012HIGHOA,MadauleNDG_EC_VP2014,Filbet2022EC_DG_VP}, Burgers-Poisson equations \cite{Liu2015}, and Drift-Diffusion equations \cite{Liu2016DriftDiffusionIMEX,Liu2010DriftDiffusion}, etc. 

The compressible SG Euler equations belong to the family of hyperbolic conservation laws with source terms. In these applications, near-equilibrium flows often encounter, and conventional schemes may fail to capture small perturbations from the hydrostatic equilibrium states due to numerical discretization errors, unless a very refined mesh is applied, which leads to significant computational cost. In such cases, well-balanced (WB) schemes become important which satisfy the following properties: preserving the steady solutions exactly up to machine error; capturing small perturbations well even on relatively coarse meshes; preserving the accuracy of the proposed schemes. For WB schemes, it is crucial to consider a balance between the flux and source terms. In recent decades, WB schemes have drawn much attention. Many early works were proposed for shallow water equations (SWEs) over a non-flat bottom topology, and we refer the reader to \cite{Bremudez19941049,AudusseSWE2004,XING2005FdWbSwe,LEVEQUE1998346,Greenberg1996,Xing2010PpWbDgSwe,Xing2014WbDgMovingWater} for different kinds of WB reconstruction. These approaches have also been extended to compressible Euler equations with a static gravitational source term, within the framework of finite difference (FD) methods \cite{li2018FD_WB,Xing2012HighOW}, finite volume (FV) methods \cite{KAPPELI2014WB,BOTTA2004WB,SecondFVPPEulerKlingenberg,Käppeli2019HWB,REN2023FVPP,Berberich2021WBFV,2015WB2rd,Varma2019WB2nd}, discontinuous Galerkin (DG) methods \cite{LI2018WB,wu2021uniformly,DU2024WBPP,Li2015WellBalancedDG,Weijie2022WBDG}, etc. A study focusing on the comparison between WB schemes and non-WB schemes can be found in \cite{2019ComparisonWBandNonwb}. The main challenge of WB schemes for the time dependent source term arises from the reformulation of gravitational potential, where the balance between pressure gradient and the source term, i.e. $-\rho\nabla\phi = \nabla p$, cannot be directly used.

%One of the popular approach in designing WB schemes is hydrostatic reconstruction, which is  proposed for SWE in \cite{AudusseSWE2004} and can be extended to Euler equations similarly. The basic idea is to separate the hydrostatic states and perturbation states numerically. In this framework, the numerical flux is constructed through hydrostatic reconstruction, combined with the modifications to the source term, where the hydrostatic component ensures the WB property of the scheme, while the perturbation component maintains the desired accuracy of the numerical solution. It is an important idea in designing WB schemes for the lake at rest steady states \cite{2006WBFVaClassbalancelaw,Xing2010PpWbDgSwe,AudusseSWE2004} and the moving water equilibrium \cite{Xing2014WbDgMovingWater,2007WBHydroRecon}
%With the consistent numerical flux, such as Harten-Lax-van Leer-contact (HLLC) numerical flux, the schemes can be proved to be WB. In \cite{LI2018WB}, 

In addition to the WB property, TEC is also crucial for the simulation of the SG Euler equations. Since the gravitational force is self-generated by the gas, the gravitational potential energy can be regarded as part of the total energy. For SG Euler equations under a hydrostatic equilibrium, the gravitational energy may take negative values, so that the total energy could be much smaller than the internal energy. Numerical schemes without TEC often yield nonphysical results in certain cases \cite{JIANG201348}, particularly in the context of underresolved meshes or large time-step sizes.
TEC has also been observed to play a pivotal role in plasma models, including Vlasov–Amp\`{e}re equations \cite{CHEN2011EC_VA,ECforVA2024,CHENG2014630,LIU2023EC_SL_VA}, Vlasov–Poisson equations \cite{Dios3dimECDGforVP,Dios2012HIGHOA,MadauleNDG_EC_VP2014,Filbet2022EC_DG_VP}, and Vlasov–Maxwell equations \cite{Cheng2014EC_DG_VM,MARKIDIS2011EC_PIC_VM,YIN2023EV_Moment_VW}. In these applications, designing TEC schemes for Poisson-type equations presents greater challenges as compared to Amp\`{e}re-type equations, as there is no direct evolving equation for the external potential. In particular, the compressible SG Euler equations can be interpreted as the macroscopic limit of the gravitational Vlasov–Boltzmann–Poisson equations \cite{Brenier2000VPtoEuler,Nicolas2016BGKVPlimits,Dimarco2014APVPBtoPlasma} in the regime of vanishing dimensionless Knudsen numbers. For SG Euler equations, one popular technique for designing TEC schemes is based on the total energy conservation form, but their implementations are limited to second-order accuracy \cite{JIANG201348,mullen2021extension}. Another one is to apply summation-by-parts and the mass conservation equation to discretize the source term in the energy equation \cite{2023SPFEforEulerPoisson}, which would require careful modifications when extending to high order in time discretizations \cite{Zhang_2022}. In this paper, we adopt the first technique and will extend it to high order accuracy.

For compressible SG Euler equations, one recent work by \cite{Zhang_2022} addressed the satisfaction of both WB and TEC properties through a hydrostatic reconstruction approach applied to the flux term. Their proposed scheme was developed within the framework of DG methods, to ensure the consistency between the Euler equations and the Poisson equation. However, their scheme was constrained to the case of spherical symmetry, where the Poisson equation can be solved analytically by direct integration. In contrast, in high-dimensional spatial domains without spherical symmetry the analytical solution is not available, which makes the generalization of WB property to multiple spatial dimensions being much more complicated \cite{Zhang_2022}.

% In this paper, 
% The original DG scheme was developed by Reed and Hill for neutron transport equations \cite{ReedDG_transport}.  Since then, it has gained widespread popularity in various scientific and engineering simulations due to several attractive features, including ease of parallelization, natural support for $hp$-adaptivity, and natural extension to complex geometries. The local discontinuous Galerkin (LDG) scheme represents a significant advancement within the DG framework, specifically designed to handle partial differential equations (PDEs) involving high order derivatives. The difficulty of LDG methods lies in the choice of numerical flux, which can significantly affect the stability and accuracy of the scheme. Given the versatility and effectiveness of the LDG methods, there exists an extensive body of literature covering its applications across numerous fields. For problems involving coupled Poisson equations, LDG methods have been widely used as they rewrite the higher-order equations into several first-order equations, such as elliptic equations \cite{Superconvergence_LDG_elliptic, Priori_analysis_LDG_elliptic, CASTILLO20061307, Castillo2002performance}, Vlasov-Poisson equations \cite{Dios3dimECDGforVP,Dios2012HIGHOA,MadauleNDG_EC_VP2014,Filbet2022EC_DG_VP}, Burgers-Poisson equations \cite{Liu2015}, and Drift-Diffusion equations \cite{Liu2016DriftDiffusionIMEX,Liu2010DriftDiffusion}. 
The main objective of this paper is to construct a high order DG scheme for compressible SG Euler equations, which can preserve the WB and TEC properties in high dimensional spaces. Specifically, we employ the following main techniques to achieve both WB and TEC properties as well as maintaining high order accuracy:
\begin{enumerate}
\itemsep=0pt
\item The WB property is obtained based on the idea of hydrostatic reconstruction. We decompose the gravitational potential into equilibrium and perturbation components.  Modifications are applied to the Harten-Lax-van Leer-contact (HLLC) flux and source term for the equilibrium part in order to hold the WB property. A standard treatment of the perturbation part is enough to ensure high order accuracy. Preserving the WB property would perform better in near equilibrium state.

\item The TEC property is realized based on the conservative form for the total energy. To address the loss of accuracy due to a time derivative appeared in the energy flux, we apply summation-by-parts to restore the desired order of accuracy. The resulting scheme can be easily generalized to high order temporal accuracy with multi-stage Runge-Kutta time discretizations. 

\item The high order accuracy in space is obtained by the DG framework. The DG methods can be flexibly designed with high order accuracy for general geometries in high dimensions. We also borrow the weak formulation of DG methods to restore the high order accuracy in the energy flux, which is not easy for a FV/FD method \cite{JIANG201348,mullen2021extension}. Besides, we adopt the oscillation eliminating (OE) technique introduced in \cite{2024PengOEDG} to control numerical oscillations and we reformulate the oscillation damping procedure without destroying the WB property.
\end{enumerate}  

The rest of the paper is organized as follows. In Section \ref{sec2}, we will briefly review the compressible SG Euler equations, the polytropic steady-states for the SG Euler equations, and the total-energy-conservation form. The  multi-dimensional HLLC flux, the standard LDG scheme, and the structure-preserving LDG scheme will be described in detail in Section \ref{sec3}, followed by an OE technique to control numerical oscillations. The WB and TEC properties are proven in Section \ref{sec4}. Extensive two- and three-dimensional numerical experiments are performed to validate accuracy, robustness, shock capturing ability, WB property, and TEC property in Section \ref{sec5}. A brief conclusion will be drawn in the last section.

%%%%%%%%%%%%%%%%%%%%%%%%%%%%%
%%%%%%%%%%%%%%%%%%%%%%%%%%%%%
%%%%%%%%%%%%%%%%%%%%%%%%%%%%%
%%%%%%%%%%%%%%%%%%%%%%%%%%%%%
%%%%%%%%%%%%%%%%%%%%%%%%%%%%%
%%%%%%%%%%%%%%%%%%%%%%%%%%%%%
\section{Mathmatical Model}
\label{sec2}
\setcounter{equation}{0}
\setcounter{figure}{0}
\setcounter{table}{0}

In this section, we introduce the compressible SG Euler equations, along with a discussion of their steady-state solutions and the TEC property. For the sake of easy presentation, we introduce the following notations:
\begin{equation}
	\mfU = \begin{bmatrix}
		\rho  \\
		\rho \mathbf{u}\\
		E     \\
	\end{bmatrix}, \mfF(\mfU) = \begin{bmatrix}
		\rho \mathbf{u}\\
		\rho \mathbf{u} \otimes \mfu + p \mfI_d\\
		(E + p) \mfu
	\end{bmatrix}, \mfS(\mfU, \nabla\phi) = 
	\begin{bmatrix}
		0\\
		-\rho \nabla \phi\\
		-\rho \bu \cdot \nabla \phi
	\end{bmatrix},
\end{equation}
where $\mfF$ and $\mfS$ denote the flux term and source term respectively. Then the equations \eqref{S2:eq:1} and \eqref{S2:eq:2} can be rewritten as follows:
\begin{equation}
	\mfU_t + \nabla\cdot\mfF(\mfU) = \mfS(\mfU,\nabla\phi),\quad \Delta\phi = 4\pi G \rho. \label{S2:eq:17}
\end{equation}

%\subsection{Compressible Self-gravitating Euler Equations} \hfill \\

% %%%%%%%%%%%%%%%%%%%%%%%%%%%%%
% %%%%%%%%%%%%%%%%%%%%%%%%%%%%%
% %%%%%%%%%%%%%%%%%%%%%%%%%%%%%

\subsection{Polytropic steady-states and Lane-Emden equation} \hfill \\

The compressible SG Euler equations admit the hydrostatic stationary solution. The pressure gradient force is balanced by the gravitational force:
\begin{equation}
	\rho = \rho(\mfx),\; \mfu = 0,\; \nabla p = -\rho\nabla\phi,\label{S2:eq:6}
\end{equation}
where $\mfx=(x,y,z)^T$ denotes the spatial coordinates.
Moreover, an additional Poisson equation (\ref{S2:eq:2}) has to be satisfied. A general polytropic hydrostatic equilibrium is considered in such self-gravity model and some astrophysical applications, which is characterized by:
\begin{equation}
	p = \kappa \rho^{\gamma},\label{S2:eq:7}
\end{equation}
where $\kappa$ is a positive constant. By combining equations~\eqref{S2:eq:2},\eqref{S2:eq:6}, and \eqref{S2:eq:7}, we obtain
\begin{equation}
	\nabla\cdot\left(\kappa\gamma\rho^{\gamma-2}\nabla\rho\right) = - 4\pi G\rho. \label{S2:eq:8}
\end{equation}
We introduce the following transformations:
\begin{equation}
	\rho = \lambda \theta^n,\quad \gamma = \dfrac{n+1}{n},\label{S2:eq:9}
\end{equation}
where $\lambda$ and $n$ are both positive constants, and $\theta$ is a function of $\mfx$. When $n \ne 0,1$, equation~\eqref{S2:eq:8} becomes nonlinear and does not admit a simple analytical solution. Substituting \eqref{S2:eq:9} into \eqref{S2:eq:8}, it can be simplified as:
\begin{equation}
	\dfrac{\kappa(n+1)\lambda^{\frac{1-n}{n}}}{4\pi G}\Delta \theta = -\theta^n. \label{S2:eq:10}
\end{equation}
We apply the following scaling:
\begin{equation*}
	a = \sqrt{	\dfrac{\kappa(n+1)\lambda^{\frac{1-n}{n}}}{4\pi G}}, \qquad \xi = \dfrac{x}{a}, \qquad \eta =  \dfrac{y}{a}, \qquad \mu = \dfrac{z}{a}.\label{S2:eq:11}
\end{equation*}
Thus the equation~\eqref{S2:eq:10} can be rewritten in the following non-dimensional form:
\begin{equation}
	\Delta'\theta:= \dfrac{\partial^2\theta}{\partial \xi^2} + \dfrac{\partial^2\theta}{\partial \eta^2} + 
	\dfrac{\partial^2\theta}{\partial \mu^2} = a^2\left(\dfrac{\partial^2\theta}{\partial x^2} + \dfrac{\partial^2\theta}{\partial y^2} + 
	\dfrac{\partial^2\theta}{\partial z^2}\right) = a^2\Delta \theta = -\theta^n.\label{S2:eq:12}
\end{equation}
In the spherically symmetric case, we introduce the radial coordinate $r = \sqrt{\xi^2+\eta^2+\mu^2}$, and the equation~\eqref{S2:eq:12} reduces to the classical Lane-Emden equation\cite{maciel2015introduction}:
\begin{equation}
	\dfrac{1}{r^2}\pdrv{ }{r}\bdp{r^2\pdrv{\theta}{r}} = - \theta^n.\label{S2:eq:13}
\end{equation} 
The Lane-Emden equation can be analytically solved \cite{maciel2015introduction} only for a few special integer values of the index $n$,
\begin{equation}
	\begin{cases}
		\theta_0(r) = 1 - \dfrac{1}{6}r^2, & n = 0,i.e.\, \gamma = \infty,\\[8pt]
		\theta_1(r) = \dfrac{\sin(r)}{r}, & n = 1,i.e.\, \gamma = 2,\\[8pt]
		\theta_5(r) = \dfrac{1}{\sqrt{1+\frac{1}{3}r^2}}, & n = 5,i.e.\, \gamma = \dfrac{6}{5}.
	\end{cases}\label{S2:eq:14}
\end{equation}

However, the polytropic hydrostatic equilibrium in two dimensional  radial symmetry case is quite different. To avoid introducing too many new symbols, we reuse the notation $r = \sqrt{\xi^2+\eta^2}$, then we can obtain
\begin{equation}
	\dfrac{1}{r}\pdrv{ }{r}\bdp{r\pdrv{\theta}{r}} = \pdrvn{\theta}{r}{2} + \dfrac{1}{r}\pdrv{\theta}{r} = - \theta^n. \label{lane2D}
\end{equation}
Actually, when $n = 1$, the equation $\eqref{lane2D}$ is the $0th$ Bessel ordinary equation. It has no explicit solution but a integral form:
\begin{equation*}
	\theta(r) = \dfrac{1}{\pi}\int_{0}^{\pi}\cos(r\sin \alpha)\dd{\alpha}.
\end{equation*} 
In our numerical tests, it can be computed by a high order quadrature rule.

We use the superscript ``$e$'' to denote functions corresponding to the steady state. As in the case of the compressible Euler equations with a static gravitational potential, the equilibrium state is described by the density $\rho^e$, velocity $\mfu^e$, and pressure $p^e$. In contrast, the equilibrium state of the SG Euler equations must additionally involve the gravitational potential $\phi^e(\mfx)$. For the solution $\theta(\mfx)$ of \eqref{S2:eq:10}, the equilibrium state has following analytic form:
\begin{equation}
	\rho^e = \lambda\bdp{\theta(\mfx)}^n,\; \mfu^e = 0,\; p^e = \kappa (\rho^e)^\gamma,\; \phi^e = -\dfrac{\kappa\gamma}{\gamma - 1}(\rho^e)^{\gamma - 1}.
    \label{eqstate}
\end{equation}

%\begin{rem}
%	\textcolor{blue}{	We note that the Poisson equation (\ref{S2:eq:2}) is only valid in three-dimensional case. The gravitational potential has an explicit form which is given by the well-known Green's function
%		\begin{equation*}
%			\phi(\mfx) = - \int_{\mathbb{R}^3}\dfrac{G\rho(\mfx')}{\Vert \mfx - \mfx' \Vert}\dd{\mfx'}.\label{S2:eq:4}
%		\end{equation*}
%		The Laplace operator of $1/\Vert \mfx - \mfx' \Vert$ is a delta function only in three-dimension, i.e.
%		\begin{equation}
%			\Delta \bdp{\dfrac{1}{\Vert \mfx - \mfx' \Vert}} = -4\pi \delta(\mfx - \mfx'). \label{S2:eq:5}
%		\end{equation}
%		In the two-dimensional case, although \eqref{S2:eq:5} does not hold, the mathematical model can still be formulated based on the Poisson equation.}
%
%\end{rem}
% %%%%%%%%%%%%%%%%%%%%%%%%%%%%%
% %%%%%%%%%%%%%%%%%%%%%%%%%%%%%
% %%%%%%%%%%%%%%%%%%%%%%%%%%%%%
\subsection{Total-energy-conserving Property} \hfill \\

The total energy $E_{tot}$ is defined as the sum of the total non-gravitational energy $E$ and the gravitational energy $\frac{1}{2}\rho\phi$. Thus we can rewrite the energy equation of (\ref{S2:eq:1}) in a conservative form:
\begin{equation}
	\pdrv{}{t}\bdp{E+\dfrac{1}{2}\rho\phi} + \nabla\cdot\bdp{(E+p)\mfu + \mfF_g} = 0,\label{S2:eq:15}
\end{equation}
with the energy flux $\mfF_g$:
\begin{equation}
	\mathbf{F}_g = \dfrac{1}{8\pi G}\bdp{\phi\nabla\left(\dfrac{\partial \phi}{\partial t}\right) -\dfrac{\partial \phi}{\partial t}\nabla\phi} +\rho\mfu\phi.
\end{equation}
The energy conservative form \eqref{S2:eq:15} can be obtained by several simple steps:
\begin{equation}
	\begin{split}
		&\pdrv{}{t}\bdp{\dfrac{1}{2}\rho\phi} + \nabla\cdot\mfF_g - \rho\mfu\nabla\phi\\
		= & \dfrac{1}{2}\pdrv{\rho}{t}\phi +\dfrac{1}{2} \pdrv{\phi}{t}\rho + \dfrac{1}{8\pi G}\bdp{\phi\Delta\bdp{\pdrv{\phi}{t}} - \pdrv{\phi}{t}\Delta\phi} + \nabla\cdot(\rho\mfu)\phi \\
		= & \dfrac{1}{8\pi G}\bdp{\phi\Delta\bdp{\pdrv{\phi}{t}} + \pdrv{\phi}{t}\Delta\phi} + \dfrac{1}{8\pi G}\bdp{\phi\Delta\bdp{\pdrv{\phi}{t}} - \pdrv{\phi}{t}\Delta\phi} - \dfrac{1}{4\pi G}\Delta\bdp{\pdrv{\phi}{t}}\phi \\
		=& 0.
	\end{split}\label{S2:eq:16}
\end{equation}
By adding the energy equation of \eqref{S2:eq:1} and \eqref{S2:eq:16} we can get the conservative form \eqref{S2:eq:15}.
Notice that there exists the time derivative term $\partial\phi/\partial t$ in $\mfF_g$. Inspired by \cite{JIANG201348}, we introduce $\dot{\phi} $ to take place of  $\partial\phi/\partial t$, which satisfy another Poisson equation:
\begin{equation}
	\Delta(\dot{\phi}) = \Delta\left(\dfrac{\partial \phi}{\partial t}\right) = 4\pi G \dfrac{\partial\rho}{\partial t} = -4\pi G \nabla\cdot(\rho\mfu). \label{poissonphidt}
\end{equation}
Therefore, the gravity flux can be written into a compact form:
\begin{equation}
	\mathbf{F}_g = \dfrac{1}{8\pi G}\bdp{\phi\nabla\dot{\phi} -\dot{\phi}\nabla\phi} +\rho\mfu\phi. \label{S2:eq:19}
\end{equation}
When periodic boundary condition or compact support boundary condition is applied on $\Omega$, the total energy conservation is derived:
\begin{equation}
		\dfrac{\dd}{\dd{t}}\int_{\Omega}\bdp{E+\dfrac{1}{2}\rho\phi}\dd{\mfx} = 0.
\end{equation}

%%%%%%%%%%%%%%%%%%%%%%%%%%%%%
%%%%%%%%%%%%%%%%%%%%%%%%%%%%%
%%%%%%%%%%%%%%%%%%%%%%%%%%%%%
%%%%%%%%%%%%%%%%%%%%%%%%%%%%%
%%%%%%%%%%%%%%%%%%%%%%%%%%%%%
%%%%%%%%%%%%%%%%%%%%%%%%%%%%%

\section{Numerical schemes}
\label{sec3}
\setcounter{equation}{0}
\setcounter{figure}{0}
\setcounter{table}{0}

In this section, we will present the high order, WB and TEC LDG scheme for the compressible SG Euler equations. The standard LDG scheme will also be presented. 

% In order to design the WB scheme to preserve the steady states with time-dependent gravitation potential, we propose to decompose the gravitational source term into equilibrium state and perturbation state. The idea is inspired by the hydrostatic reconstruction of flux, which has been widely used in the construction of WB schemes in \cite{LI2018WB,Zhang_2022}.

% %%%%%%%%%%%%%%%%%%%%%%%%%%%%%
% %%%%%%%%%%%%%%%%%%%%%%%%%%%%%
% %%%%%%%%%%%%%%%%%%%%%%%%%%%%%
\subsection{Notations}\label{3.1} \hfill \\

We first introduce following notations for LDG methods. We assume that the computation domain $\Omega\in\mathbb{R}^d$, where $d$ denotes the space dimension, is partitioned into mesh $\mT_h=\{K\}$ with its boundary $\Gamma_h$. The mesh is regular, i.e. there exists a positive constant $\sigma$ satisfying that
\begin{subequations}
	\begin{equation*}
		\max_{K\in\mT_h}\dfrac{h_K}{\rho_K}\leq\sigma,
	\end{equation*}
	\begin{equation*}
		h_K = \text{diam}(K),\quad \rho_K = \max\{\text{diam}(S):S\text{ is a ball contained in }K\},
	\end{equation*}
\end{subequations}
where $\text{diam}(K) = \sup\{||\mfy_1 - \mfy_2||,\forall \mfy_1,\mfy_2\in K\}$. 

We define several discontinuous finite element space of $V_h^k,\Sigma_h^k$ as:
\begin{subequations}
	\begin{equation}
		V_h^k = \{v\in L^2(\Omega):\,v|_K\in \mathbb{P}^k(K),\;\forall K\in\mT_h\},
	\end{equation}
	\begin{equation}
		\Sigma_h^k = \{\bv\in (L^2(\Omega))^d:\,\bv|_K\in (\mathbb{P}^k(K))^d,\;\forall K\in\mT_h\},
	\end{equation}
\end{subequations}
where $\mathbb{P}^k(K)$ can be either the space of polynomials of total degree up to $k$, i.e. $P^k$, or  the space of polynomials of degree in each component up to $k$, i.e. $Q^k$. We denote by $\mE_I$ the set of all interior faces of the triangulation $\mT_h$, while $\mE_B$ the set of the boundary of the triangulation $\mT_h$. Let $K^+$ and $K^-$ be two adjacent elements of $\mT_h$ and $\mfx$ is an arbitrary point on the $(d-1)$ dimensional face of $\mE = \partial K^+\cap\partial K^-\in \mE_I$. The vector $\mfn^+$ and $\mfn^-$ are the corresponding outward unit normals for the points on the face $\mE$. We also define the average $\av{\cdot}$ and jump $\jump{\cdot}$ of a function $u\in V_h^k$ and $\mfq\in \Sigma_h^k$ at $\mfx\in\mE_I$ as follows:
\begin{equation*}
\begin{split}
	\av{u} = \dfrac{1}{2}(u^+ + u^-), &\qquad \jump{u} =u^+\mfn^+ + u^-\mfn^-, \\
	\av{\mfq} = \dfrac{1}{2}(\mfq^+ + \mfq^-), &\qquad \jump{\mfq} = \mfq^+\cdot\mfn^+ + \mfq^-\cdot\mfn^-.
\end{split}\label{S3:eq:2}
\end{equation*}
Moreover, we define the operator $\mfP$ to represent the $L^2$ projection from $L^2(\Omega)$ onto $V_h^k$, i.e. for a function $u\in L^2(\Omega)$, then $\mfP u\in V_h^k$. It is defined by:
\begin{equation}
	\int_{K} u\,v\dd{\mfx} = \int_{K} \mfP u\,v\dd{\mfx},\qquad \forall v\in V_h^k,\forall K\in\mT_h.  \label{projector}
\end{equation}

% %%%%%%%%%%%%%%%%%%%%%%%%%%%%%
% %%%%%%%%%%%%%%%%%%%%%%%%%%%%%
% %%%%%%%%%%%%%%%%%%%%%%%%%%%%%
\subsection{The multi-dimensional HLLC flux} \hfill \\

The HLLC flux is based on the approximate Riemann solver and consider the contact wave compared to the Harten–Lax–van Leer (HLL) flux. The HLLC flux is defined by:
\begin{equation}
	\mfF^{\text{hllc}}(\mfU_L, \mfU_R;\mfn) = \begin{cases}
		\mfF(\mfU_L), &\text{ if }0\leq S_L,\\
		\mfF_{*L}, &\text{ if } S_L\leq 0 \leq S_*,\\
		\mfF_{*R}, &\text{ if } S_*\leq 0 \leq S_R,\\
		\mfF(\mfU_R), &\textbf{ if } 0\geq S_R,
	\end{cases}\label{S2:eq:18}
\end{equation}
where $S_L,S_R,S_*$ are respectively the estimated left signal velocity, right signal velocity, and middle wave speed. The unit vector $\mfn$ denotes the direction of Riemann problem. The state velocity $\widehat{u}_L,\widehat{u}_R$ along the direction $\mfn$ is defined by $\widehat{u} = \mfu\cdot\mfn$. With sound speed $c = \sqrt{\gamma p/\rho}$, they can be computed by:
\begin{subequations}
	\begin{equation*}
		S_L = \min\{\widehat{u}_L - c_L, \widehat{u}_R - c_R\},\quad S_R = \max\{\widehat{u}_L + c_L, \widehat{u}_R + c_R\},
	\end{equation*}
	\begin{equation*}
		S_* = \dfrac{p_R - p_L + \rho_L\widehat{u}_L(S_L-\widehat{u}_L) - \rho_R\widehat{u}_R(S_R-\widehat{u}_R)}{\rho_L(S_L-\widehat{u}_L) - \rho_R(S_R-\widehat{u}_R)}.
	\end{equation*}
\end{subequations}
The fluxes $\mfF_{*L}$ and $\mfF_{*R}$ are given by:
\begin{subequations}
	\begin{equation*}
		\mfF_{*i} = \mfF_i + S_i(\mfU_{*i} - \mfU_i),\quad i=L,R,
	\end{equation*}
	\begin{equation*}
		\mfU_{*i} = \rho_i\bdp{\dfrac{S_i - \widehat{u}_i}{S_i - S_*}}\begin{pmatrix}
			1 \\
			\mfu_i \\ 
			\frac{E_i}{\rho_i} + (S_*-\widehat{u}_i)(S_*+\frac{p_i}{\rho_i(S_i-\widehat{u}_i)})
		\end{pmatrix}.
	\end{equation*}
\end{subequations}
Compared to the HLL flux, the HLLC flux considers the effect of contact and possesses the so-called contact property as outlined below:
\begin{lem}
	\label{contactdiscontinuity}
	For any two states $\mfU_L = (\rho_L,\mathbf{0},p/(\gamma-1))^T$ and $\mfU_R = (\rho_R,\mathbf{0},p/(\gamma-1))^T$, the multidimensional HLLC flux satisfies:
	\begin{equation}
		\mfF^{{\rm hllc}}(\mfU_L, \mfU_R;\mfn) = (0,p\mfn^T,0)^T.\label{S2:L1}
	\end{equation}
\end{lem}
The proof of Lemma \ref{contactdiscontinuity} is straightforward and is omitted here. 
For more details of HLLC flux, we refer the readers to \cite{Batten1997multidimensionalHLLC,Toro2009RiemannSolver}.
The above contact property is very important in obtaining WB property.

To clearly demonstrate the importance of structure-preserving in the numerical section, we will describe two types of numerical schemes:
\begin{enumerate}
	\item \textbf{The standard scheme}: a baseline scheme that does not account for the WB and TEC properties;
	\item \textbf{The structure-preserving scheme}: our carefully designed scheme that enjoys both WB and TEC properties.
\end{enumerate}

\subsection{The standard LDG scheme} \label{3.2} \hfill \\

We begin by briefly introducing the standard LDG scheme for SG Euler equations without consideration of WB or TEC properties. Following the idea of LDG scheme, we add the auxiliary term $\mfg = \nabla\phi$, representing the gravitational force. The compressible SG Euler equations \eqref{S2:eq:1}-\eqref{S2:eq:2} can be written as follows:
\begin{equation}
	\left\{
	\begin{aligned}
		&\mfU_t + \nabla\cdot\mfF(\mfU) = \mfS(\mfU,\mfg),\\
		&\mfg = \nabla\phi,\\
		&\nabla\cdot\mfg = 4\pi G\rho.
	\end{aligned}\right. \label{S3:eq:3}
\end{equation}
The standard semi-discrete LDG scheme for \eqref{S3:eq:3} is given as follows: Find $\mfU_h\in [V_h^k]^{d+2}$, $\phi\in V_h^k$, $\mfg\in \Sigma_h^k$, such that for any test function $v\in V_h^k$, $\mfw\in \Sigma_h^k$, $\psi\in V_h^k$, the following holds:

\begin{subequations}
\begin{align}
	\int_K(\mfU_h)_tv\dd{\mfx} &- \int_K\mfF(\mfU_h)\cdot\nabla v\dd{\mfx} + \sum_{\mE\in \partial K}\int_{\mE}\widehat{\mfF}_{\mfn_{\mE,K}}v^{\text{int}(K)}\dd{s} = \int_K\mfS(\mfU_h,\mfg_h)v\dd{\mfx}, \label{S3:eq:4}\\
	\int_K\mfg_h\cdot\mfw\dd{\mfx} &= \sum_{\mE\in \partial K}\int_{\mE}\widehat{\phi_h}\mfw^{\text{int}(K)}\cdot\mfn_{\mE,K}\dd{s} - \int_{K}\phi_h\nabla\cdot\mfw\dd{\mfx},	\label{S3:eq:5}\\
	\int_K\mfg_h\cdot\nabla\psi\dd{\mfx} &= \sum_{\mE\in \partial K}\int_{\mE}\psi^{\text{int}(K)}\widehat{\mfg_h}\cdot\mfn_{\mE,K}\dd{s} - \int_{K}4\pi G\rho_h\psi\dd{\mfx},\label{S3:eq:6}
\end{align}\label{scheme:1}
\end{subequations}
where $\partial K$ denotes the boundary of cell $K$, and  $\mfn_{\mE,K}$ is the outward unit normal to the edge of $\mE$ of $K$. The superscripts $``\text{int}(K)"$ and $``\text{ext}(K)"$ indicate the limit value at the cell interface taken from the interior or the exterior of $K$. The numerical flux $\widehat{\mfF}_{\mfn_{\mE},K}$ in the Euler equations is chosen as HLLC flux \eqref{S2:eq:18}, i.e.
\begin{equation*}
	\widehat{\mfF}_{\mfn_{\mE,K}} = \mfF^{{\rm hllc}}\bdp{\mfU^{{\rm int}(K)}, \mfU^{{\rm ext}(K)};\mfn_{\mE,K}}.
\end{equation*}
Following the idea of LDG scheme for Poisson equation in \cite{Superconvergence_LDG_elliptic,Priori_analysis_LDG_elliptic}, we introduce two penalty parameters $C_{11}$ and $\mfC_{12}$, thus the numerical flux $\widehat{\phi_h}$ and $\widehat{\mfg_h}$ in \eqref{S3:eq:5} and \eqref{S3:eq:6} are defined by:
\begin{subequations}
\label{S3:eq:14}
\begin{align}
	\widehat{\mfg_h} &= \av{\mfg_h} - C_{11}\jump{\phi_h} - \mfC_{12}\jump{\mfg_h},\label{S3:eq:15}\\
	\widehat{\phi_h} &= \av{\phi_h} + \mfC_{12}\cdot\jump{\phi_h}.\label{S3:eq:16}
\end{align}      
\end{subequations}
In this papar, we compute on the tensor mesh and set the parameter $C_{11} = 1$ and $\mfC_{12} = (0.5,0.5,0.5)^T$ when $d=3$, while $C_{11} = 1$ and $\mfC_{12} = (0.5,0.5)^T$ when $d=2$. After applying the Schur complement, the stiffness matrix is symmetric and positive definite. The initial conditions of \eqref{S3:eq:3} are given by $\{\rho^0(\mfx),(\rho\mfu)^0(\mfx),E^0(\mfx)\}$, and the initial conditions for the scheme \eqref{scheme:1} are given by the projection operator $\mfP$, i.e.
\begin{equation*}
	\left\{\rho_h^0(\mfx),\,(\rho\mfu)_h^0(\mfx), E_h^0(\mfx)\right\} = \left\{\mfP\rho^0(\mfx),\,\mfP(\rho\mfu)^0(\mfx),\mfP E^0(\mfx)\right\}.
\end{equation*}
The Euler-forward standard LDG scheme is outlined as follows:
\begin{subequations}
	\label{FullyDiscreteStandard}
	\begin{align}
		\int_K\bdp{\dfrac{\mfU^{n+1}_h -\mfU^{n}_h }{\Delta t}}v\dd{\mfx} &- \int_K\mfF(\mfU^n_h)\cdot\nabla v\dd{\mfx} + \sum_{\mE\in \partial K}\int_{\mE}\widehat{\mfF}^n_{\mfn_{\mE,K}}v^{\text{int}(K)}\dd{s} = \int_K\mfS(\mfU^n_h,\mfg^n_h)v\dd{\mfx}, \\
		\int_K\mfg^n_h\cdot\mfw\dd{\mfx} &= \sum_{\mE\in \partial K}\int_{\mE}\widehat{\phi^n_h}\mfw^{\text{int}(K)}\cdot\mfn_{\mE,K}\dd{s} - \int_{K}\phi^n_h\nabla\cdot\mfw\dd{\mfx},	\\
		\int_K\mfg^n_h\cdot\nabla\psi\dd{\mfx} &= \sum_{\mE\in \partial K}\int_{\mE}\psi^{\text{int}(K)}\widehat{\mfg^n_h}\cdot\mfn_{\mE,K}\dd{s} - \int_{K}4\pi G\rho^n_h\psi\dd{\mfx}.
	\end{align}
\end{subequations}

\begin{rem}
	Following a similar approach in \cite{CASTILLO20061307,Castillo2002performance}, the stiffness matrix arising from the LDG methods exhibits the following block structure:
	\begin{equation}
		\begin{bmatrix}
			K & B \\
			-B^T & S 
		\end{bmatrix}\begin{bmatrix}
			\mfg_h\\
			\phi_h
		\end{bmatrix} = \begin{bmatrix}
			g \\
			f
		\end{bmatrix}.\label{S3:eq:17}
	\end{equation}
	Directly solving the linear system \eqref{S3:eq:17} is challenging due to the nonsymmetric saddle-point structure. A more practical approach is to apply the Schur complement, reducing the problem to solving a smaller linear system:
	\begin{equation}
		A\phi_h = (S + B^TK^{-1}B)\phi_h = f + B^TK^{-1}\mfg_h.\label{stiffmatrix1}
	\end{equation} 
	In that case, the stiffness matrix $A$ in \eqref{stiffmatrix1} is symmetric and positive definite and a suitable Krylov subspace method can be applied, such as Conjugate Gradient method (CG) and preconditioned Conjugate Gradient method (PCG). Since the sparse matrix remains unchanged at each time step, we apply the direct solver, i.e., Cholesky factorization at the beginning and reuse it throughout the computations. However, for the periodic boundary condition, the stiffness matrix exhibits another block structure:
	\begin{equation*}
		\begin{bmatrix}
			K & B \\
			-C^T & S 
		\end{bmatrix}\begin{bmatrix}
			\mfg_h\\
			\phi_h
		\end{bmatrix} = \begin{bmatrix}
			g \\
			f
		\end{bmatrix}.\label{Poissonsolver_period}
	\end{equation*}
	Similarly, we solve the smaller linear system:
	\begin{equation}
		A\phi_h = (S + C^TK^{-1}B)\phi_h = f + C^TK^{-1}\mfg_h.\label{stiffmatrix2}
	\end{equation}
	Therefore, the stiffness matrix $A$ in \eqref{stiffmatrix2} will not be symmetric any more.  We adopt the LU decomposition here.
\end{rem}
% %%%%%%%%%%%%%%%%%%%%%%%%%%%%%
% %%%%%%%%%%%%%%%%%%%%%%%%%%%%%
% %%%%%%%%%%%%%%%%%%%%%%%%%%%%%
\subsection{The structure-preserving LDG scheme} \label{3.3}\hfill \\

However, the scheme \eqref{scheme:1} does not inherently satisfy the WB or TEC properties. In this section, we introduce our structure-preserving LDG scheme. To obtain WB property, we split the source term into equilibrium state and perturbation state. To obtain the TEC property, we rewrite the energy equation in SG Euler equations \eqref{S2:eq:0} in the conservative form \eqref{S2:eq:15}.  Specifically, the SG Euler equations can be rewritten as:
\begin{subequations}
\label{S3:eq:18}
\begin{align}
	&\pdrv{\rho}{t} + \nabla\cdot(\rho\mfu) = 0,\label{S3:eq:7}\\
	&\pdrv{\rho\mfu}{t} + \nabla\cdot(\rho\mfu\otimes\mfu+p\mathbf{I}) = -\rho\mfg^e -\rho\mfg^{\delta}, \label{S3:eq:8}\\
	&\dfrac{\partial }{\partial t}\left(E_{tot}\right) + \nabla\cdot\left((E+p)\mfu + \mfF_g\right) = 0, \label{S3:eq:9}
\end{align}
where the gravitational force $\mfg$ is decomposed into an equilibrium $\mfg^{e}$ and another perturbation $\mfg^{\delta}$ part, which satisfy
\begin{align}
	&\mfg^e = \nabla\phi^e,\; \nabla\cdot\mfg^e = 4\pi G\rho^e,\label{S3:eq:11:1}\\
	&\mfg^{\delta} = \nabla\phi^{\delta},\; \nabla\cdot\mfg^{\delta} = 4\pi G\rho^{\delta},\label{S3:eq:11:2}
\end{align}
where the gravitational potential $\phi$ and the density $\rho$ are correspondingly decomposed as
\begin{equation}
\phi =  \phi^e + \phi^{\delta},\; \rho = \rho^e + \rho^{\delta}\label{S3:eq:11:3}.
\end{equation}
Besides, the total energy $E_{tot}$ and energy flux $\mfF_g$ in \eqref{S3:eq:9} are defined as:
\begin{align}
	&\mathbf{F}_g = \dfrac{1}{8\pi G}\bdp{\phi\dot{\mfg} -\dot{\phi}\mfg} +\rho\mfu\phi, \label{S3:eq:10}\\
	&E_{tot} = E + \dfrac{1}{2}\rho\phi,
\end{align}
where $\dot{\phi}:=\partial\phi/\partial t$ satisfy another Poisson equation \eqref{poissonphidt}:
\begin{equation}
	\dot{\mfg} = \nabla\dot{\phi},\;\nabla\cdot \dot{\mfg} =  -4\pi G \nabla\cdot(\rho\mfu).\label{S3:eq:12}
\end{equation}
\end{subequations} 
We use the following notations for ease of presentation:
\begin{equation}
	\mfU = \begin{bmatrix}
		\rho  \\
		\rho \mathbf{u}\\
		E     \\
	\end{bmatrix},
	\mfW = \begin{bmatrix}
		\rho  \\
		\rho \mathbf{u}\\
		E_{tot}     \\
	\end{bmatrix}, \mfF(\mfW) = \begin{bmatrix}
		\rho \mathbf{u}\\
		\rho \mathbf{u} \otimes \mfu + p \mfI\\
		(E + p) \mfu
	\end{bmatrix}, \mfG(\mfW) = \begin{bmatrix}
	\mathbf{0}\\
	\mathbf{0}\\
	\mfF_g
	\end{bmatrix},\mfS = 
	\begin{bmatrix}
		0\\
		-\rho \nabla \phi\\
		0
	\end{bmatrix}.
\end{equation}
The semi-discrete structure-preserving LDG scheme for \eqref{S3:eq:18} can be written as follows: Find $\mfW_h\in [V_h^k]^{d+2}$, $\phi_h^e\in V_h^k$, $\phi_h^{\delta}\in V_h^k$, $\dot{\phi_h}\in V_h^k$, $\mfg_h^e\in \Sigma_h^k$, $\mfg_h^{\delta}\in \Sigma_h^k$, and $\dot{\mfg_h}\in \Sigma_h^k$, such that for any test function $v\in V_h^k$, $\mfw\in \Sigma_h^k$, $\psi\in V_h^k$, the following holds:
\begin{subequations}
\label{scheme:2}
\begin{equation}
	\begin{split}
		\int_K(\mfW_h)_tv\dd{\mfx} &- \int_K\mfF(\mfW_h)\cdot\nabla v\dd{\mfx} + \sum_{\mE\in \partial K}\int_{\mE}\widehat{\mfF}_{\mfn_{\mE,K}}v^{\text{int}(K)}\dd{s} \\
		&- \int_K\mfG(\mfW_h)\cdot\nabla v\dd{\mfx} + \sum_{\mE\in \partial K}\int_{\mE}\widehat{\mfG}_{\mfn_{\mE,K}}v^{\text{int}(K)}\dd{s}= \int_K\mfS v\dd{\mfx}, \label{S3:eq:21}
	\end{split}
\end{equation}
where $\bdp{\mfg^e_h,\phi^e_h}$ and $\bdp{\mfg^{\delta}_h,\phi^{\delta}_h}$ satisfy the following Poisson equation:
\begin{equation}
	\begin{split}
		\int_K\mfg^e_h\cdot\mfw\dd{\mfx} &= \sum_{\mE\in \partial K}\int_{\mE}\widehat{\phi^e_h}\mfw^{\text{int}(K)}\cdot\mfn_{\mE,K}\dd{s} - \int_{K}\phi^e_h\nabla\cdot\mfw\dd{\mfx},\\
		\int_K\mfg^e_h\cdot\nabla\psi\dd{\mfx} &= \sum_{\mE\in \partial K}\int_{\mE}\psi^{\text{int}(K)}\widehat{\mfg^e_h}\cdot\mfn_{\mE,K}\dd{s} - \int_{K}4\pi G\rho^e_h\psi\dd{\mfx},\label{S3:eq:29}
	\end{split}
\end{equation}
\begin{equation}
	\begin{split}
		\int_K\mfg^{\delta}_h\cdot\mfw\dd{\mfx} &= \sum_{\mE\in \partial K}\int_{\mE}\widehat{\phi^{\delta}_h}\mfw^{\text{int}(K)}\cdot\mfn_{\mE,K}\dd{s} - \int_{K}\phi^{\delta}_h\nabla\cdot\mfw\dd{\mfx},\\
		\int_K\mfg^{\delta}_h\cdot\nabla\psi\dd{\mfx} &= \sum_{\mE\in \partial K}\int_{\mE}\psi^{\text{int}(K)}\widehat{\mfg^{\delta}_h}\cdot\mfn_{\mE,K}\dd{s} - \int_{K}4\pi G\rho^{\delta}_h\psi\dd{\mfx},\label{S3:eq:30}
	\end{split}
\end{equation}
\begin{equation}
	\phi_h = \phi_h^{\delta} + \phi_h^e,\;\mfg_h = \mfg_h^{\delta} + \mfg_h^e,\;\rho_h = \rho_h^{\delta} + \rho_h^e,
\end{equation}
and let $\bdp{\dot{\mfg_h},\dot{\phi_h}}$ be the solution of the following Poisson equation:
\begin{equation}
	\begin{split}
		\int_K\dot{\mfg_h}\cdot\mfw\dd{\mfx} &= \sum_{\mE\in \partial K}\int_{\mE}\widehat{\dot{\phi_h}}\mfw^{\text{int}(K)}\cdot\mfn_{\mE,K}\dd{s} - \int_{K}\dot{\phi_h}\nabla\cdot\mfw\dd{\mfx},\\
		\int_K\dot{\mfg_h}\cdot\nabla\psi\dd{\mfx} &= \sum_{\mE\in \partial K}\int_{\mE}\psi^{\text{int}(K)}\widehat{\dot{\mfg_h}}\cdot\mfn_{\mE,K}\dd{s} - \int_K 4\pi G \nabla\cdot(\rho\mfu)_h \psi \dd{\mfx},\label{S3:eq:31}
	\end{split}
\end{equation}
\end{subequations}
where the flux $\widehat{\phi^e_h}$, $\widehat{\mfg^e_h}$, $\widehat{\phi^{\delta}_h}$, $\widehat{\mfg^{\delta}_h}$, $\widehat{\dot{\phi_h}}$, and $\widehat{\dot{\mfg_h}}$ are the same as \eqref{S3:eq:14}, i.e.,
\begin{subequations}
\begin{align}
	\widehat{\mfg^e_h} &= \av{\mfg^e_h} - C_{11}\jump{\phi^e_h} - \mfC_{12}\jump{\mfg^e_h},\\
	\widehat{\phi^e_h} &= \av{\phi^e_h} + \mfC_{12}\cdot\jump{\phi^e_h},
\end{align}
\begin{align}
	\widehat{\mfg^{\delta}_h} &= \av{\mfg^{\delta}_h} - C_{11}\jump{\phi^{\delta}_h} - \mfC_{12}\jump{\mfg^{\delta}_h},\\
	\widehat{\phi^{\delta}_h} &= \av{\phi^{\delta}_h} + \mfC_{12}\cdot\jump{\phi^{\delta}_h},
\end{align}
\begin{align}
	\widehat{\dot{\mfg_h}} &= \av{\dot{\mfg_h}} - C_{11}\jump{\dot{\phi_h}} - \mfC_{12}\jump{\dot{\mfg_h}},\\
	\widehat{\dot{\phi_h}} &= \av{\dot{\phi_h}} + \mfC_{12}\cdot\jump{\dot{\phi_h}}.\label{fluxdot}
\end{align} \label{LDGflux}
\end{subequations}

For simplicity, let $\mD_1$ denote the discrete Poisson solver in \eqref{S3:eq:29} and \eqref{S3:eq:30}, and let $\mD_2$ denote the discrete Poisson solver in \eqref{S3:eq:31}. In other words, we have:
\begin{equation}
	\bdp{\mfg^e_h,\phi^e_h} = \mD_1\bdp{4\pi G \rho_h^e},\; \bdp{\mfg^{\delta}_h,\phi^{\delta}_h} = \mD_1\bdp{4\pi G \rho^{\delta}_h},\;
	\bdp{\dot{\mfg_h},\dot{\phi_h}} = \mD_2\bdp{4\pi G\nabla\cdot(\rho\mfu)_h}. 
	\label{Poissonsolver}
\end{equation}

We note that the equilibrium state $\{\rho^e(\mfx),\mfu^e(\mfx) = \mathbf{0},p^e(\mfx),\phi^e(\mfx)\}$ to be preserved is assumed to be stationary which is obtained from \eqref{eqstate}. Numerically the discrete equilibrium state variables $\rho_h^e$, $\mfu_h^e$ and $p_h^e$ are given by the projector in \eqref{projector}, i.e. $\rho_h^e = \mfP\rho^e$, $\mfu_h^e = \mfP\mfu^e = \mathbf{0}$, and $p_h^e = \mfP p^e$. The discrete gravitational potential at equilibrium $\phi_{h}^{e}$ is obtained from \eqref{Poissonsolver}, namely the equilibrium state is initialized at the beginning and remain unchanged for later time. Therefore, only the perturbation part $\{\rho_h^{\delta}, \phi_h^{\delta}\}$ evolves over time.

For structure preserving, the WB property is obtained for the DG method with a modified HLLC flux and a revised source term approximation, which will be shown in Section~\ref{3.3.1} and Section~\ref{3.3.3}. The TEC property is ensured as the conservative form for the total energy $E_{tot}$ is used. However, directly solving the Poisson equation \eqref{S3:eq:12} appeared in the energy flux with \eqref{S3:eq:31} suffers from a loss of accuracy due to the temporal derivative. To address this issue, we apply the summation-by-parts approach to the right-hand side of \eqref{S3:eq:31}, which will be shown in Section~\ref{3.3.2} and numerically demonstrated in  Section \ref{sec5}.  

\subsubsection{The modified HLLC numerical flux}
\label{3.3.1}\hfill \\

Based on the contact property of HLLC flux, we follow the idea of the modifications in \cite{wu2021uniformly}:
\begin{equation}
	\widehat{\mfF}_{\mfn_{\mathcal{E}, K}} = \mfF^{\text{hllc}} \left( 
	\frac{p_h^{e, \star}}{p_h^{e, \text{int}(K)}} \mfU_h^{\text{int}(K)}, 
	\frac{p_h^{e, \star}}{p_h^{e, \text{ext}(K)}} \mfU_h^{\text{ext}(K)}; 
	\mfn_{\mE, K} 
	\right) = \left(\widehat{f}^{[1]},\widehat{\bm{f}}^{[2],T},\widehat{f}^{[3]}\right)_{\mfn_{\mE, K}}^T,\label{S3:eq:13}
\end{equation}
where $p_h^{e,*}$ is chosen as
\begin{equation*}
	p_h^{e,*} = \dfrac{1}{2}\bdp{p_h^{e,\text{int}(K)} + p_h^{e,\text{ext}(K)}}.
\end{equation*}
Based on the Lemma \ref{contactdiscontinuity}, if  $\mfU_h^{{\rm int}(K)} = (\rho_h^{e,{\rm int}(K)},\mathbf{0},p_h^{e,{\rm int}(K)}/(\gamma-1))^T,\mfU_h^{{\rm ext}(K)} = (\rho_h^{e,{\rm ext}(K)},\mathbf{0},p_h^{e,{\rm ext}(K)}/(\gamma-1))^T$, then $\widehat{\mfF}_{\mfn_{\mathcal{E}, K}} = \bdp{0,p_h^{e,*}\mfn^T,0}^T$. With the $N$ points numerical quadrature on the $d-1$ dimensional face $\mE$, we obtain the discrete version of the flux term of \eqref{S3:eq:21}:
\begin{equation}
	\int_{\mE}\widehat{\mfF}_{\mfn_{\mE,K}}v^{{\rm int}(K)}\dd{s} \approx \sum_{\mu = 1}^{N}w_{\mE}^{(\mu)} \mfF\bdp{\mfU_h(\mfx_{\mE}^{(\mu)})}v^{{\rm int}(K)}(\mfx_{\mE}^{(\mu)}),\label{S3:eq:19}
\end{equation}
where $\{\mfx_{\mE}^{(\mu)},w_{\mE}^{(\mu)}\}_{\mu=1}^N$ denote the quadrature points and weights on $\mE$.

\subsubsection{Source term approximations} \label{3.3.3} \hfill \\

Let the source term $\mfS := \bdp{0,\mfS^{[2],T},0}^T$ with $\mfS^{[2]} = -\rho\nabla\phi^e -\rho\nabla\phi^{\delta}$. For the static gravitational potential, we need a balance between the flux gradient and source term. The treatment is similar to \cite{wu2021uniformly,DU2024WBPP}. The gravitational potential at the equilibrium state $\phi^e$ is considered as the static gravitational potential, while the perturbation part $\phi^{\delta}$ is directly addressed. 
\begin{equation}
\begin{split}
	\int_{K}\mathbf{S}^{[2]}v\dd{\mfx} &= \int_{K}-\rho\nabla\phi v\dd{\mfx} =\int_{K}-\rho\nabla(\phi^e + \phi^\delta) v\dd{\mfx} =  \int_{K}\rho \dfrac{\nabla p^e}{\rho^e}v\dd{\mfx} +  \int_{K}-\rho \nabla\phi^\delta v\dd{\mfx}\\
	& = \int_{K}\nabla p^e \bdp{\dfrac{\rho}{\rho^e} - \dfrac{\overline{\rho}}{\overline{\rho^e}}}v\dd{\mfx} +  \int_{K}\dfrac{\overline{\rho}}{\overline{\rho^e}}p^ev\dd{\mfx}+  \int_{K}-\rho \nabla\phi^\delta v \dd{\mfx}\\
	& = \int_{K}\nabla p^e \bdp{\dfrac{\rho}{\rho^e} - \dfrac{\overline{\rho}}{\overline{\rho^e}}}v\dd{\mfx} + \dfrac{\overline{\rho}}{\overline{\rho^e}}\bdp{\sum_{\mE\in\partial K}\int_{\mE}p^ev\mfn_{\mE,K}\dd{s} - \int_{K}p^e\nabla v\dd{\mfx}} + \int_{K}-\rho\nabla\phi^\delta v  \dd{\mfx}.
\end{split}\label{S3:eq:23}
\end{equation}
With the $Q$ points numerical quadrature in the $d$ dimensional element $K$ and the numerical quadrature in \eqref{S3:eq:19}, the fully discrete version of \eqref{S3:eq:23} is given by:
\begin{equation}
\begin{split}
	\left<\mfS^{[2]},v\right>_{K}&:= \sum_{q = 1}^{Q}w_{K}^{(q)} \left(\dfrac{\rho_h(\mfx_K^{(q)})}{\rho^e_h(\mfx_K^{(q)})} - \dfrac{\overline{(\rho_h)_K}}{\overline{(\rho^e_h)_K}}\right)\nabla p_h^e(\mfx_K^{(q)})v(\mfx_K^{(q)}) \\
	&\quad + \dfrac{\overline{(\rho_h)_K}}{\overline{(\rho^e_h)_K}}\left(\sum_{\mE\in\partial K}\sum_{\mu = 1}^{N}w_{\mE}^{(\mu)}  p^e_h(\mfx_{\mE}^{(\mu)})v^{\text{int}(K)}\mfn_{\mE,K}(\mfx_{\mE}^{(\mu)})  - \sum_{q = 1}^{Q}w_K^{(q)} p_h^e(\mfx_K^{(q)})\nabla v(\mfx_K^{(q)})\right) \\
	&\quad + \sum_{q = 1}^{Q}-w_K^{(q)}\rho_h(x_K^{(q)})\mfg^{\delta}_h(\mfx_K^{(q)})v(\mfx_K^{(q)}), \label{S3:eq:24}
\end{split}
\end{equation}
where $\{\mfx_{K}^{(q)},w_{K}^{(q)}\}_{q=1}^Q$ denote the quadrature points and weights in $K$.

\subsubsection{Energy flux $(\mfF_g)$ approximations} \label{3.3.2} \hfill \\

As mentioned in \eqref{S2:eq:19}, the energy flux has a very complex form. Notice that $\bdp{\mfg_h,\phi_h}$ is also split into equilibrium state and perturbation state, thus it is obtained by:
\begin{equation*}
	\bdp{\mfg_h,\phi_h} = \bdp{\mfg^e_h,\phi^e_h} + \bdp{\mfg^{\delta}_h,\phi^{\delta}_h} = \bdp{\mfg^e_h,\phi^e_h} + \mD_1(4\pi G(\rho_h - \rho_h^e)).
\end{equation*}
However, the Poisson equation $\mD_2$ for the time derivative of gravitational potential needs a summation-by-parts approach to avoid an order deficiency. Specifically, \eqref{S3:eq:31} will be modified as:
\begin{equation}
	\label{Poisson2}
	\begin{split}
		\int_K\dot{\mfg_h}\cdot\mfw\dd{\mfx} &= \sum_{\mE\in \partial K}\int_{\mE}\widehat{\dot{\phi_h}}\mfw^{{\rm int}(K)}\cdot\mfn_{\mE,K}\dd{s} - \int_{K}\dot{\phi_h}\nabla\cdot\mfw\dd{\mfx},\\
		\int_K\dot{\mfg_h}\cdot\nabla\psi\dd{\mfx} &= \int_{\mE}\psi^{{\rm int}(K)}\widehat{\dot{\mfg_h}}\cdot\mfn_{\mE,K}\dd{s} + \int_K 4\pi G (\rho\mfu)_h \cdot\nabla\psi \dd{\mfx} - \sum_{\mE\in \partial K}\int_{\mE}4\pi G\widehat{f}^{[1]}_{\mfn_{\mE,K}}\psi^{{\rm int}(K)}\dd{s}.
	\end{split}
\end{equation}
Although the notation $\mD_2$ was previously defined, in the rest of this paper it will denote the modified Poisson solver \eqref{Poisson2}, i.e.,
\begin{equation*}
	\bdp{\dot{\mfg_h},\dot{\phi_h}} = \mD_2\bdp{4\pi G\nabla\cdot(\rho\mfu)_h}.
\end{equation*}
Therefore, the numerical energy flux is given as follows:
\begin{equation*}
	(\widehat{\mfF_g})_{\mfn_{\mathcal{E},K}} = \dfrac{1}{8\pi G}\left(\widehat{\phi_h}\widehat{\dot{\mfg_h}} - \widehat{\dot{\phi_h}}\widehat{\mfg_h}\right)\cdot \mfn_{\mathcal{E},K}+ \widehat{f}^{[1]}_{\mfn_{\mE,K}} \widehat{\phi_h},\label{S3:eq:20}
\end{equation*}
where $(\widehat{\mfg_h},\widehat{\phi_h}) = (\widehat{\mfg^{\delta}_h},\widehat{\phi^{\delta}_h}) + (\widehat{\mfg^{e}_h},\widehat{\phi^{e}_h})$ and $\bdp{\widehat{\dot{\phi_h}},\widehat{\dot{\mfg_h}}}$ are defined in \eqref{LDGflux}. The term $\widehat{f}^{[1]}_{\mfn_{\mE,K}}$ denotes the first component of the modified HLLC flux introduced in \eqref{S3:eq:13}. The modified HLLC flux is evaluated first, followed by the energy flux. Similar to \eqref{S3:eq:19}, we omit the numerical quadrature on $\mE$ of $	(\widehat{\mfF_g})_{\mfn_{\mathcal{E},K}}$ here. In the internal of element $K$, the energy flux is given by:
\begin{equation*}
	\mfF_g = \dfrac{1}{8\pi G}\bdp{\phi_h\dot{\mfg_h} -\dot{\phi_h}\mfg_h} + (\rho\mfu)_h\phi_h.
\end{equation*}
With the $Q$ points numerical quadrature in the $d$ dimensional element $K$ similar in \eqref{S3:eq:24}, we obtain the fully discrete version of $\mfG$. The numerical quadrature of the third term of $\mfG$ in $K$ is given by:
\begin{equation}
	\int_K\mfF_g\cdot\nabla v\dd{\mfx} \approx\sum_{q=1}^{Q}w_K^{(q)}\mfF_g(\mfx_K^{(q)})\cdot \nabla v(\mfx_K^{(q)}), \label{S3:eq:22}
\end{equation}
where $\{\mfx_{K}^{(q)},w_{K}^{(q)}\}_{q=1}^Q$ denote the quadrature points and weights in $K$.

\subsubsection{Flowchart for the structure-preserving LDG scheme} \label{3.3.4} \hfill \\

Suppose that the initial conditions for \eqref{S3:eq:18} are $\rho^0(\mfx)$, $(\rho\mfu)^0(\mfx)$, $E^0(\mfx)$, and $\phi^0(\mfx)$. These variables are initialized by the projector \eqref{projector}:
\begin{equation}
\rho_h^0(\mfx) = \mfP\rho^0(\mfx),\quad (\rho\mfu)_h^0(\mfx) = \mfP(\rho\mfu)^0(\mfx),\quad E_h^0(\mfx) = \mfP E^0(\mfx). \label{S3:eq:32}
\end{equation}
The initial gravitational potential $\phi_h^0$ and the total energy $(E_{tot})_h^0$ are obtained by applying the Poisson solver $\mD_1$, i.e.,
\begin{equation}
(\mfg_h^0,\phi_h^0) = \mD_1(4\pi G\rho_h^0),\quad
(E_{tot})^0_h = \mfP\bdp{E_h^0 + \tfrac{1}{2}\rho_h^0\phi_h^0} := E_h^0 + \mfP\bdp{\tfrac{1}{2}\rho_h^0\phi_h^0}. \label{S3:eq:33}
\end{equation}

Based on the above approximations and modifications for the flux term and source term, the structure-preserving  LDG scheme with Euler-forward method is implemented in the following steps:
\begin{itemize}
\label{WB_EC}
\item \label{Step1}\textbf{Step 1.} Suppose that the variables $(\rho_h^n,(\rho\mfu)_h^n,(E_{tot})^n_h)$ at time $t^n$ remains known. The gravitational force and gravitational potential $(\mfg_h^n,\phi_h^n)$ is given by:
\begin{equation}
	\bdp{\mfg_h^n,\phi^n_h} = \bdp{\mfg_h^e,\phi_h^e} + \bdp{(\mfg_h^{\delta})^n,(\phi_h^{\delta})^n},\label{S3:eq:25}
\end{equation}
where $\bdp{(\mfg_h^{\delta})^n,(\phi_h^{\delta})^n}$ satisfy
\begin{equation}
	\bdp{(\mfg_h^{\delta})^n,(\phi_h^{\delta})^n} = \mD_1\bdp{4\pi G(\rho_h^{\delta})^n} = \mD_1\bdp{4\pi G\bdp{\rho_h^{n} - \rho_h^e}}.
\end{equation}
The time derivative term $\bdp{\dot{\mfg_h^n},\dot{\phi^n_h}}$ is given by the modified Poisson solver $\mD_2$ \eqref{Poisson2}:
\begin{equation}
	\bdp{\dot{\mfg_h^n},\dot{\phi_h^n}} = \mD_2\bdp{4\pi G\nabla\cdot(\rho\mfu)^n_h}. \label{S3:eq:26}
\end{equation}
The total non-gravitational energy $E_h^n$ is given by:
\begin{equation}
	E_h^n = (E_{tot})_h^n - \mfP\bdp{\dfrac{1}{2}\rho_h^n\phi_h^n}.\label{S3:eq:27}
\end{equation}
\item \label{Step2}\textbf{Step 2.} Compute the modified HLLC flux. The modified HLLC flux is given by \eqref{S3:eq:13}. We then get the HLLC flux at time $t^n$.
\begin{equation}
	\widehat{\mfF}^n_{\mfn_{\mathcal{E}, K}} = \mfF^{\text{hllc}} \left( 
	\frac{p_h^{e, \star}}{p_h^{e, \text{int}(K)}} \mfU_h^{n,\text{int}(K)}, 
	\frac{p_h^{e, \star}}{p_h^{e, \text{ext}(K)}} \mfU_h^{n,\text{ext}(K)}; 
	\mfn_{\mE, K} 
	\right) = \left(\widehat{f}^{n,[1]},\widehat{\bm{f}}^{n,[2],T},\widehat{f}^{n,[3]}\right)_{\mfn_{\mE, K}}^T.\label{S3:eq:28}
\end{equation}
\item  \textbf{Step 3.}\label{Step3} Compute the energy flux in the internal of an element $K$ and on the face $\mE$ based on \eqref{S3:eq:25}, \eqref{S3:eq:26} and \eqref{S3:eq:28}. In the internal of element, the energy flux is given by:
\begin{equation}
	\mfF^n_g = \dfrac{1}{8\pi G}\bdp{\phi^n_h\dot{\mfg^n_h} -\dot{\phi^n_h}\mfg^n_h} + (\rho\mfu)^n_h\phi^n_h.
\end{equation}
On the face of element, the energy flux is given by 
\begin{equation}
	(\widehat{\mfF_g})^n_{\mfn_{\mathcal{E},K}} = \dfrac{1}{8\pi G}\left(\widehat{\phi^n_h}\widehat{\dot{\mfg^n_h}} - \widehat{\dot{\phi^n_h}}\widehat{\mfg^n_h}\right)\cdot \mfn_{\mathcal{E},K}+ f^{n,[1]}_{\mfn_{\mE,K}} \widehat{\phi^n_h}.
\end{equation}
\item \label{Step4}\textbf{Step 4.} Evolve to obtain $(\rho_h^{n+1},(\rho\mfu)_h^{n+1},(E_{tot})^{n+1}_h)$. From the approximations of source term \eqref{S3:eq:24} and the semi-discrete form \eqref{S3:eq:21}, we have:
\begin{subequations}
\label{FullyDiscreteStructure}
\begin{equation}
	\int_K\rho_h^{n+1}v\dd{\mfx} = \int_K\rho_h^{n}v\dd{\mfx} + \Delta t \bdp{\int_K\mfF^{n,[1]}_h\cdot\nabla v\dd{\mfx} - \sum_{\mE\in \partial K}\int_{\mE}\widehat{f}^{n,[1]}_{\mfn_{\mE,K}}v^{\text{int}(K)}\dd{s}},
\end{equation}
\begin{equation}
	\int_K(\rho\mfu)_h^{n+1}v\dd{\mfx} = \int_K(\rho\mfu)_h^{n}v\dd{\mfx} + \Delta t \bdp{\int_K\mfF^{n,[2]}_h\cdot\nabla v\dd{\mfx} - \sum_{\mE\in \partial K}\int_{\mE}\widehat{\bm{f}}^{n,[2]}_{\mfn_{\mE,K}}v^{\text{int}(K)}\dd{s} + \left<\mfS^{[2],n},v\right>_{K}},
\end{equation}
\begin{equation}
\begin{split}
	\int_K(E_{tot})_h^{n+1}v\dd{\mfx} = \int_K(E_{tot})_h^{n}v\dd{\mfx} &+ \Delta t \bdp{\int_K\mfF^{n,[3]}_h\cdot\nabla v\dd{\mfx} - \sum_{\mE\in \partial K}\int_{\mE}\widehat{f}^{n,[3]}_{\mfn_{\mE,K}}v^{\text{int}(K)}\dd{s}} \\
	& + \Delta t\bdp{\int_K(\mfF_g)^{n}_h\cdot\nabla v\dd{\mfx} - \sum_{\mE\in \partial K}\int_{\mE}(\widehat{\mfF_g})^n_{\mfn_{\mE,K}}v^{\text{int}(K)}\dd{s}}.
\end{split} \label{Etotupdate}
\end{equation}
\end{subequations}
\end{itemize}

% \textcolor{red}{The details of the structure-preserving LDG scheme with high-order time discretizations are omitted here, as the extension to high-order Runge–Kutta methods do not compromise the WB and TEC properties.}

% \textcolor{red}{The fully discrete structure-preserving LDG scheme holds the WB property. In other words, if the initial conditions are given by the steady states $\rho^0(\mfx)  = \rho^e(\mfx)$, $\mfu^0(\mfx) = \mfu^e(\mfx) = \mathbf{0}$, $p^0(\mfx) = p^e(\mfx)$, and $\phi^0(\mfx) = \phi^e(\mfx)$, then we have $\rho_h^n(\mfx)  = \rho_h^e(\mfx)$, $\mfu_h^e(\mfx) = \mathbf{0}$, $p_h^n(\mfx) = p_h^e(\mfx)$, and $\phi_h^n(\mfx) = \phi_h^e(\mfx)$ for all $n$. Moreover, if either compactly supported boundary conditions or periodic boundary conditions are applied, the total energy is conserved, i.e. $(E_{tot})_h^n = (E_{tot})_h^0$.}

\begin{rem}
	The projection procedure in \eqref{S3:eq:33} and \eqref{S3:eq:27} is essential for preserving the WB property. Since the initial total energy $(E_{tot})_h^0$ is projected onto the DG space $V_h^k$, the same projection procedure is required during the subsequent evolution to ensure consistency. This guarantees the pressure $p_h^n$ remains invariant. As shown in Section \ref{sec4}, omitting this procedure leads to the loss of the WB property.
\end{rem}
\subsection{High order time discretizations}\label{3.4} \hfill \\

When first order Euler-forward method is considered, both the standard LDG scheme \eqref{FullyDiscreteStandard} and the structure-preserving LDG scheme \eqref{FullyDiscreteStructure} can be written in a compact form:
\begin{equation}
	\mfU_h^{n+1} = \mfU_h^{n} + \Delta t\mathcal{L}(\mfU_h^{n},t^n);\qquad \mfW_h^{n+1} = \mfW_h^{n} + \Delta t\mathcal{H}(\mfW_h^{n},t^n),
\end{equation}
where $\mathcal{L}(\cdot,\cdot)$ and $\mathcal{H}(\cdot,\cdot)$ represent the spatial discretization operators associated with the respective schemes, and $\Delta t$ is the time step size. When combined with high order strong-stability-preserving  Runge-Kutta (SSP-RK) time discretizations \cite{Shu1988TVD_RK}, both schemes are formulated as follows:

For the second order SSP-RK method:
\begin{equation*} 
\begin{split} 
	\mfU_h^{(1)} &= \mfU_h^{n} + \Delta t\mathcal{L}(\mfU_h^{n},t^n),\\
	\mfU_h^{n+1} &= \dfrac{1}{2}\mfU_h^{n} + \dfrac{1}{2}\bdp{\mfU^{(1)} + \Delta t\mathcal{L}(\mfU_h^{(1)},t^n + \Delta t)},
\end{split} \qquad 
	\begin{split} \mfW_h^{(1)} &= \mfW_h^{n} + \Delta t\mathcal{H}(\mfW_h^{n},t^n),\\ 
	\mfW_h^{n+1} &= \dfrac{1}{2}\mfW_h^{n} + \dfrac{1}{2}\bdp{\mfW^{(1)} + \Delta t\mathcal{H}(\mfW_h^{(1)},t^n + \Delta t)}.
\end{split} 
\end{equation*}

For the third order SSP-RK method:
\begin{equation*} 
\begin{split} 
	\mfU_h^{(1)} &= \mfU_h^{n} + \Delta t\mathcal{L}(\mfU_h^{n},t^n),\\
	\mfU_h^{(2)} &= \dfrac{3}{4}\mfU_h^{n} + \dfrac{1}{4}\bigg(\mfU^{(1)} + \Delta t\mathcal{L}(\mfU_h^{(1)},t^n + \Delta t)\bigg),\\ \mfU_h^{n+1} &= \dfrac{1}{3}\mfU_h^{n} + \dfrac{2}{3}\bdp{\mfU^{(2)} + \Delta t\mathcal{L}(\mfU_h^{(2)},t^n + \dfrac{1}{2}\Delta t)}, 
\end{split} \quad 
\begin{split} 
	\mfW_h^{(1)} &= \mfW_h^{n} + \Delta t\mathcal{H}(\mfW_h^{n},t^n),\\
	\mfW_h^{(2)} &= \dfrac{3}{4}\mfW_h^{n} + \dfrac{1}{4}\bdp{\mfW^{(1)} + \Delta t\mathcal{H}(\mfW_h^{(1)},t^n + \Delta t)},\\ 
	\mfW_h^{n+1} &= \dfrac{1}{3}\mfW_h^{n} + \dfrac{2}{3}\bdp{\mfW^{(2)} + \Delta t\mathcal{H}(\mfW_h^{(2)},t^n + \dfrac{1}{2}\Delta t)}.
\end{split} 
\end{equation*}

\subsection{Oscillation-eliminating approach} \label{3.4.1} \hfill \\

We now address the issue of numerical oscillations that may arise when high order methods are applied to non-smooth problems. To control the numerical oscillations, we adopt the OE technique in \cite{2024PengOEDG}. Here, we give a brief introduction to the OE technique, for more details of this method we refer the readers to \cite{2024PengOEDG}.

The semi-discrete OEDG method can be expressed as an ordinary equation system with operator splitting:
\begin{equation*}
	\dfrac{\dd{\mfU}}{\dd{t}} + \mfL_f(\mfU) + \mathbf{\Sigma}(\mfU)\mfU = 0\Longrightarrow\dfrac{\dd{\mfU}}{\dd{t}} = -\mfL_f(\mfU)
	\text{ and }\dfrac{\dd{\mfU}_{\sigma}}{\dd{t}} = -\mathbf{\Sigma}(\mfU)\mfU_{\sigma}.
\end{equation*}
Here $\mathbf{\Sigma}$ is the damping term in the OE method and has scale-invariant and evolution-invariant properties. It is worth noting that the damping ODE is quai-linear while $\mathbf{\Sigma}(\mfU)$ is fixed. The general high order RK or multi-step discretization can both combined with OE method to avoid certain numerical oscillation. We take the $s-$stage explicit RK method combined with the OE technique as an example:
\begin{equation*}
	\begin{split}
		\mathbf{U}^{n,0}_h &= \mathbf{U}^{n}_{\sigma} = \mathbf{U}^{n}_h,\\
		\mathbf{U}^{n,l+1}_h &= \sum_{0\leq m\leq l}\bdp{a_{lm}\mathbf{U}^{n,m}_{\sigma} - \Delta t b_{lm}\mfL_f\bdp{\mfU^{n,m}_{\sigma}}},\quad l = 0,1,\cdots,s-1,\\
		\mathbf{U}^{n+1}_h &= \mathbf{U}^{n,s}_{\sigma},
	\end{split}
\end{equation*} 
where $\mfU^{n,m}_{\sigma}$ denote the function after the OE damping technique applied on $\mfU^{n,m}_{h}$, i.e.
\begin{equation*}
	\mathbf{U}^{n,l+1}_{\sigma} = \mathcal{F}_{\Delta t}\mathbf{U}^{n,l+1}_{h},\quad  l = 0,1,\cdots,s-1.
\end{equation*}
Here, $\Delta t$ is the time-step size, and $\mathcal{F}_{\Delta t}$ denotes the OE technique for notational convenience. Actually, we define $\mathcal{F}_{\Delta t}\bdp{\mfU_{h}}(\mfx) = \mathbf{U}_{\sigma}(\mfx,\Delta t)$ with $\mathbf{U}_{\sigma}(\mfx,\hat{t})$ being the solution of the following initial value problem:
\begin{equation}
	\left\{
	\begin{aligned}
		&\dfrac{\dd}{\dd{\hat{t}}}\int_{K}\mfU_{\sigma}\cdot \mfv\dd{\mfx} + \sum_{m=0}^k \vartheta_{K}^m\bdp{\mfU_{h}}\int_K\bdp{\mfU_{\sigma} - \mfP^{m-1}\mfU_{\sigma}}\cdot \mfv\dd{\mfx} = 0,\quad \forall \mfv\in[\mathbb{P}^k(K)]^N,\\
		&\mfU_{\sigma}(\mfx,0) = \mfU_{h}(\mfx).
	\end{aligned}
	\right. \label{S3:eq:34}
\end{equation}
The coefficient $\vartheta_K^m(\mfU_h)$ is defined as:
\begin{equation*}
	\vartheta_K^m(\mfU_h) = \sum_{e\in\partial K}\beta_e\dfrac{\sigma^m_{e,K}(\mfU_h)}{h_{e,K}}.
\end{equation*}
$\beta_e$ is a suitable estimate of the local maximum wave speed in the direction $\mfn_e$ on the interface $e$ of the element $K$, which is set as the spectral radius of $\sum_{i=1}^{d}n_e^{(i)}\pdrv{\mfF_i}{\mfu}\Big|_{\mfu = \mfu_K}$, where $n_e^{(i)}$ is the $i$th component of $\mfn_e$. $h_{e,K}$, $\sigma^m_{e,K}(\mfU_h)$ is defined as:
\begin{equation*}
	h_{e,K} = \sup_{\mfx\in K}\text{dist}(\mfx,e),\qquad 
	\sigma^m_{e,K}(\mfU_h) = \max_{1\leq i\leq N}\sigma^m_{e,K}(u^{(i)}_h),
\end{equation*}
while $u^{(i)}_h$ denoting the $i$th component of $\mfu_h$, and
\begin{equation*}
	\sigma^m_{e,K}(u^{(i)}_h) = \begin{cases}
		0, & \text{if } u^{(i)}_h\equiv \overline{u}^{(i)}_h,\\
	\dfrac{(2m+1)h_{e,K}^m}{2(2k-1)m!}\sum\limits_{|\alpha|=m}\dfrac{\frac{1}{|e|}\int_e\left|\jump{\partial^{\alpha}u^{(i)}_h}_e\right|\dd{S}}{||u^{(i)}_h- \text{avg}(u^{(i)}_h)||_{L_\infty}(\Omega)}, & \text{otherwise}.
	\end{cases}
\end{equation*}
The multi-index $\alpha = (\alpha_1,\cdots,\alpha_d)$ with $|\alpha| = \sum_i\alpha_i$, and the derivative $\partial^{\alpha}u^{(i)}_h$ is defined by:
\begin{equation*}
	\partial^{\alpha}u^{(i)}_h = \dfrac{\partial^{|\alpha|}}{\partial x_1^{\alpha_1}\cdots x_d^{\alpha_d}}u^{(i)}_h.
\end{equation*}
The damping equation \eqref{S3:eq:34} can be explicitly implemented as:
\begin{equation*}
	\mathcal{F}_{\Delta t}\mfU^{n,l+1}_{h} = \mfU^{(0)}_{K}\phi_K^{(0)}(\mfx) + \sum_{j=1}^{k}e^{-\Delta t\sum_{m=0}^{j}\vartheta_K^m(\mfU^{n,l+1}_{h})}\sum_{|\alpha| = j}\mfU^{(\alpha)}_{K}\phi^{(\alpha)}_K(\mfx),
\end{equation*}
where $\mfU^{(\alpha)}_{K}$ denote the moment coefficient based on the orthogonal basis $\bdc{\phi_K^{\alpha}}_{|\alpha|\leq k}\subset \mathbb{P}^k(K)$. It is easy to find that $\mfU^{(0)}_{K} = \text{avg}(\mfU^{n}_{K})$, thus the OE procedure will not influence the average value of the element $K$. To design OE method for WB scheme, take third order SSP-RK as example, the fully discrete scheme is outlined as follows:
\begin{equation*}
	\begin{aligned} 
		\tilde{\mfW}_h^{(1)} &= \mfW_h^{n} + \Delta t\mathcal{H}(\mfW_h^{n},t^n),
		&
		\mfW_h^{(1)} = \mfW_h^e + \mathcal{F}_{\Delta t}(\tilde{\mfW}_h^{(1)} - \mfW_h^e),\\
		\tilde{\mfW}_h^{(2)} &= \dfrac{3}{4}\mfW_h^{n} + \dfrac{1}{4}\bdp{\mfW^{(1)} + \Delta t\mathcal{H}(\mfW_h^{(1)},t^n + \Delta t)},
		&
		\mfW_h^{(2)} = \mfW_h^e + \mathcal{F}_{\Delta t}(\tilde{\mfW}_h^{(2)} - \mfW_h^e),\\ 
		\tilde{\mfW}_h^{n+1} &= \dfrac{1}{3}\mfW_h^{n} + \dfrac{2}{3}\bdp{\mfW^{(2)} + \Delta t\mathcal{H}(\mfW_h^{(2)},t^n + \dfrac{1}{2}\Delta t)},
		&
		\mfW_h^{n+1} = \mfW_h^e + \mathcal{F}_{\Delta t}(\tilde{\mfW}_h^{n+1} - \mfW_h^e). \\
	\end{aligned} 
\end{equation*}
Specifically, suppose that $\tilde{\mfW}_h\big|_K \in V_h^k$ and
\begin{equation*}
	\tilde{\mfW}_h\big|_K = \mfW_h^e + \mfW_h^{\delta} =\sum_{j=0}^{k}\sum_{|\alpha| = j}\mfW^{e,(\alpha)}_{K}\phi^{(\alpha)}_K(\mfx) + \sum_{j=0}^{k}\sum_{|\alpha| = j}\mfW^{\delta,(\alpha)}_{K}\phi^{(\alpha)}_K(\mfx).
\end{equation*}
Then we apply the OE technique on the perturbation part 
\begin{equation}
	\mfW_h = \mfW_h^e + \mathcal{F}_{\Delta t}(\mfW_h^{\delta}) = \sum_{j=0}^{k}\sum_{|\alpha| = j}\mfW^{e,(\alpha)}_{K}\phi^{(\alpha)}_K(\mfx) + \mfW^{\delta,(0)}_{K}\phi_K^{(0)}(\mfx) + \sum_{j=1}^{k}e^{-\Delta t\sum_{m=0}^{j}\vartheta_K^m(\tilde{\mfW}_h)}\sum_{|\alpha| = j}\mfW^{\delta,(\alpha)}_{K}\phi^{(\alpha)}_K(\mfx).\label{damping}
\end{equation}
It holds that if the damping operator $\mathcal{F}_{\Delta t}$ remains high order accuracy, then
\begin{equation*}
	||\mfW_h - \tilde{\mfW}_h|| = ||\mfW_h^{\delta} - \mathcal{F}_{\Delta t}(\mfW_h^{\delta})|| = O(\Delta x^{k+1}).
\end{equation*}
Thus the modified damping procedure \eqref{damping} will not destroy the accuracy of our structure-preserving LDG scheme.

\section{Structure-preserving properties}
\label{sec4}
\setcounter{equation}{0}
\setcounter{figure}{0}
\setcounter{table}{0}
Based on the flowchart of the structure-preserving LDG scheme, we now prove the WB and TEC properties of our proposed structure-preserving LDG scheme.

%\subsection{Well-balanced property} \label{4.1} \hfill \\

\begin{thm}\label{thm1}
	The numerical method described in Section~\ref{3.3} together with the flowchart in Section~\ref{3.3.4}, preserves the explicitly known equilibrium state $ \mfu^0=\mathbf{0}$, $p^0=-\rho^0\nabla \phi^0$, $\Delta  \phi^0 = 4\pi G\rho^0$. Numerically, if the initial conditions are given by the steady states $\rho^0  = \rho^e$, $\mfu^0 = \mfu^e = \mathbf{0}$, $p^0 = p^e$, and $\phi^0 = \phi^e$, then we have $\rho_h^n  = \rho_h^e$, $\mfu_h^e = \mathbf{0}$, $p_h^n = p_h^e$, and $\phi_h^n = \phi_h^e(\mfx)$ for all $n$, thereby confirming the well-balanced property of the scheme.
\end{thm}
\begin{proof}
	The proof is provided in Appendix~\ref{pro:WB}.
\end{proof}

%\subsection{Energy-conserving property}\label{4.2} \hfill \\

\begin{thm}
	The numerical method described in Section \ref{3.3} together with the flowchart in Section \ref{3.3.4} conserves the total energy:
	\begin{equation}
		\int_{\Omega}(E_{tot})_h^{n+1}\dd{\mfx} = \int_{\Omega}(E_{tot})_h^{n}\dd{\mfx},
	\end{equation}
	under either compactly supported boundary conditions or periodic boundary conditions.
\end{thm}
\begin{proof}
	The proof follows directly from \eqref{Etotupdate}. Let the test function $v = 1$ and summation from the all elements, we have that
	\begin{equation*}
		\begin{split}
			\sum_{K}\int_K(E_{tot})_h^{n+1}\dd{\mfx} &= \sum_{K}\int_K(E_{tot})_h^{n}\dd{\mfx} + \Delta t \bdp{\sum_{K}\sum_{\mE\in \partial K}\int_{\mE}\widehat{f}^{n,[3]}_{\mfn_{\mE,K}}\dd{s} + \sum_{K}\sum_{\mE\in \partial K}\int_{\mE}(\widehat{\mfF_g})^n_{\mfn_{\mE,K}}\dd{s}} \\
			&= \sum_{K}\int_K(E_{tot})_h^{n}\dd{\mfx} + \Delta t \bdp{\sum_{\mE\in \partial K\cap \mE_B}\int_{\mE}\widehat{f}^{n,[3]}_{\mfn_{\mE,K}}\dd{s} + \sum_{\mE\in \partial K\cap \mE_B}\int_{\mE}(\widehat{\mfF_g})^n_{\mfn_{\mE,K}}\dd{s}}
		\end{split}
	\end{equation*}
	For the periodic boundary conditions, we have that
	\begin{equation*}
		\sum_{\mE\in \partial K\cap \mE_B}\int_{\mE}\widehat{f}^{n,[3]}_{\mfn_{\mE,K}}\dd{s} = \sum_{\mE\in \partial K\cap \mE_B}\int_{\mE}(\widehat{\mfF_g})^n_{\mfn_{\mE,K}}\dd{s} = 0.
	\end{equation*}
	And for the compactly supported boundary conditions, we have that 
	\begin{equation*}
		\mfu\big|_{\mE_B} = 0,
	\end{equation*}
	\begin{equation*}
		\widehat{f}^{n,[3]}_{\mfn_{\mE,K}} = (\widehat{\mfF_g})^n_{\mfn_{\mE,K}} = 0, \quad\text{ on }\mE_B.
	\end{equation*}
	We can finally obtain the TEC property.
\end{proof}

\vspace{0.1cm}

%%%%%%%%%%%%%%%%%%%%%%%%%%%%%
%%%%%%%%%%%%%%%%%%%%%%%%%%%%%
%%%%%%%%%%%%%%%%%%%%%%%%%%%%%
%%%%%%%%%%%%%%%%%%%%%%%%%%%%%
%%%%%%%%%%%%%%%%%%%%%%%%%%%%%
%%%%%%%%%%%%%%%%%%%%%%%%%%%%%

\section{Numerical tests}
\label{sec5}
\setcounter{equation}{0}
\setcounter{figure}{0}
\setcounter{table}{0}

In this section, we will present our numerical schemes through a series of numerical examples. For comparison purposes, we also include the numerical results obtained by the standard LDG scheme. All simulations are implemented in C++ with double-precision. Unless otherwise specified, we use the PardisoLLT solver from Intel MKL, accessed through the Eigen interface, to solve the large sparse matrices arising from the discretized Poisson equations. The time step is
\begin{equation*}
	\Delta t = {\rm CFL}\dfrac{ h}{\max\limits_{K\in\mT_h} \left(||\overline{\mfu}_K||_{\infty} + \overline{c}_K\right)},
\end{equation*}
where $\overline{c}_K = \sqrt{\gamma\overline{p}_K/\overline{\rho}_K}$, $h = \max_{K\in\mT_h}h_K$ and $||\cdot||_{\infty}$ denotes the maximum component of average velocity. With $f\in L^r(\Omega)$ and its numerical approximation $f_h\in V_h^k\subset L^r(\Omega)$,
the numerical $L^r$ errors and convergence orders are defined as 
\begin{equation*}
    L^r\text{ error } = ||f-f_h||_{L^r(\Omega)},\qquad L^r\text{ order } = \log_2\bdp{\dfrac{||f-f_h||_{L^r(\Omega)}}{||f-f_{h/2}||_{L^r(\Omega)}}}.
\end{equation*}

\subsection{Two dimensional case}
In this subsection, we present certain numerical results for the 2D SG Euler equations. We examine certain key properties of our proposed schemes, including accuracy test, WB property test, shock capturing ability and TEC property test.
\begin{exa} {\em
\label{exam1}
({\bf{Accuracy test}}) To assess the convergence order of our structure-preserving LDG scheme, we construct the following smooth problem:
\begin{subequations}
	\begin{equation*}
		\rho(x,y,t) = \sin\bdp{\dfrac{\sqrt{2}}{2a}\bdp{x+y-2t}},\, u(x,y,t) = 1,\, v(x,y,t) = 1,
	\end{equation*}
	\begin{equation*}
		p(x,y,t) = \kappa\rho^2,\; \phi(x,y,t) = - 4\pi G a^2\rho,
	\end{equation*}
\end{subequations}
where
\begin{equation*}
	a = \sqrt{\dfrac{\kappa(n+1)\lambda^{\frac{1-n}{n}}}{4\pi G}}.
\end{equation*}
Let $\kappa = 2\pi$, $G = 1/4$, $\lambda = 1$, $n = 1$, and $\gamma = 2$. The computational domain $\Omega$ is set as 
\begin{equation*}
	\Omega = \left[\dfrac{1}{8}\sqrt{2}\pi a,\dfrac{3}{8}\sqrt{2}\pi a\right]^2.
\end{equation*}
\begin{enumerate}
	\item Without the OE technique. In Tables \ref{SPLDG:acctest:nonOE}, \ref{SPLDG:acctest:nonSBP}, and \ref{SPLDG:acctest:SBP}, we do not apply the OE technique and observe the expected convergence orders for the structure-preserving LDG scheme. A comparison of different Poisson solver for computing the time derivative of the gravitational potential \eqref{S3:eq:12} is presented in Tables \ref{SPLDG:acctest:nonSBP} and \ref{SPLDG:acctest:SBP}, where we note a degradation in accuracy when summation-by-parts is not employed.
	
\begin{table}[htbp]
	\centering
	\caption{ Example \ref{exam1}. Numerical errors and convergence orders of density $\rho$ in the structure-preserving LDG scheme without OE technique. T = 0.8. }
	\label{SPLDG:acctest:nonOE}
	\begin{tabular}{c|c|c c|c c|c c}
		\hline
		Element & Mesh & $L^1$ error & order & $L^2$ error & order & $L^\infty$ error & order \\ 
		\hline
		\multirow{4}{*}{$P^1$} & $5 \times 5$    & 1.11e-02 & --    & 6.14e-03 & --    & 1.40e-02 & --    \\
		%\cline{2-8}
		& $10 \times 10$  & 2.46e-03 & 2.18  & 1.36e-03 & 2.18  & 3.26e-03 & 2.11  \\
		%\cline{2-8}
		& $20 \times 20$  & 5.82e-04 & 2.08  & 3.30e-04 & 2.04  & 1.02e-03 & 1.68  \\
		%\cline{2-8}
		& $40 \times 40$  & 1.43e-04 & 2.03  & 8.23e-05 & 2.00  & 2.90e-04 & 1.81  \\
		\hline 
		\multirow{4}{*}{$P^2$} & $5 \times 5$    & 2.16e-04 & --    & 1.46e-04 & --    & 6.17e-04 & --    \\
		%\cline{2-8}
		& $10 \times 10$  & 2.39e-05 & 3.18 &	 1.60e-05 & 3.20	& 7.39e-05 & 3.06  \\
		%\cline{2-8}
		& $20 \times 20$  & 2.81e-06 & 3.09	&
		1.87e-06 & 3.10	& 9.41e-06 & 2.97  \\
		%\cline{2-8}
		& $40 \times 40$  & 3.41e-07 & 3.05	& 
		2.26e-07 & 3.04	& 1.13e-06 & 3.06 \\
		\hline 
		\multirow{4}{*}{$P^3$} & $5 \times 5$    & 1.26e-05 & --    & 7.63e-06 & --    & 2.56e-05 & --    \\
		%\cline{2-8}
		& $10 \times 10$  & 6.11e-07 & 4.37  & 3.74e-07 & 4.35  & 1.43e-06 & 4.17  \\
		%\cline{2-8}
		& $20 \times 20$  & 3.38e-08 & 4.18  & 2.06e-08 & 4.18  & 8.59e-08 & 4.06  \\
		%\cline{2-8}
		& $40 \times 40$  & 2.04e-09 & 4.05  & 1.26e-09 & 4.04  & 5.43e-09 & 3.98  \\
		\hline
	\end{tabular}
\end{table}
\begin{table}[htbp] 
	\caption{Example \ref{exam1}. Numerical errors and convergence orders of density $\rho$ in the structure-preserving LDG scheme but without summation-by-parts in the Poisson equation for the time derivative of gravitational potential. T = 0.8. $Q^2$ element.}
	\label{SPLDG:acctest:nonSBP}
	\begin{tabular}{c|c c|c c|c c}
		\hline
		Mesh & $L^1$ error & order & $L^2$ error & order & $L^\infty$ error & order \\ 
		\hline
		$5 \times 5$  & 5.26e-04 & --    & 3.01e-04 & --    & 6.68e-04 & --    \\
%		\hline
		$10 \times 10$  & 7.50e-05 & 2.81    & 4.26e-05 & 2.82    & 8.62e-05 & 2.95    \\
%		\hline
		$20 \times 20$  & 1.25e-05 & 2.58  & 6.82e-06 & 2.64  & 1.16e-05 & 2.89  \\
%		\hline
		$40 \times 40$  & 2.46e-06 & 2.35  & 1.31e-06 & 2.38  & 1.71e-06 & 2.76  \\
		\hline
	\end{tabular}
\end{table}
\begin{table}[htbp] 
	\caption{Example \ref{exam1}. Numerical errors and convergence orders of density $\rho$ in the structure-preserving LDG scheme with summation-by-parts in the Poisson equation for the time derivative of gravitational potential. T = 0.8. $Q^2$ element.}
	\label{SPLDG:acctest:SBP}
	\begin{tabular}{c|c c|c c|c c}
		\hline
		Mesh & $L^1$ error & order & $L^2$ error & order & $L^\infty$ error & order \\ 
		\hline
		$5 \times 5$  & 1.16e-04 & --    & 6.81e-05 & --    & 2.23e-04 & --    \\
		%\hline
		$10 \times 10$  & 1.42e-05 & 3.03   & 8.36e-06 & 3.03  & 2.64e-05 & 3.08    \\
		%\hline
		$20 \times 20$  & 1.75e-06 & 3.02  & 1.03e-06 & 3.01  & 3.57e-06 & 2.89  \\
		%\hline
		$40 \times 40$  & 2.17e-07 & 3.01  & 1.28e-07 & 3.01  & 4.77e-07 & 2.91  \\
		\hline
	\end{tabular}
\end{table}
\item With the OE technique.  In Table \ref{SPLDG:acctest:OE}, we apply the OE technique and still observe the expected convergence orders for the structure-preserving LDG scheme.
\begin{table}[htbp]
	\centering
	\caption{ Example \ref{exam1}. Numerical errors and convergence orders of different variables in the structure-preserving LDG scheme with OE technique. T = 0.2. }
	\label{SPLDG:acctest:OE}
	\begin{tabular}{c|c|c c|c c|c c|c c}
		\hline
		 \multicolumn{2}{c|}{} & \multicolumn{4}{c|}{$\rho_h$} & \multicolumn{4}{c}{$(\rho u)_h$} \\ 
		\hline
		Element & Mesh & $L^1$ error & order &  $L^\infty$ error & order & $L^1$ error & order & $L^\infty$ error & order\\
		\hline
		\multirow{5}{*}{$P^2$} 
		& $10\times10$ & 1.99e-05 & -- & 2.33e-04 & -- & 2.99e-05 & --  & 9.76e-05 & --  \\
		%\hline
		& $20\times20$ & 2.09e-06 & 3.25 & 3.70e-05 & 2.66 & 3.13e-06 & 3.26 & 1.32e-05 & 2.88 \\
		%\hline
		& $40\times40$ & 2.63e-07 & 2.99 & 4.94e-06 & 2.90  & 3.65e-07 & 3.10 & 1.61e-06 & 3.04 \\
		%\hline
		& $80\times80$ & 3.55e-08 & 2.89 & 6.19e-07 & 3.00 & 4.57e-08 & 3.00 & 2.20e-07 & 2.88 \\
		%\hline
		& $160\times160$ & 4.68e-09 & 2.92 & 7.67e-08 & 3.01 & 5.77e-09 & 2.98 & 2.89e-08 & 2.93\\
		\hline
		\multicolumn{2}{c|}{} 
		& \multicolumn{4}{c|}{$(\rho v)_h$} 
		& \multicolumn{4}{c}{$E_{tot}$} \\
		\hline
		Element & Mesh 
		& $L^1$ error & order & $L^\infty$ error & order 
		& $L^1$ error & order & $L^\infty$ error & order \\
		\hline
		\multirow{5}{*}{$P^2$} 
		& $10\times10$  & 2.99e-05 & --  & 9.76e-05 & --  & 3.39e-05 & --  & 1.93e-04 & -- \\
		& $20\times20$  & 3.13e-06 & 3.26 & 1.32e-05 & 2.88 & 4.16e-06 & 3.03 & 2.60e-05 & 2.89 \\
		& $40\times40$  & 3.65e-07 & 3.10 & 1.61e-06 & 3.04 & 5.24e-07 & 2.99 & 3.21e-06 & 3.02 \\
		& $80\times80$  & 4.57e-08 & 3.00 & 2.20e-07 & 2.88 & 6.66e-08 & 2.98 & 3.90e-07 & 3.04 \\
		& $160\times160$ & 5.77e-09 & 2.98 & 2.89e-08 & 2.93 & 8.41e-09 & 2.99 & 4.74e-08 & 3.04 \\
		\hline
	\end{tabular}
\end{table}
\end{enumerate}
	}
\end{exa}

\begin{exa} {\em
		\label{exam2}
		({\bf{Well-balanced property}})
	Consider the following steady states:
	\begin{subequations}
		\begin{equation*}
			\theta(r) = \dfrac{1}{\pi}\int_{0}^{\pi}\cos(r\sin \alpha)\dd{\alpha},\, r = \dfrac{\sqrt{x^2+y^2}}{a},
		\end{equation*}
		\begin{equation*}
			\rho^e(x,y) = \lambda \theta^n,\; u^e = 0,\; v^e = 0,\; p^e(x,y) = \kappa(\rho^e)^2,\; \phi^e(x,y) = - \dfrac{\gamma}{\gamma-1}\kappa\bdp{\rho^e}^{\gamma-1},
		\end{equation*}\label{eq:2D}
	\end{subequations}
	where
	\begin{equation*}
		a = \sqrt{\dfrac{\kappa(n+1)\lambda^{\frac{1-n}{n}}}{4\pi G}} = \sqrt{\dfrac{\kappa}{2\pi G}}.
	\end{equation*}
	Let $\kappa = 1$, $G = 1$, $\lambda = 1$, $n = 1$, and $\gamma = 2$. The computational domain $\Omega=[-0.5,0.5]^2$.
	We set the time $T = 5.0$. We show the numerical $L^1$ errors  and $L^{\infty}$ errors computed on different grids in Table \ref{SPLDG:WBtest}. We observe that the errors are at the level of machine precision; namely, the well-balanced property is achieved.

	\begin{table}[h!]
		\centering
		\caption{Example \ref{exam2}. Numerical errors and convergence orders of different variables in the structure-preserving LDG schemes. T = 5.0. $P^2$ element.}
		\label{SPLDG:WBtest}
		\begin{tabular}{c|c|c c|c|c|c c}
			\hline
			Variables & Mesh & $L^1$ error & $L^\infty$ error & Variables & Mesh & $L^1$ error & $L^\infty$ error \\ 
			\hline 
			\multirow{4}{*}{$\rho_h$} & $10 \times 10$  & 2.00e-15 & 1.20e-14 & \multirow{4}{*}{$(\rho u)_h$} & $10\times10$  & 2.72e-14 & 1.92e-13 \\
			% \cline{2-4} \cline{6-8}
			& $20 \times 20$  & 2.73e-15 & 1.78e-14 &  & $20 \times 20$  & 5.17e-14 & 5.04e-13 \\
			%\cline{2-4} \cline{6-8}
			& $40 \times 40$  & 4.32e-15 & 3.86e-14 &  & $40 \times 40$  & 9.87e-14 & 1.08e-12 \\
			%\cline{2-4} \cline{6-8}
			& $80 \times 80$  & 8.10e-15 & 7.06e-14 &  & $80 \times 80$  & 2.00e-13 & 2.36e-12 \\
			\hline
			\multirow{4}{*}{$(\rho v)_h$} & $10 \times 10$  & 2.80e-14 & 2.82e-13 & \multirow{4}{*}{$E_h$} & $10 \times 10$  & 1.58e-15 & 1.07e-14 \\
			%\cline{2-4} \cline{6-8}
			& $20 \times 20$  & 5.19e-14 & 8.05e-13 &  & $20 \times 20$  & 2.31e-15 & 1.78e-14 \\
			%\cline{2-4} \cline{6-8}
			& $40 \times 40$  & 1.03e-13 & 1.70e-12 &  & $40 \times 40$  & 6.42e-15 & 3.20e-14 \\
			%\cline{2-4} \cline{6-8}
			& $80 \times 80$  & 2.00e-13 & 3.45e-12 &  & $80 \times 80$  & 1.70e-14 & 6.57e-14 \\
			\hline
		\end{tabular}
	\end{table}
}
\end{exa}

\begin{exa} {\em
	\label{exam3}
	({\bf{Pertubation to equilibrium state}}) 
	In this example, we consider two types of small perturbations to the equilibrium state on \eqref{eq:2D}. The first is a radially symmetric perturbation given by
	\begin{equation*}
		p = p^e + \mu e^{-100(x^2 + y^2)},
	\end{equation*}
	where $\mu = 0.01$. The simulation is carried out up to final time $T = 0.1$ on a uniform $200\times 200$ mesh, with transmissive boundary conditions applied. Figure \ref{fig:sym2D} illustrates the contour plots of the velocity magnitude and pressure perturbation obtained from both the structure-preserving LDG scheme and the standard LDG scheme. We observe that the standard LDG scheme fails to preserve the radial symmetry of the solution, resulting in a non-radially symmetric profile. Another kind of perturbation is an asymmetric perturbation defined by
	\begin{equation*}
		p = p^e + \mu e^{-100((x - 0.3)^2 + (y - 0.3)^2)}.
	\end{equation*}
	with $\mu = 0.1$. The simulation is performed up to final time $T = 0.1$ on a uniform $200\times 200$ mesh under transmissive boundary conditions as well. In Figure \ref{fig:nonsym2D}, we show the contour plots of the both schemes.  It is evident that the structure-preserving LDG scheme accurately captures the small perturbation, whereas the standard LDG scheme fails to do so.
	\begin{figure}[htbp]
		\centering
		\begin{minipage}[b]{0.45\textwidth}
			\centering
			\includegraphics[scale=0.4]{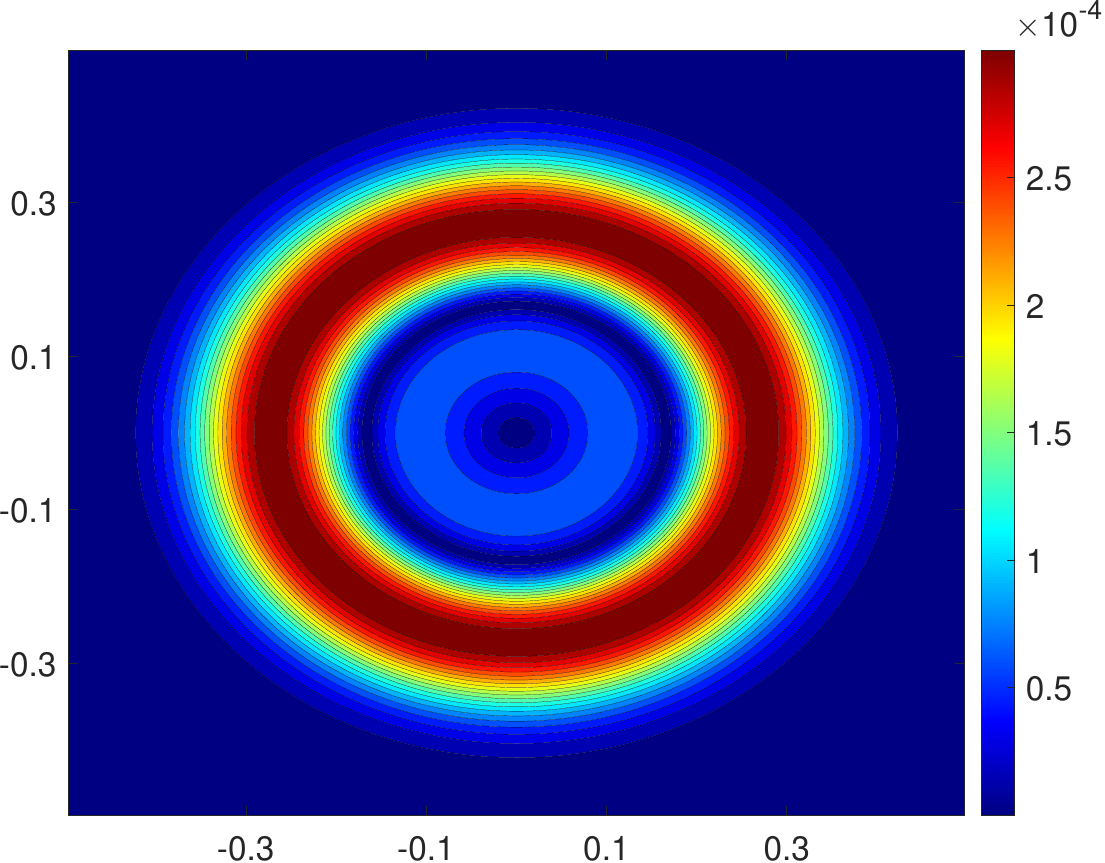}
			\caption*{(a) Structure-preserving LDG scheme: velocity magnitude $||\mfu||$.}
		\end{minipage}
		\hfill
		\begin{minipage}[b]{0.45\textwidth}
			\centering
			\includegraphics[scale=0.4]{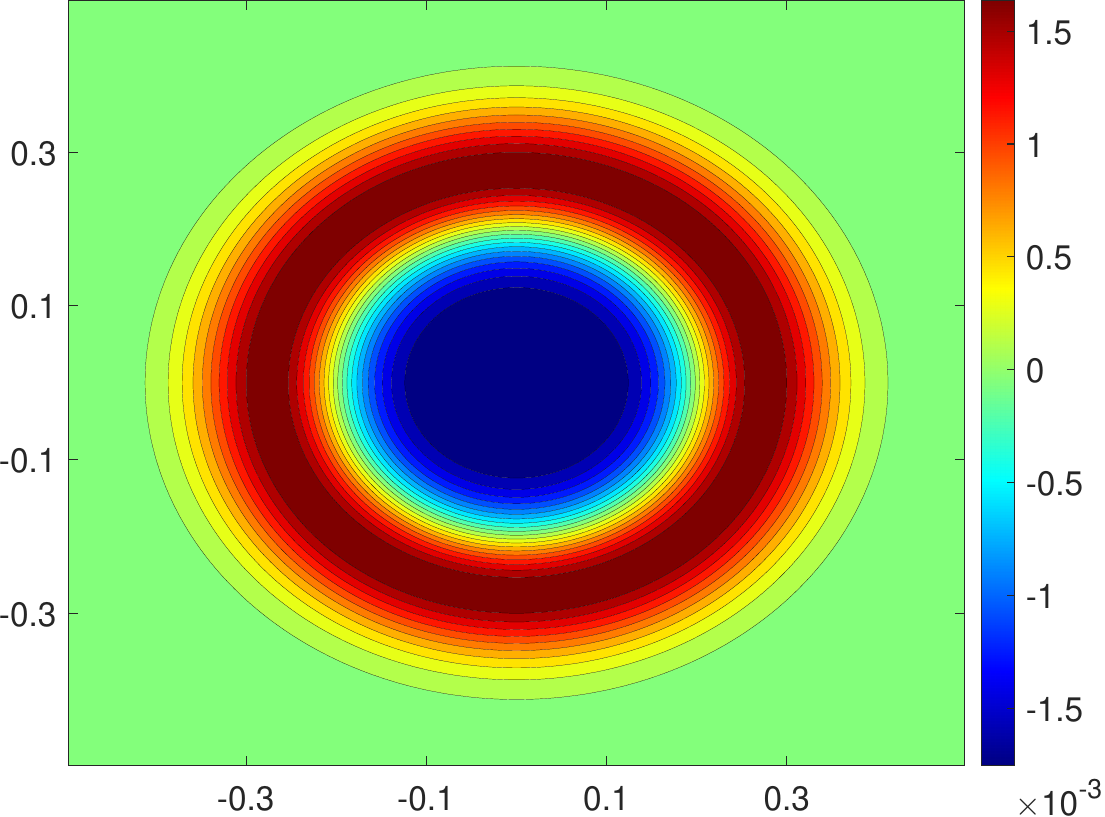}
			\caption*{(b) Structure-preserving LDG scheme: pressure perturbation.}
		\end{minipage}
		\vskip\baselineskip
		\begin{minipage}[b]{0.45\textwidth}
			\centering
			\includegraphics[scale=0.4]{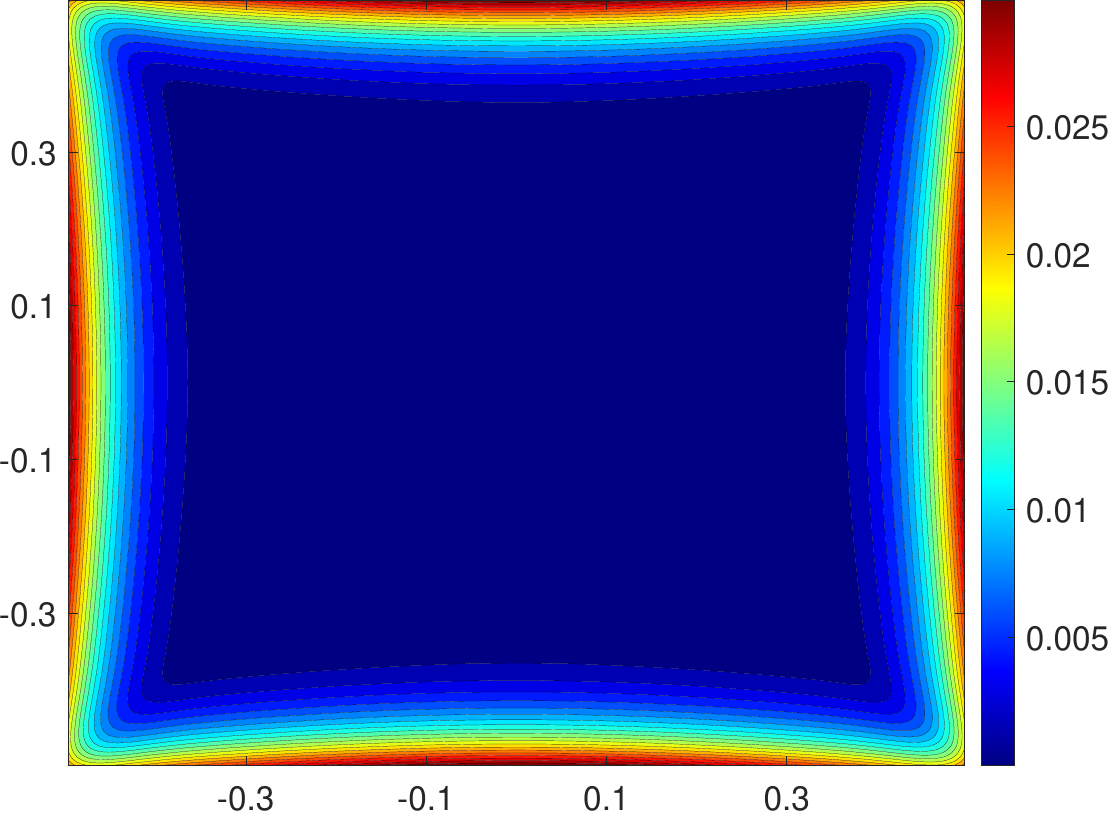}
			\caption*{(c) Standard LDG scheme: velocity magnitude $||\mfu||$.}
		\end{minipage}
		\hfill
		\begin{minipage}[b]{0.45\textwidth}
			\centering
			\includegraphics[scale=0.4]{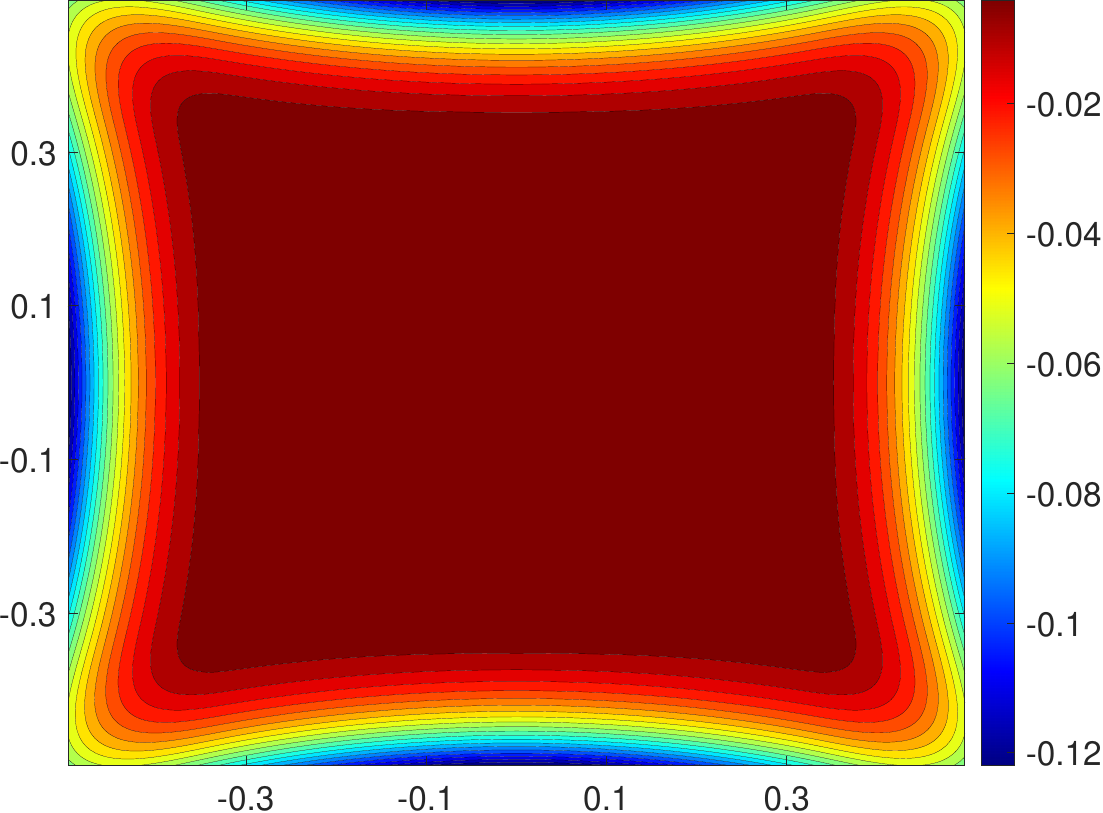}
			\caption*{(d) Standard LDG scheme: pressure perturbation.}
		\end{minipage}
		\caption{Example \ref{exam3}: The contour plots of the velocity magnitude and pressure perturbation of radially symmetric perturbation at time $T = 0.1$ by structure-preserving LDG scheme and standard LDG scheme on $200\times 200$ uniform meshes.}
		\label{fig:sym2D}
	\end{figure}
	
	\begin{figure}[htbp]
		\centering
		\begin{minipage}[b]{0.45\textwidth}
			\centering
			\includegraphics[scale=0.4]{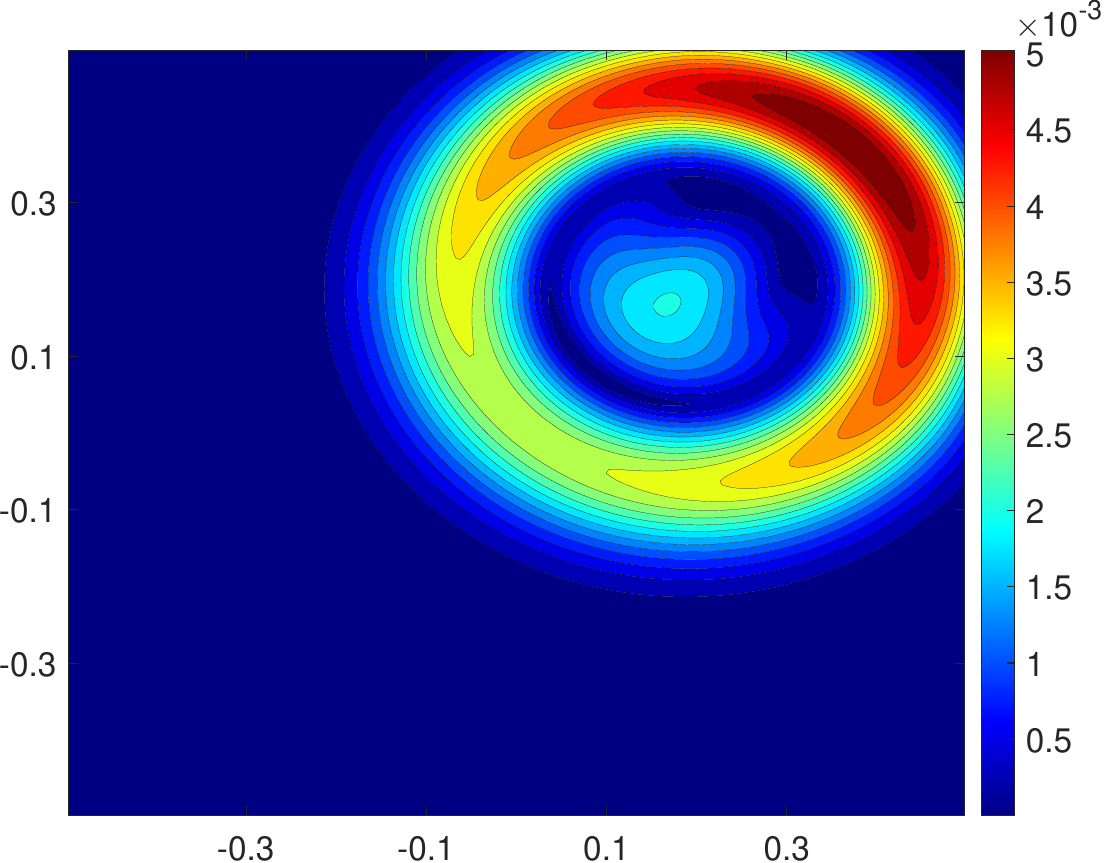}
			\caption*{(a) Structure-preserving LDG scheme: velocity magnitude $||\mfu||$.}
		\end{minipage}
		\hfill
		\begin{minipage}[b]{0.45\textwidth}
			\centering
			\includegraphics[scale=0.4]{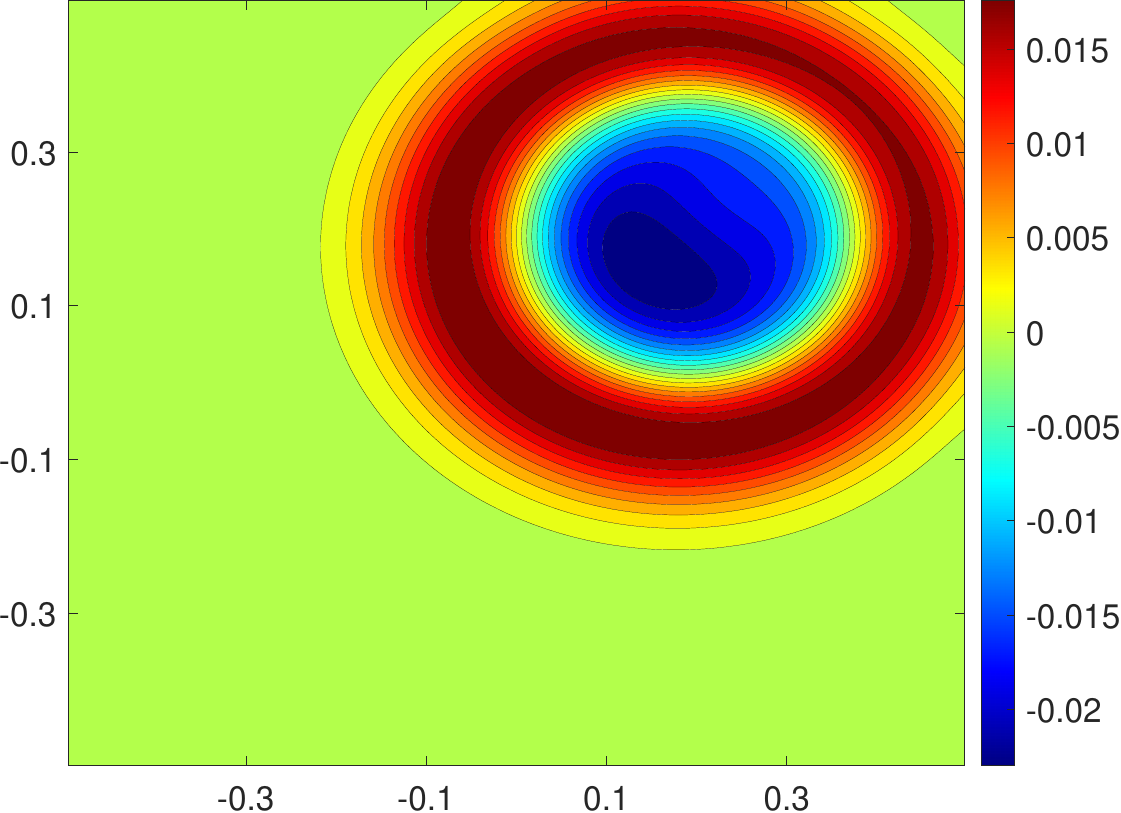}
			\caption*{(b) Structure-preserving LDG scheme: pressure perturbation.}
		\end{minipage}
		\vskip\baselineskip
		\begin{minipage}[b]{0.45\textwidth}
			\centering
			\includegraphics[scale=0.4]{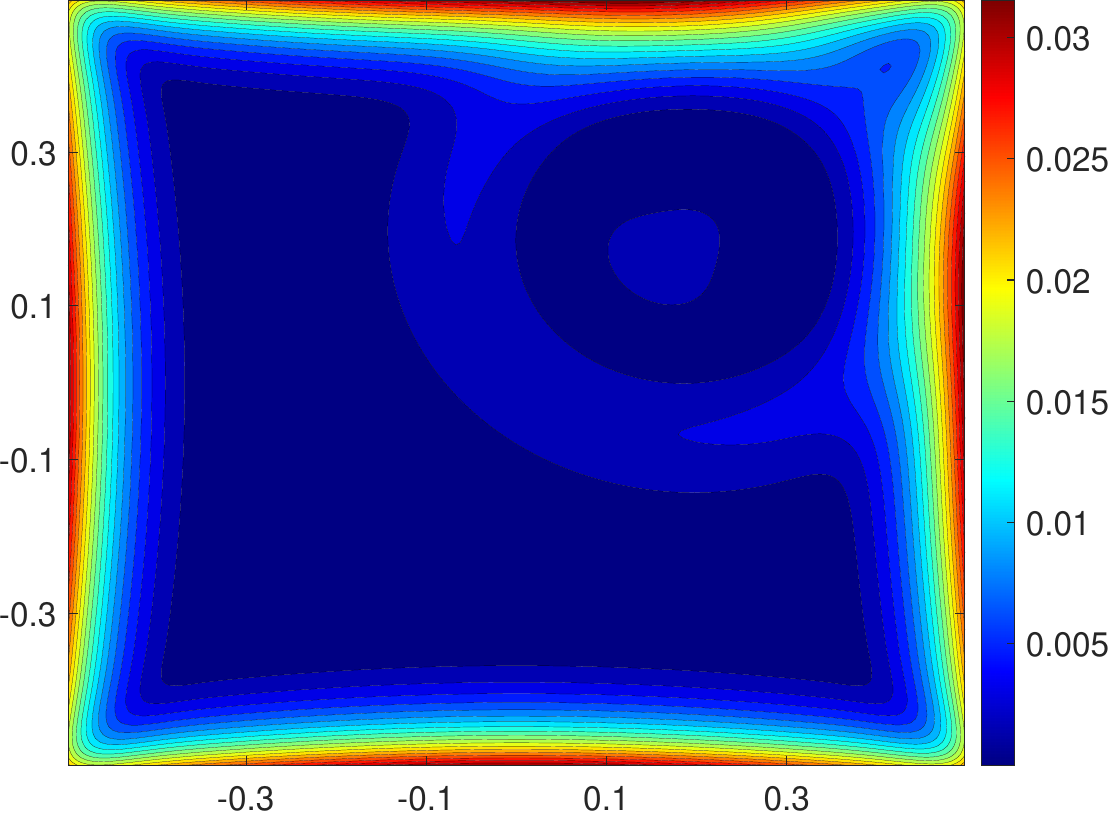}
			\caption*{(c) Standard LDG scheme: velocity magnitude $||\mfu||$.}
		\end{minipage}
		\hfill
		\begin{minipage}[b]{0.45\textwidth}
			\centering
			\includegraphics[scale=0.4]{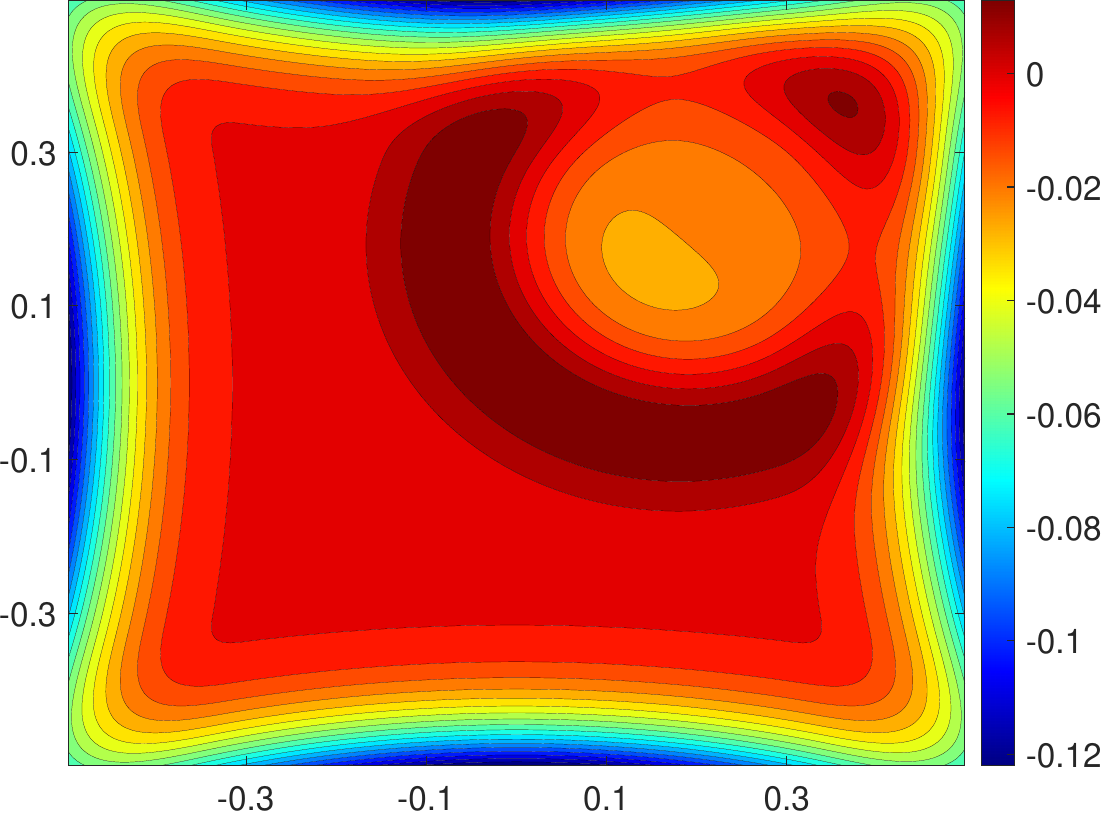}
			\caption*{(d) Standard LDG scheme: pressure perturbation.}
		\end{minipage}
		\caption{Example \ref{exam3}: The contour plots of the velocity magnitude and pressure perturbation of non-radially symmetric perturbation at time $T = 0.1$ by structure-preserving LDG schemes and standard LDG schemes on $200\times 200$ uniform meshes.}
		\label{fig:nonsym2D}
	\end{figure}
}
\end{exa}

\begin{exa} {\em
		\label{blast2D}
		({\bf{Blast wave problem}}) To further check the robustness, TEC property, and ability in shock-capturing of the algorithm, we consider the blast wave problem. We apply a huge jump around the center similar in \cite{wu2019provably,wu2021uniformly,DU2024WBPP,REN2023FVPP} to simulate the explosion:
		\begin{equation*}
			p_0(r) = p^e(r) + \left\{
			\begin{aligned}
				&100, &{\rm for}& \; r < 0.1, \\
				&0, &{\rm for}& \; r \geq 0.1.
			\end{aligned} \right.
		\end{equation*}
		In this test, both OE technique and positivity-preserving (PP) limiter in \cite{2020PPlimiterDG,wu2021uniformly,2012PPTwodimension} are applied. This PP limiter can ensure the positivity of the quadrature points in every cell. In this numerical test, we set the boundary condition as transmissive boundary conditions.
		The numerical results on a $400\times 400$ mesh at the final time $0.05$ are shown in Figure~\ref{fig:explotion2D}. We observe that our structure-preserving LDG scheme can capture shock discontinuities as well as the standard LDG scheme and does not influence the robustness and the shock-capturing ability. The radial symmetric structure of density and pressure is also preserved well.
		
		Moreover, we can observe the TEC property of our structure preserving scheme in Figure~\ref{fig:Etot}. The structure-preserving LDG scheme can conserve the total energy to $7.97\times10^{-6}$, while the standard LDG scheme produces an error of about $2.78\times10^{-4}.$ 
		
		Besides, we add the following pressure perturbation mentioned in \cite{Käppeli2019HWB} to the equilibrium state
		\begin{equation*}
			\delta p(\mfx) = 100\sum_{i = 1}^{5}\mfI_{B(\mfx_i,r)}(\mfx),
		\end{equation*}
		where $B(\mfx_i,r) = \bdc{\mfx'\in\mathbb{R}^d:||\mfx' - \mfx_i||<r}$ denotes the open ball of radius $r$ centered on $\mfx_i$ and $I_B$ is the indicator function for the set $B$, i.e.
		\begin{equation*}
			I_B(\mfx) = \begin{cases}
				1, & \text{if }\mfx\in B, \\
				0, & \text{otherwise}.
			\end{cases}
		\end{equation*}
		We set five pressure blast points at 
		\begin{equation*}
			\mfx_1 = [-0.25,0.3]^T,\;
			\mfx_2 = [-0.25,0.1]^T,\;
			\mfx_3 = [0.025,0.3]^T,\;
			\mfx_4 = [0.025,0.225]^T,\;
			\mfx_5 = [0.1,-0.1]^T,
		\end{equation*}
		with radius $r = 0.05$ and the transmissive boundary conditions is applied. The numerical results of both schemes can be found in Figure~\ref{fig:multiexplotion2D} and they perform equally well .
		\begin{figure}[htbp]
			\centering
			\begin{minipage}[b]{0.45\textwidth}
				\centering
				\includegraphics[scale=0.4]{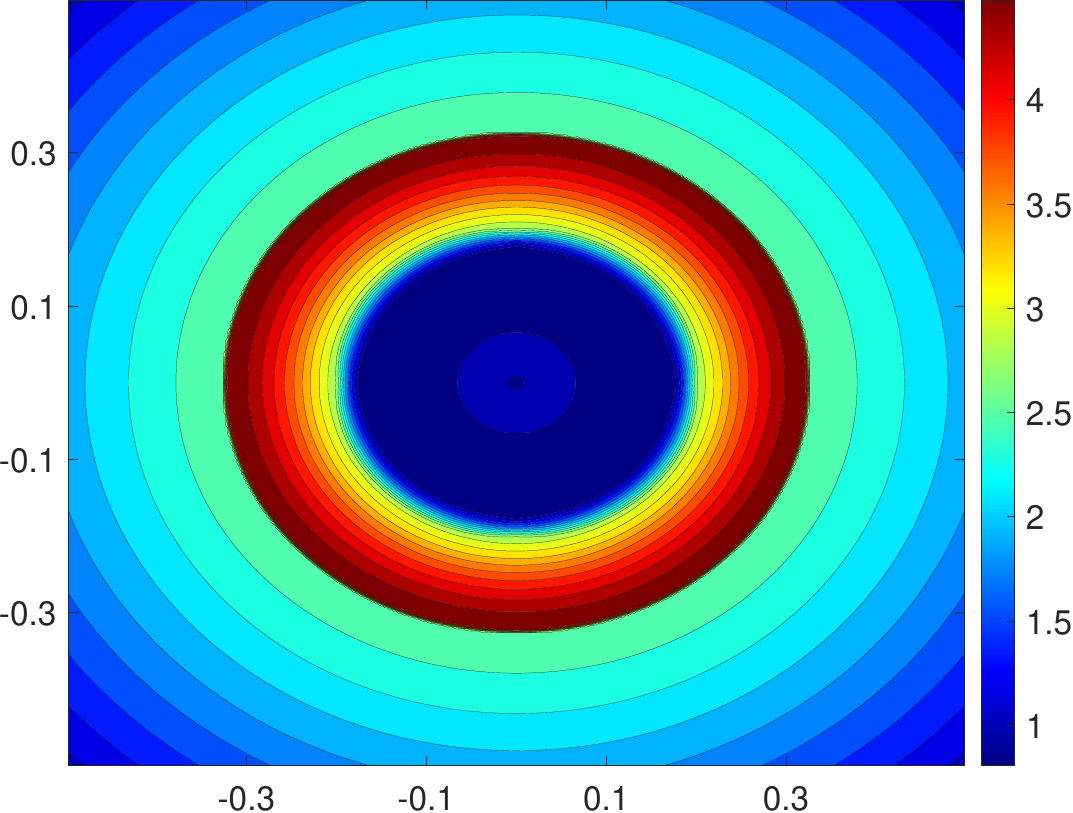}
				\caption*{(a) density $\rho$.}
			\end{minipage}
			\hspace{0.05\textwidth} % 调整间距
			\begin{minipage}[b]{0.45\textwidth}
				\centering
				\includegraphics[scale=0.4]{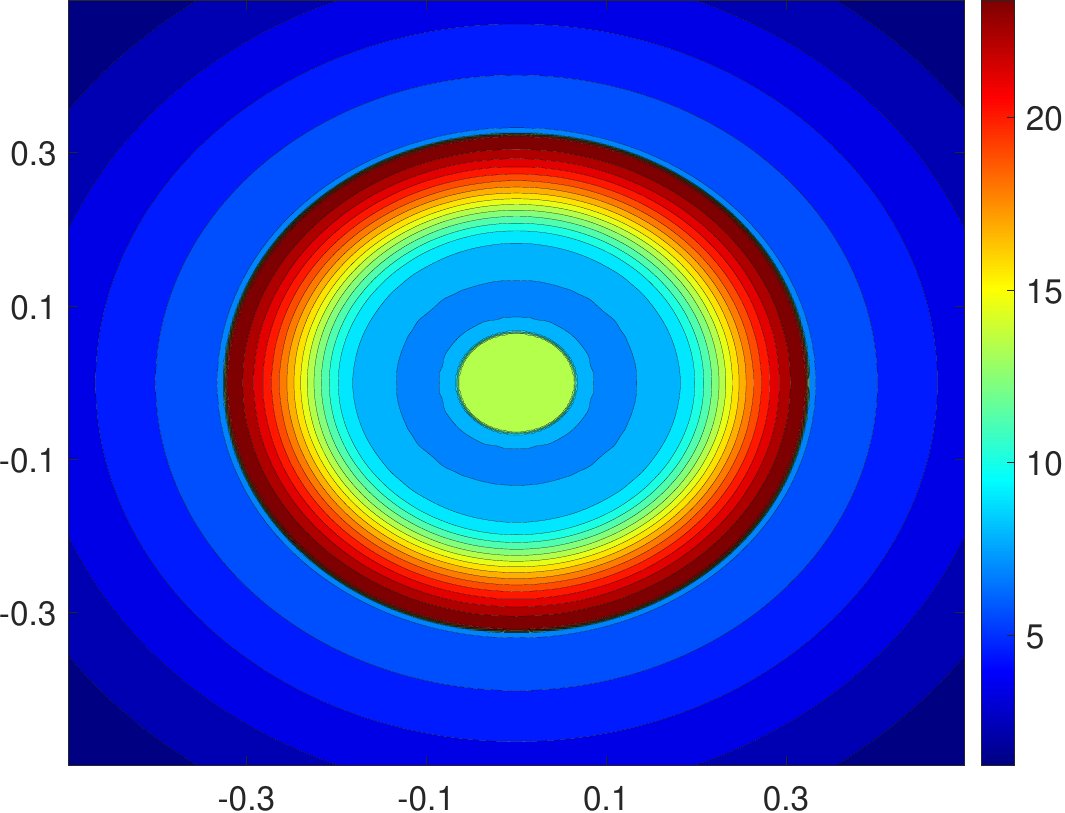}
				\caption*{(b) pressure $p$.}
			\end{minipage}
			\begin{minipage}[b]{0.45\textwidth}
				\centering
				\includegraphics[scale=0.4]{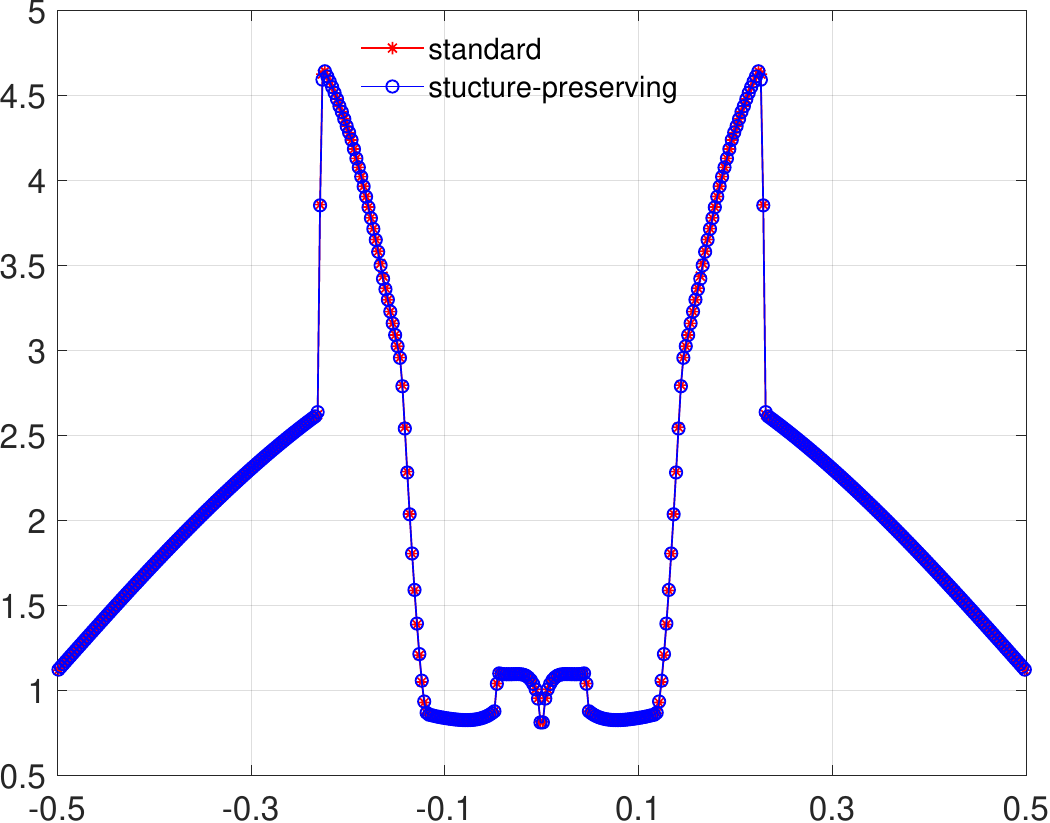}
				\caption*{(c) density $\rho$ along $x = y$.}
			\end{minipage}
			\hspace{0.05\textwidth} % 调整间距
			\begin{minipage}[b]{0.45\textwidth}
				\centering
				\includegraphics[scale=0.4]{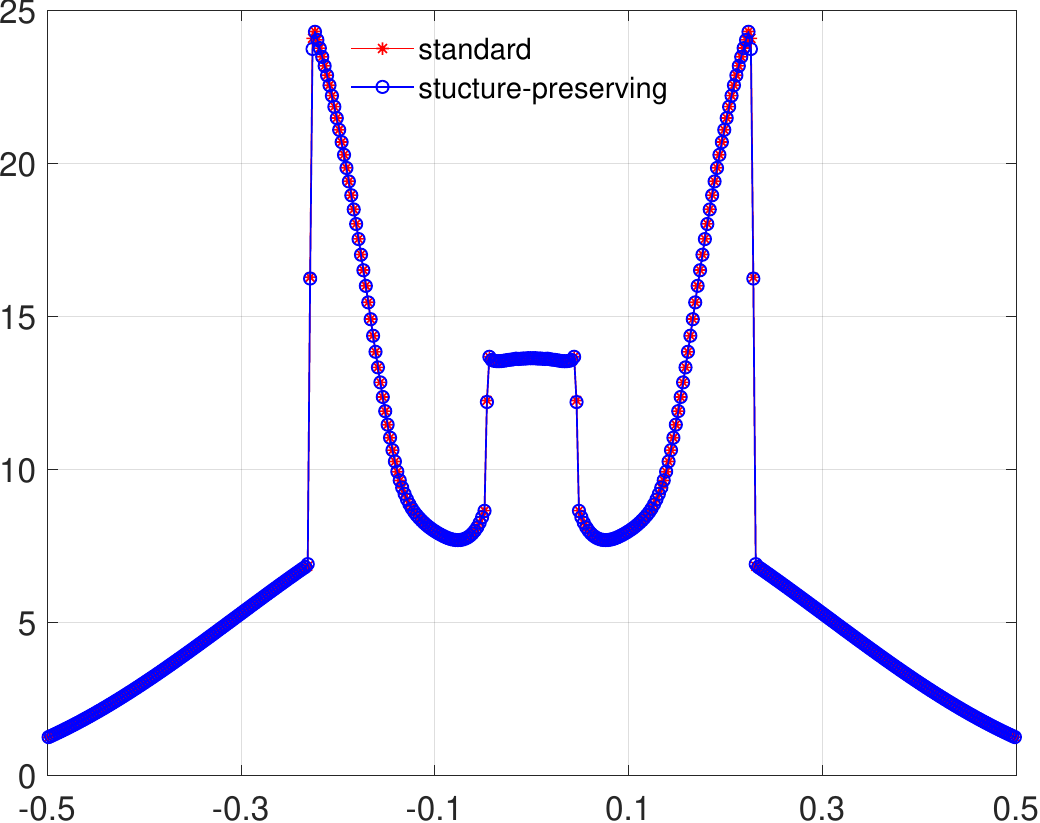}
				\caption*{(d) pressure $p$ along $x = y$.}
			\end{minipage}
			\caption{Example \ref{blast2D}: The contour plots of the density $\rho$ (a) and pressure $p$ (b) of blast wave problem at time $T = 0.05$ by the structure-preserving LDG scheme on $400\times 400$ uniform mesh. And the comparison between the structure-preserving LDG scheme and the standard LDG scheme along the line $x = y$ (c),(d).}
			\label{fig:explotion2D}
		\end{figure}
		\begin{figure}[htbp]
			\centering
			\includegraphics[scale=0.4]{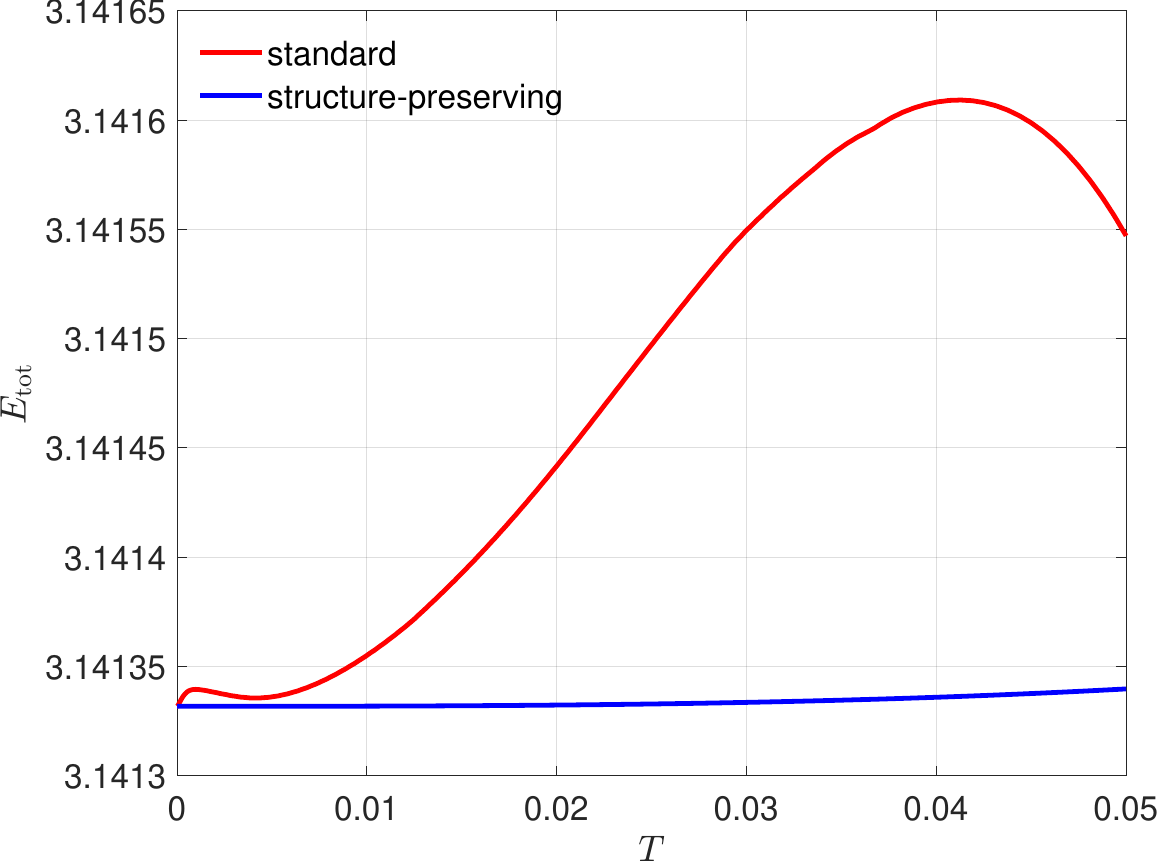}
			\caption*{(a) Total energy $E_{tot}$ with time $T$.}
			\caption{Example \ref{blast2D}: The plot of the total energy $E_{tot}$ (a) blast wave problem until $T = 0.05$ by structure-preserving LDG scheme and standard LDG scheme on a $400\times 400$ uniform mesh.}
			\label{fig:Etot}
		\end{figure}
		\begin{figure}[htbp]
			\centering
			\begin{minipage}[b]{0.45\textwidth}
				\centering
				\includegraphics[scale=0.4]{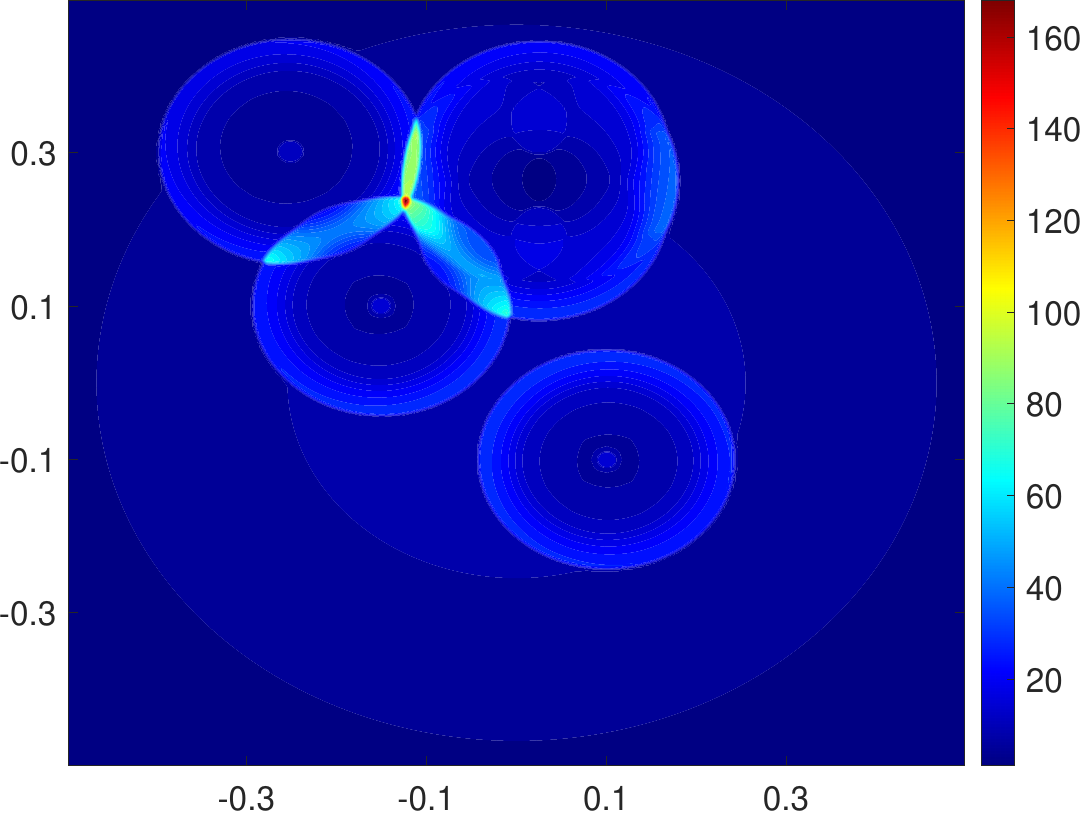}
				\caption*{(a)}
			\end{minipage}
			\hspace{0.05\textwidth} % 调整间距
			\begin{minipage}[b]{0.45\textwidth}
				\centering
				\includegraphics[scale=0.4]{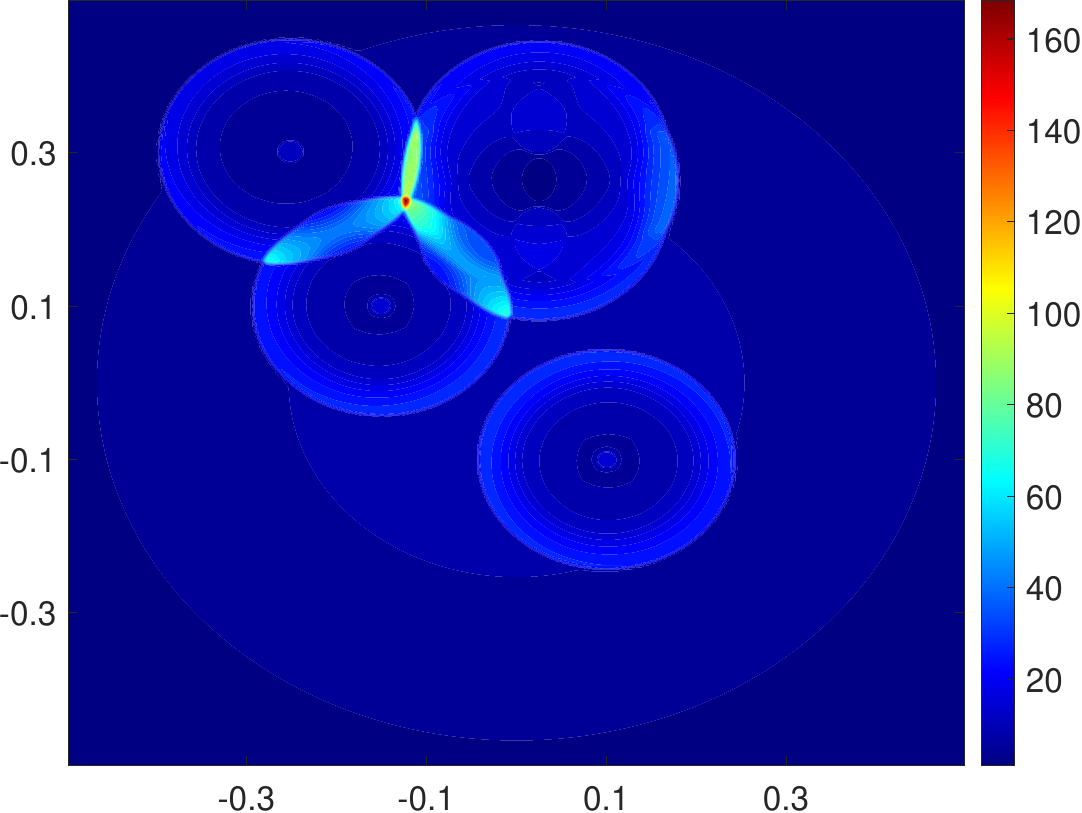}
				\caption*{(b)}
			\end{minipage}
			\caption{Example \ref{blast2D}: The contour plots of the pressure by the structure-preserving LDG scheme (a) and the standard LDG scheme (b) of another blast wave problem at time $T = 0.02$ on a $400\times 400$ uniform mesh.}
			\label{fig:multiexplotion2D}
		\end{figure}
	}
\end{exa}

\begin{exa} {\em
		\label{jeans}
		\begin{figure}[htbp]
			\centering
			\begin{minipage}[b]{0.45\textwidth}
				\centering
				\includegraphics[scale=0.4]{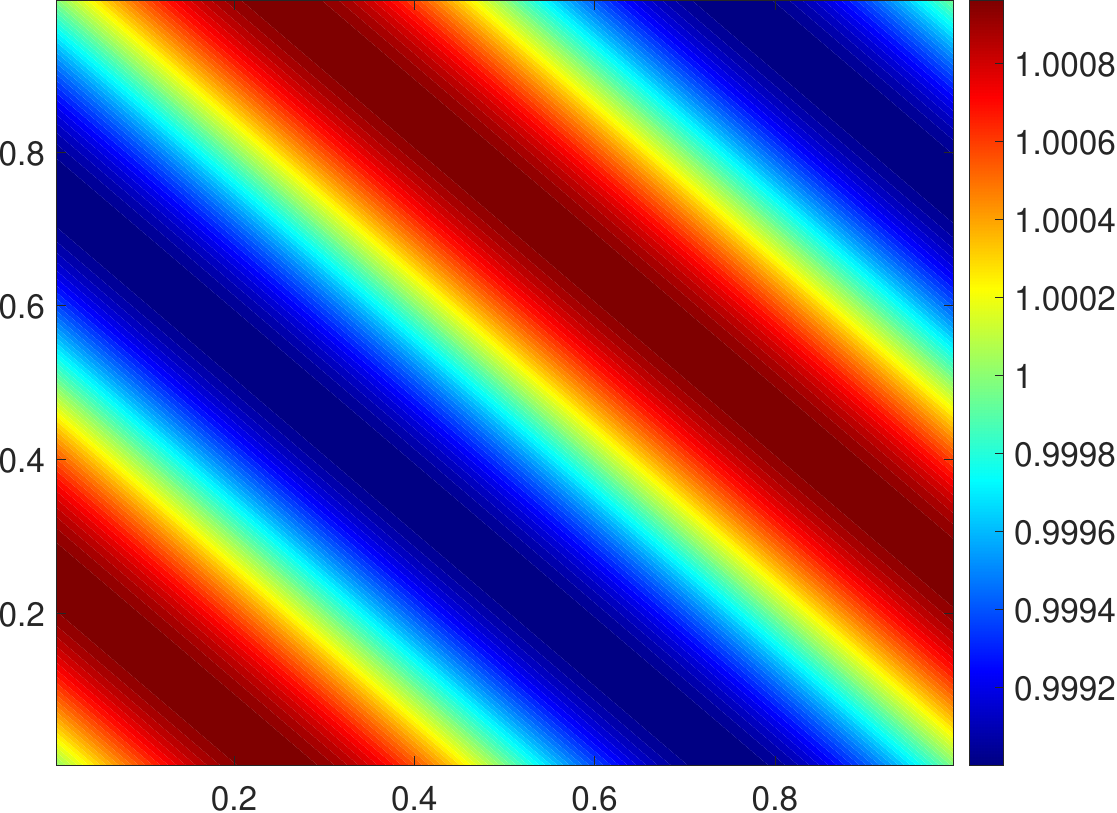}
				\caption*{(a) $T = 0.0$.}
			\end{minipage}
			\hspace{0.05\textwidth} % 调整间距
			\begin{minipage}[b]{0.45\textwidth}
				\centering
				\includegraphics[scale=0.4]{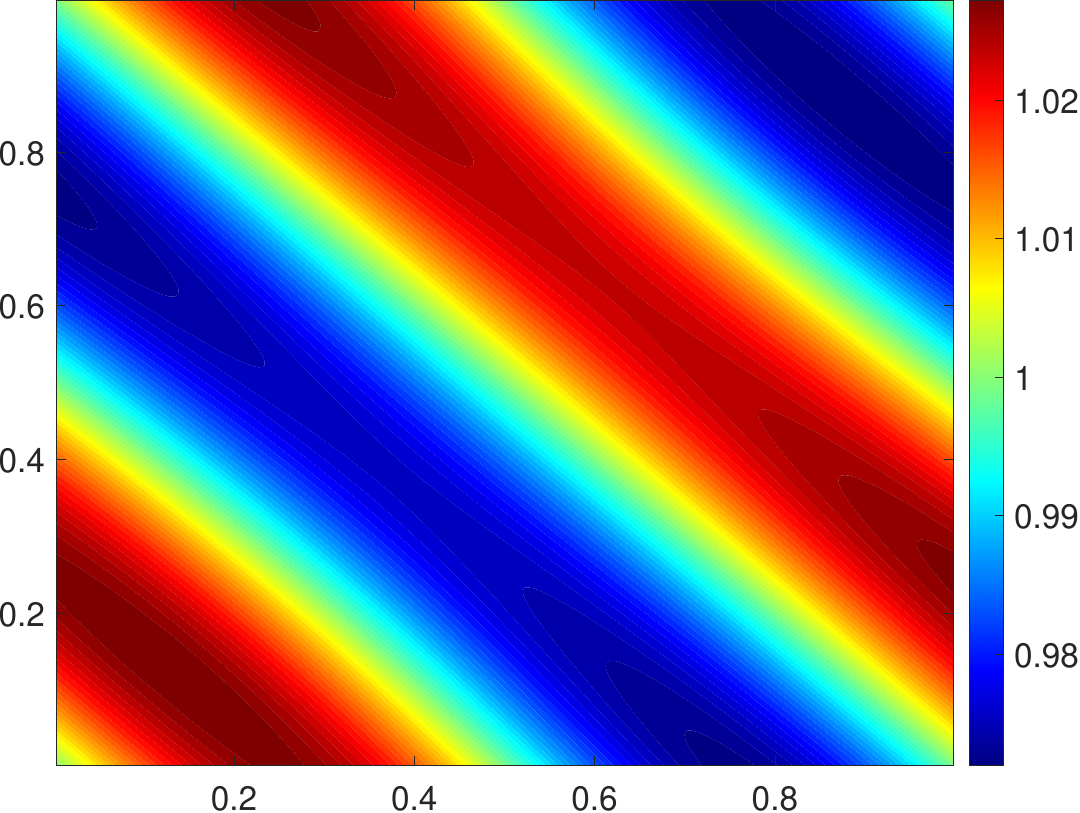}
				\caption*{(b) $T = 1.8$.}
			\end{minipage}
			\begin{minipage}[b]{0.45\textwidth}
				\centering
				\includegraphics[scale=0.4]{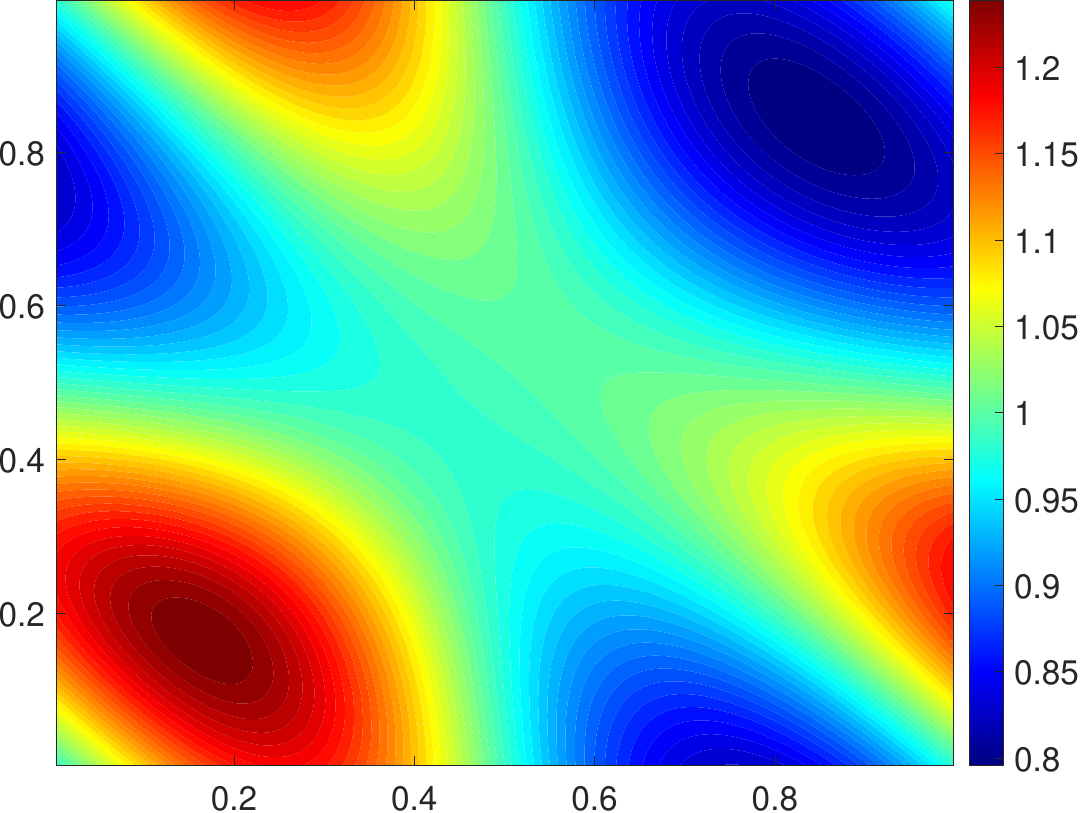}
				\caption*{(c) $T = 2.4$.}
			\end{minipage}
			\hspace{0.05\textwidth} % 调整间距
			\begin{minipage}[b]{0.45\textwidth}
				\centering				\includegraphics[scale=0.4]{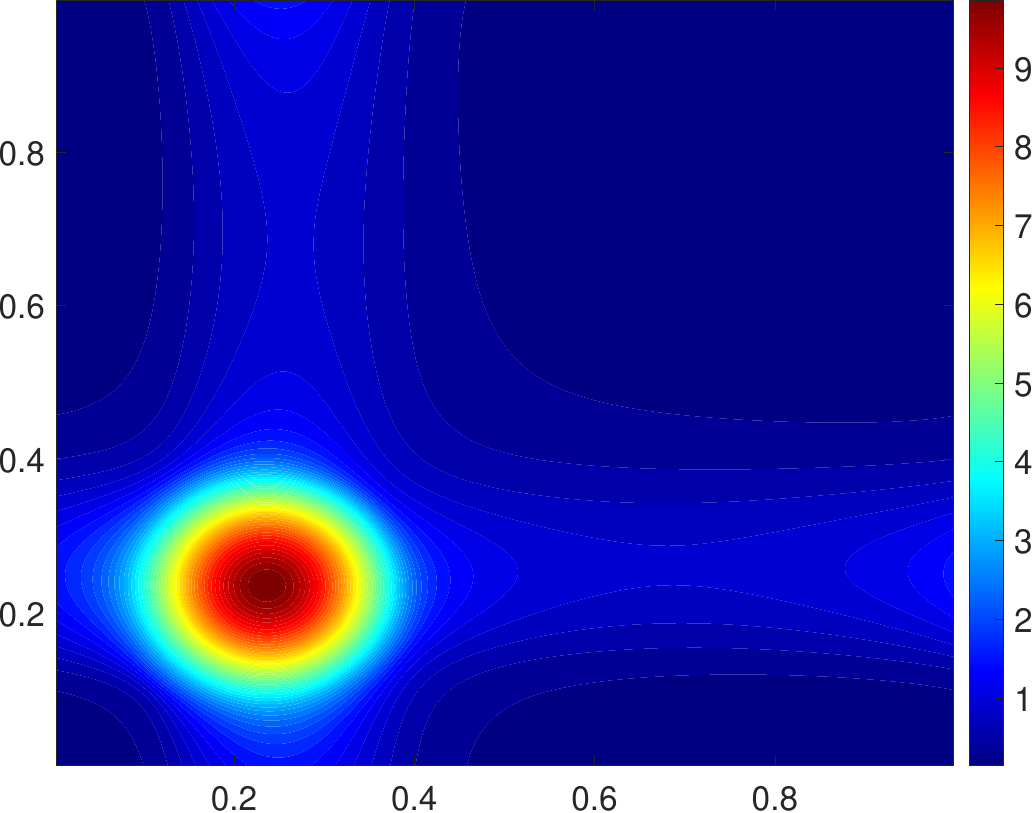}
				\caption*{(d) $T = 2.8$.}
			\end{minipage}
			\caption{Example \ref{jeans}: The contour plots of the density $\rho$ at times $T = 0.0$ (a), $T = 1.8$ (b), $T = 2.4$ (c), $T = 2.8$ (d)  by the structure-preserving LDG scheme on a $200\times 200$ uniform mesh.}
			\label{fig:unstable_jeans}
		\end{figure}
		\begin{figure}[htbp]
			\centering
			\includegraphics[scale=0.4]{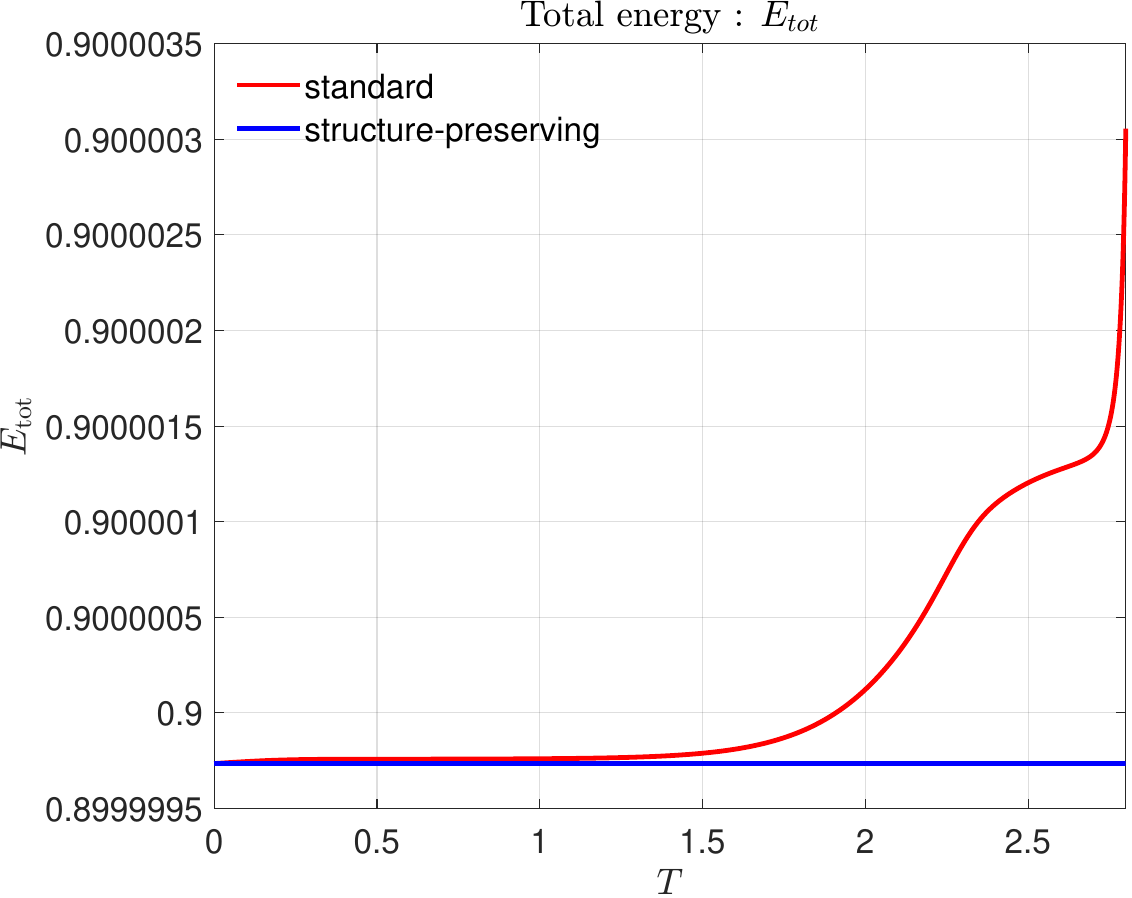}
			\caption*{(a) Total energy $E_{tot}$ with time $T$.}
			\caption{Example \ref{jeans}: The plot of the total energy $E_{tot}$ (a) in the Jeans instability problem until $T = 2.8$ by the structure-preserving LDG scheme and standard LDG scheme on a $200\times 200$ uniform mesh.}
			\label{fig:Etot_jeans_unstable}
		\end{figure}
		\begin{figure}[htbp]
			\centering
			\begin{minipage}[b]{0.45\textwidth}
				\centering
				\includegraphics[scale=0.4]{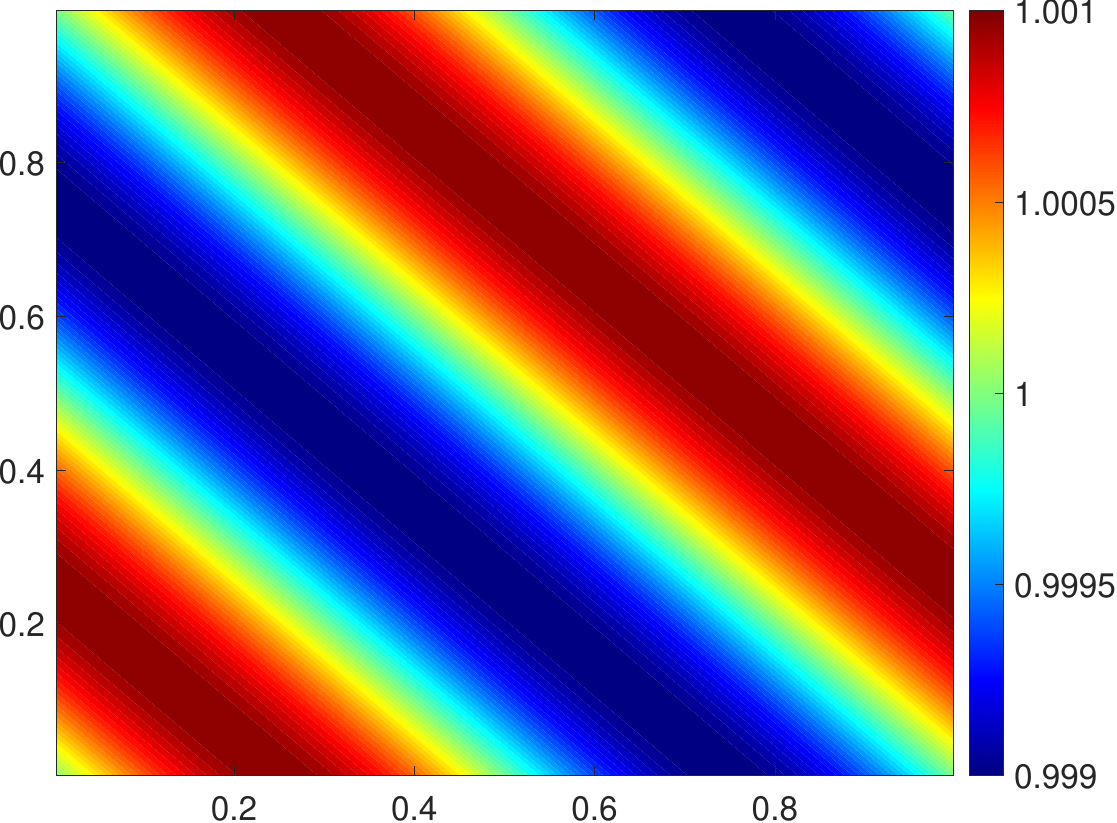}
				\caption*{(a) $T = 0.0$.}
			\end{minipage}
			\hspace{0.05\textwidth} % 调整间距
			\begin{minipage}[b]{0.45\textwidth}
				\centering
				\includegraphics[scale=0.4]{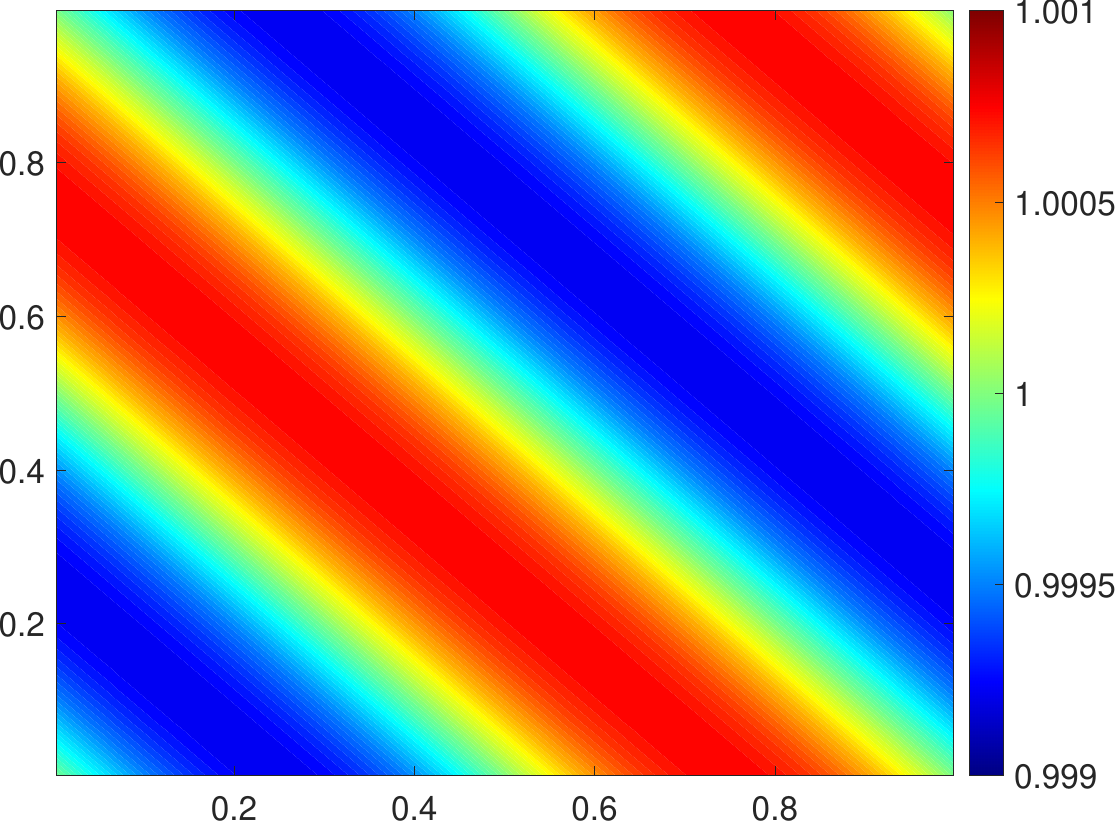}
				\caption*{(b) $T = 1.2$.}
			\end{minipage}
			\begin{minipage}[b]{0.45\textwidth}
				\centering
				\includegraphics[scale=0.4]{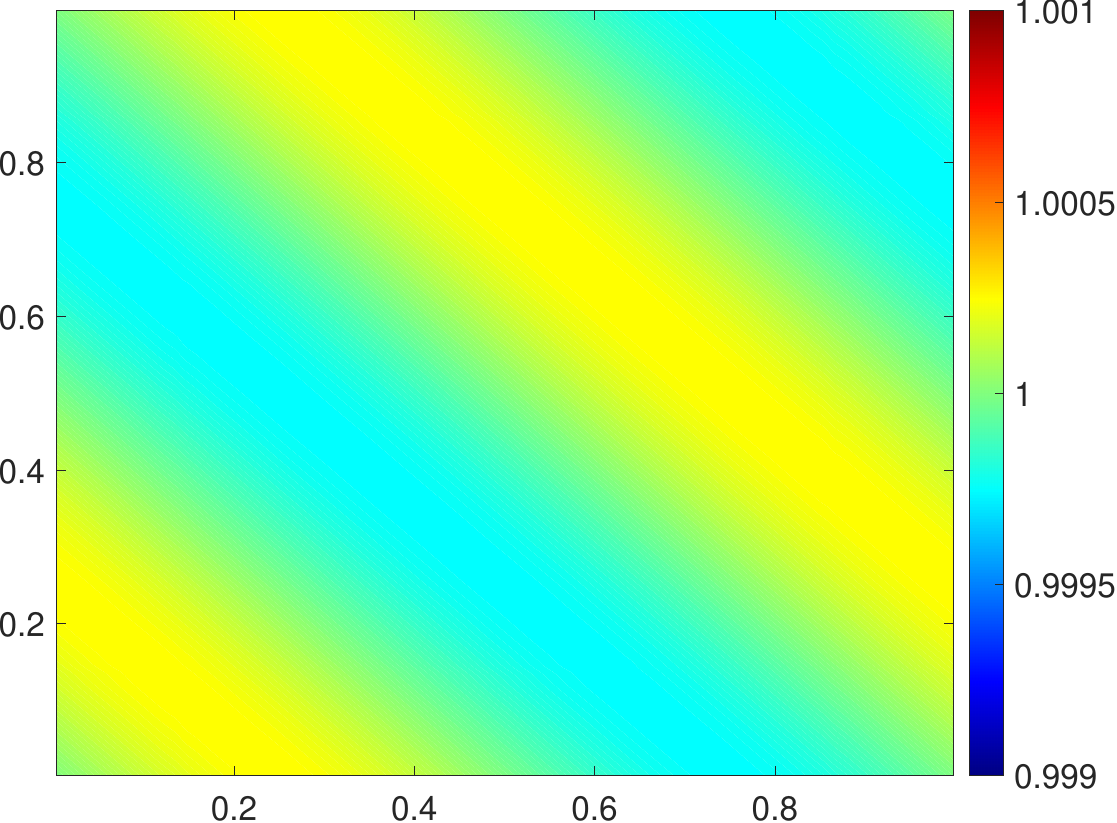}
				\caption*{(c) $T = 2.4$.}
			\end{minipage}
			\hspace{0.05\textwidth} % 调整间距
			\begin{minipage}[b]{0.45\textwidth}
				\centering				\includegraphics[scale=0.4]{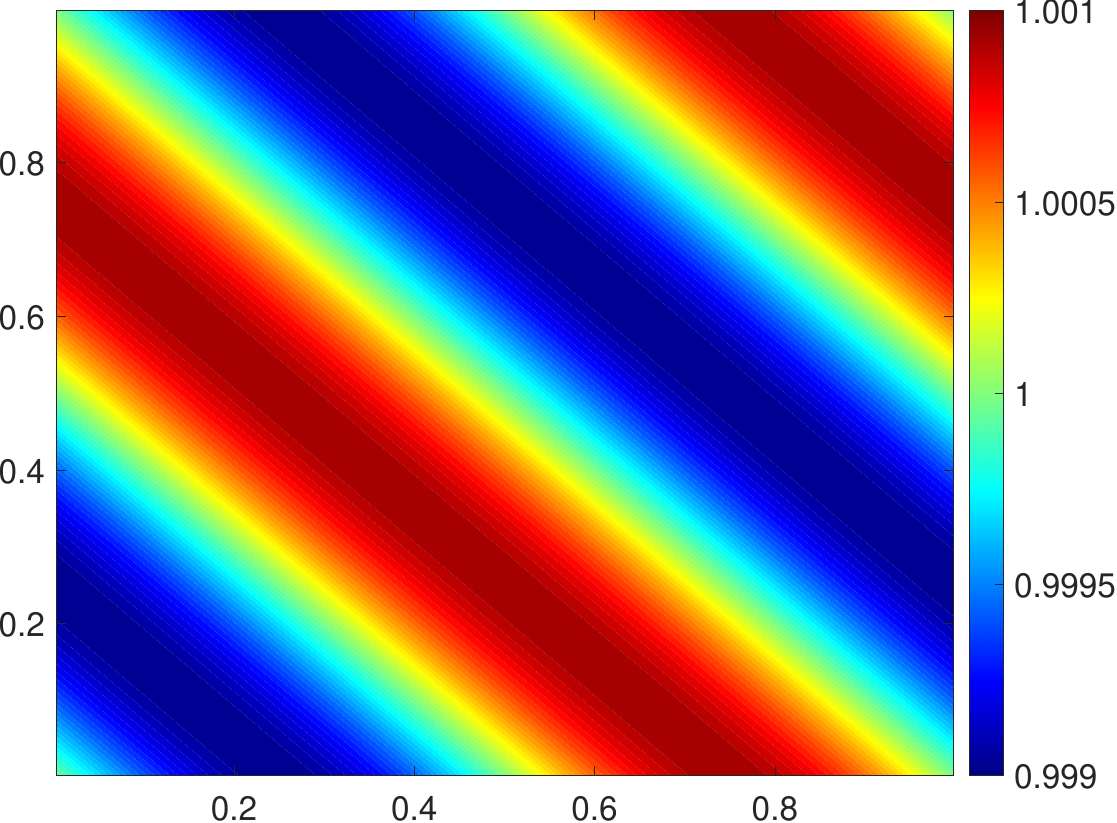}
				\caption*{(d) $T = 2.6$.}
			\end{minipage}
			\caption{Example \ref{jeans}: The contour plots of the density $\rho$ at times $T = 0.0$ (a), $T = 1.2$ (b), $T = 2.4$ (c), $T = 2.6$ (d)  by the structure-preserving LDG scheme on a $200\times 200$ uniform mesh.}
			\label{fig:stable_jeans}
		\end{figure}
		({\bf{Jeans instability}}) A classical example for an instability for compressible SG Euler equations is the so-called jeans instability problem, which was first described by Jeans \cite{1902JeansInstability}. As a fundamental problem in astrophysics, the Jeans instability problem serves as an excellent test case for verifying the TEC property of numerical schemes.
		The domain is $\Omega = [0,1]^2$ with periodic boundary conditions as well as the gravitational potential and $\gamma = 5/3$.
		
		The initial condition is similar in \cite{PurelyDG_SG2021,JIANG201348,mullen2021extension}. With the background constant density $\rho_0 = 1$ and constant pressure $p_0 = \rho_0/\gamma$, we apply a small perturbation on density and pressure.
		\begin{subequations}
			\begin{equation*}
				\rho(\mfx,0) = \rho_0\bdp{1+\delta_0\sin(\mfk\cdot\mfx)},
			\end{equation*}
			\begin{equation*}
				p(\mfx,0) = p_0\bdp{1+\delta_0\sin(\mfk\cdot\mfx)},
			\end{equation*}
			\begin{equation*}
				\mfu(\mfx,0) = \mathbf{0},
			\end{equation*}
			and the gravitational potential $\phi$ is given by:
			\begin{equation*}
				\Delta\phi = 4\pi G (\rho - \rho_0).
			\end{equation*}
		\end{subequations}
		Here, $\delta_0 = 10^{-3}$ denotes the wave amplitude, and $\mfk$ is the wave vector. The total energy is defined as
		\begin{equation*}
			E_{tot} = E + \dfrac{1}{2}(\rho - \rho_0)\phi.
		\end{equation*} 
		The amplitude of the perturbation changes with time as $e^{i\omega t}$, where $\omega$ is determined by the Jeans dispersion relation in Fourier space \cite{2006HubberSGSPH}:
		\begin{equation}
			\omega^2 = c_0^2|\mfk|^2 - 4\pi G \rho_0, \label{Jeansdispersion}
		\end{equation}
		with $c_0 = \sqrt{\gamma p_0/\rho_0}$ the sound speed. Hence we can identify the Jeans wave number $k_J$ from \eqref{Jeansdispersion}  
		\begin{equation*}
			k_J = \dfrac{\sqrt{4\pi G \rho_0}}{c_0}.
		\end{equation*}
		When $\omega^2 < 0,|\mfk|/k_J<1$, the wave perturbation becomes unstable and leads to gravitational collapse, whereas when $\omega^2 > 0,|\mfk|/k_J>1$, it undergoes stable periodic oscillations. For simplicity, we vary the gravitational constant $G$ to  investigate both unstable and stable cases.

		This test adopts the OE technique and the PP limiter, following the approach used in Example~\ref{blast2D}. We set $k_1 = k_2 = 2\pi$. The numerical results of density on a $200^2$ uniform mesh are shown in Figure~\ref{fig:unstable_jeans} and Figure~\ref{fig:stable_jeans}. Due to the periodic boundary conditions, we implement with the PardisoLU solver from Intel MKL. In Figure~\ref{fig:unstable_jeans}, we set $G = 6.674$ and observe the expected gravitational collapse near the point $(0.2375,0.2375)$. We illustrate the time history of the changes in total energy in Figure~\ref{fig:Etot_jeans_unstable}. The maximum absolute value of the changes in total energy is $2.53137\times 10^{-13}$ for the structure-preserving scheme, which is at the level of machine error, and is significantly lower than the $3.2845\times 10^{-6}$ observed in the standard scheme. A sharp increase in total energy is observed during the gravitational collapse, and the instability renders the results of the standard LDG scheme unreliable.  Additionally, Figure~\ref{fig:stable_jeans} shows the stable case with $G = 0.6674$, where the density exhibits periodic variations.
		
	}
\end{exa}

\subsection{Three dimensional case}
For the three-dimensional test, we show the performance of our structure-preserving LDG scheme and the standard LDG scheme on certain astrophysical problems. 

\begin{exa} {\em
		\label{acctest3D}
		({\bf{Accuracy test}}) 
		We consider the following smooth problem. The steady state is given by
		\begin{subequations}
			\begin{equation*}
				\rho(x,y,z,t) = \sin\bdp{\dfrac{\sqrt{3}}{3a}\bdp{x+y+z-(u_0 + v_0 + w_0)t}},
			\end{equation*}
			\begin{equation*}
				u(x,y,z,t) = u_0,\, v(x,y,z,t) = v_0,\,w(x,y,z,t) = w_0,
			\end{equation*}
			\begin{equation*}
				p(x,y,t) = \kappa\rho^2,\; \phi(x,y,t) = - 4\pi G a^2\rho,
			\end{equation*}
		\end{subequations}
		where
		\begin{equation*}
			a = \sqrt{\dfrac{\kappa(n+1)\lambda^{\frac{1-n}{n}}}{4\pi G}}.
		\end{equation*}
		Let $\kappa = 2\pi$, $G = 1/\pi$, $\lambda = 1$, $n = 1$, $\gamma = 2$, $u_0 = 0.2$, $v_0 = 0.3$, and $w_0 = 0.5$. The computational domain $\Omega$ is set as 
		\begin{equation*}
			\Omega = \left[\dfrac{1}{18}\sqrt{3}\pi a,\dfrac{5}{18}\sqrt{3}\pi a\right]^3.
		\end{equation*}
		We take the uniform mesh with $N^3$. As similar in Example~\ref{exam1}, we present the numerical errors and convergence orders of various variables at $T = 0.3$
		by the structure-preserving LDG scheme, both without the OE technique in Table~\ref{SPLDG:acctest:3D} and with the OE technique in Table~\ref{SPLDG:OEacctest:3D}. The OE technique will slightly affect the convergence orders on a coarse mesh.
	\begin{table}[htbp]
		\centering
		\caption{ Example \ref{acctest3D}. Numerical errors and convergence orders of different variables in the structure-preserving LDG scheme without OE technique. T = 0.3. }
		\label{SPLDG:acctest:3D}
		\begin{tabular}{c|c|c c|c c|c c|c c}
			\hline
			\multicolumn{2}{c|}{} & \multicolumn{4}{c|}{$\rho_h$} & \multicolumn{4}{c}{$(\rho u)_h$} \\ 
			\hline
			Element & Mesh & $L^1$ error & order & $L^\infty$ error & order & $L^1$ error & order & $L^\infty$ error & order\\
			\hline
			\multirow{4}{*}{$P^2$} 
			& $4\times4\times4$   & 1.32e-04 & --   & 5.78e-03 & --   & 5.60e-05 & --   & 1.90e-03 & --   \\  
			& $8\times8\times8$   & 1.38e-05 & 3.26 & 6.84e-04 & 3.08 & 5.98e-06 & 3.23 & 2.45e-04 & 2.96 \\  
			& $16\times16\times16$ & 1.57e-06 & 3.14 & 8.08e-05 & 3.08 & 7.05e-07 & 3.08 & 3.04e-05 & 3.01 \\  
			& $32\times32\times32$ & 1.88e-07 & 3.06 & 9.65e-06 & 3.06 & 8.53e-08 & 3.05 & 3.77e-06 & 3.01 \\  
			\hline
			\multicolumn{2}{c|}{} & \multicolumn{4}{c|}{$(\rho v)_h$} & \multicolumn{4}{c}{$(\rho w)_h$} \\  
			\hline
			Element & Mesh & $L^1$ error & order & $L^\infty$ error & order & $L^1$ error & order & $L^\infty$ error & order\\  
			\hline
			\multirow{4}{*}{$P^2$}   
			& $4\times4\times4$   & 6.85e-05 & --   & 2.60e-03 & --   & 9.58e-05 & --   & 4.05e-03 & --   \\  
			& $8\times8\times8$   & 7.26e-06 & 3.24 & 3.32e-04 & 2.97 & 9.81e-06 & 3.29 & 5.15e-04 & 2.98 \\  
			& $16\times16\times16$ & 8.23e-07 & 3.14 & 4.12e-05 & 3.01 & 1.11e-06 & 3.15 & 6.35e-05 & 3.02 \\  
			& $32\times32\times32$ & 9.83e-08 & 3.07 & 5.08e-06 & 3.02 & 1.32e-07 & 3.07 & 7.81e-06 & 3.02 \\  
			\hline
			\multicolumn{2}{c|}{} & \multicolumn{4}{c|}{$\bdp{E_{tot}}_h$} & \multicolumn{4}{c}{$\phi_h$} \\  
			\hline
			Element & Mesh & $L^1$ error & order & $L^\infty$ error & order & $L^1$ error & order & $L^\infty$ error & order\\  
			\hline
			\multirow{4}{*}{$P^2$}   
			& $4\times4\times4$   & 9.89e-05 & --   & 3.77e-03 & --   & 7.50e-05 & --   & 5.08e-03 & --   \\  
			& $8\times8\times8$   & 1.08e-05 & 3.19 & 4.72e-04 & 3.00 & 9.12e-06 & 3.04 & 6.80e-04 & 2.90 \\  
			& $16\times16\times16$ & 1.30e-06 & 3.06 & 5.78e-05 & 3.03 & 1.13e-06 & 3.01 & 8.74e-05 & 2.96 \\  
			& $32\times32\times32$ & 1.59e-07 & 3.03 & 7.19e-06 & 3.01 & 1.40e-07 & 3.01 & 1.11e-05 & 2.98 \\  
			\hline
		\end{tabular}
	\end{table}
		
	\begin{table}[htbp]
		\centering
		\caption{ Example \ref{acctest3D}. Numerical errors and convergence orders of different variables by the structure-preserving LDG scheme with OE technique. T = 0.3. }
		\label{SPLDG:OEacctest:3D}
		\begin{tabular}{c|c|c c|c c|c c|c c}
			\hline
			\multicolumn{2}{c|}{} & \multicolumn{4}{c|}{$\rho$} & \multicolumn{4}{c}{$(\rho u)_h$} \\ 
			\hline
			Element & Mesh & $L^1$ error & order & $L^\infty$ error & order & $L^1$ error & order & $L^\infty$ error & order\\
			\hline
			\multirow{4}{*}{$P^2$} 
			& $4\times4\times4$   & 1.36e-03 & --   & 1.40e-02 & --   & 8.35e-04 & --   & 9.56e-03 & --   \\  
			& $8\times8\times8$   & 3.11e-04 & 2.13 & 6.89e-03 & 1.02 & 2.19e-04 & 1.93 & 6.47e-03 & 0.56 \\  
			& $16\times16\times16$ & 1.81e-05 & 4.10 & 9.74e-04 & 2.82 & 1.62e-05 & 3.76 & 1.52e-03 & 2.09 \\  
			& $32\times32\times32$ & 3.78e-07 & 5.58 & 8.04e-06 & 6.92 & 2.77e-07 & 5.87 & 7.42e-06 & 7.68 \\  
			& $64\times64\times64$ & 2.48e-08 & 3.93 & 1.08e-06 & 2.90 & 1.60e-08 & 4.12 & 4.05e-07 & 4.20 \\ 
			\hline
			\multicolumn{2}{c|}{} & \multicolumn{4}{c|}{$(\rho v)_h$} & \multicolumn{4}{c}{$(\rho w)_h$} \\  
			\hline
			Element & Mesh & $L^1$ error & order & $L^\infty$ error & order & $L^1$ error & order & $L^\infty$ error & order\\  
			\hline
			\multirow{4}{*}{$P^2$}   
			& $4\times4\times4$   & 1.12e-03 & --   & 1.56e-02 & --   & 1.82e-03 & --   & 2.98e-02 & --   \\  
			& $8\times8\times8$   & 3.13e-04 & 1.84 & 9.37e-03 & 0.74 & 5.30e-04 & 1.78 & 1.66e-02 & 0.84 \\  
			& $16\times16\times16$ & 2.45e-05 & 3.67 & 2.52e-03 & 1.90 & 4.24e-05 & 3.64 & 4.67e-03 & 1.83 \\  
			& $32\times32\times32$ & 4.23e-07 & 5.86 & 1.37e-05 & 7.53 & 7.04e-07 & 5.91 & 2.51e-05 & 7.54 \\ 
			& $64\times64\times64$ & 2.13e-08 & 4.31 & 4.86e-07 & 4.81 & 3.32e-08 & 4.41 & 6.69e-07 & 5.23 \\  
			\hline
			\multicolumn{2}{c|}{} & \multicolumn{4}{c|}{$\bdp{E_{tot}}_h$} & \multicolumn{4}{c}{$\phi_h$} \\  
			\hline
			Element & Mesh & $L^1$ error & order & $L^\infty$ error & order & $L^1$ error & order & $L^\infty$ error & order\\  
			\hline
			\multirow{4}{*}{$P^2$}   
			& $4\times4\times4$ & 1.08e-03 & --   & 2.22e-02 & --   & 8.55e-05 & --   & 5.15e-03 & --   \\  
			& $8\times8\times8$   & 3.14e-04 & 1.78 & 1.34e-02 & 0.72 & 1.14e-05 & 2.91 & 6.79e-04 & 2.92 \\  
			& $16\times16\times16$ & 3.03e-05 & 3.37 & 3.38e-03& 1.99 & 1.32e-06 & 3.11 & 8.73e-05 & 2.96 \\  
			& $32\times32\times32$ & 5.28e-07 & 5.84 & 2.26e-05 & 3.91 & 1.42e-07 & 3.22 & 1.11e-05 & 2.98 \\  
			& $64\times64\times64$ & 3.09e-08 & 4.09 & 1.11e-06 & 4.34 & 1.76e-08 & 3.01 & 1.39e-06 & 2.99 \\ 
			\hline
		\end{tabular}
	\end{table}
	}
\end{exa}

\begin{exa} {\em
		\label{WB3D}
		({\bf{Well-balanced property}}) In this example, we consider the three-dimensional polytropic problem, which is widely used in astrophysical simulations and has been studied in \cite{maciel2015introduction,Zhang_2022,KAPPELI2014WB}. This problem originates from the Lane–Emden equation, as shown in \eqref{S2:eq:13} and \eqref{S2:eq:14}. For the case $\gamma = 2$, the equilibrium state is given by:
		\begin{equation}
			\rho^e(\mfx) = \rho_0\dfrac{\sin(\alpha r)}{\alpha r}, \quad p^e(\mfx) = K\bdp{\rho_0\dfrac{\sin(\alpha r)}{\alpha r}}^2, \quad \phi^e(\mfx) = -2K\rho_0\dfrac{\sin(\alpha r)}{\alpha r},\label{eqstate3D}
		\end{equation}
		where $r = \sqrt{x^2 + y^2 + z^2}$ is the radial distance to the origin, and
		\begin{equation*}
			\alpha = \sqrt{\dfrac{4\pi G}{2K}}.
		\end{equation*}
		Here, $\rho_0$ is a positive constant. In the following, we set $K = \rho_0 = 1$ and $G = 1/\pi$. It is worth noting that the equilibrium state given in \eqref{eqstate3D} does not strictly satisfy the Poisson equation in two dimensions. In addition, some existing studies \cite{wu2021uniformly,LI2018WB,KAPPELI2014WB,li2018FD_WB} consider a static gravitational field instead of solving the full Poisson equation.
		
		The computation domain is set as $[-0.5,0.5]^3$ with $N^3$ uniform mesh. The initial condition is set as the equilibrium state to test the WB property. We apply the exact boundary condition and the steady state is preserved up to the machine error in Table~\ref{SPLDG:WBtest3D}, which confirms the WB property of our structure-preserving LDG scheme.
		
		\begin{table}[h!]
			\centering
			\caption{Example \ref{WB3D}. Numerical errors and convergence orders of different variables by the structure-preserving LDG scheme. T = 1.0. $P^2$ element.}
			\label{SPLDG:WBtest3D}
			\begin{tabular}{c|c|c c|c|c|c c}
				\hline
				Variables & Mesh & $L^1$ error & $L^\infty$ error & Variables & Mesh & $L^1$ error & $L^\infty$ error \\ 
				\hline 
				\multirow{4}{*}{$\rho_h$} & $4\times4\times4$ & 2.33e-16 & 1.55e-15 & \multirow{4}{*}{$(\rho u)_h$} & $4\times4\times4$  & 2.80e-16 & 3.32e-15 \\
				% \cline{2-4} \cline{6-8}
				& $8\times8\times8$  & 1.46e-16 & 1.22e-15 &  & $8\times8\times8$  & 6.67e-16 & 8.04e-15 \\
				%\cline{2-4} \cline{6-8}
				& $16\times16\times16$  & 7.35e-16 & 5.11e-15 &  & $16\times16\times16$  &  1.55e-15 & 2.28e-14 \\
				%\cline{2-4} \cline{6-8}
				& $32\times32\times32$  & 1.70e-15 & 1.21e-14 &  & $32\times32\times32$  & 3.54e-15 & 6.89e-14 \\
				& $64\times64\times64$  & 3.19e-15 & 1.61e-14 &  & $64\times64\times64$  & 7.70e-15 & 1.47e-13 \\
				\hline
				\multirow{4}{*}{$(\rho v)_h$} & $4\times4\times4$  & 2.77e-16 & 2.22e-15 & \multirow{4}{*}{$(\rho w)_h$} &$4\times4\times4$  & 2.71e-16 & 2.18e-15 \\
				%\cline{2-4} \cline{6-8}
				& $8\times8\times8$  &  6.89e-16 & 7.85e-15 &  & $8\times8\times8$  & 7.04e-16 & 6.24e-15 \\
				%\cline{2-4} \cline{6-8}
				& $16\times16\times16$  & 1.59e-15 & 2.23e-14 &  & $16\times16\times16$  & 1.63e-15 & 1.59e-14 \\
				%\cline{2-4} \cline{6-8}
				& $32\times32\times32$ & 3.63e-15 & 6.21e-14 &  & $32\times32\times32$  & 3.70e-15 & 3.49e-14 \\
				& $64\times64\times64$  & 7.88e-15 & 1.52e-13 &  & $64\times64\times64$  & 7.96e-15 & 9.24e-14 \\
				\hline
				\multirow{4}{*}{$E_h$} & $4\times4\times4$  & 1.79e-16 & 9.99e-16 & \multirow{4}{*}{$\phi_h$} &$4\times4\times4$  & 0.00e+00 & 0.00e+00 \\
				%\cline{2-4} \cline{6-8}
				& $8\times8\times8$  & 1.64e-16 & 1.11e-15 &  & $8\times8\times8$  & 0.00e+00 & 0.00e+00 \\
				%\cline{2-4} \cline{6-8}
				& $16\times16\times16$  & 2.25e-16 & 2.44e-15 &  & $16\times16\times16$  & 6.41e-17 & 2.22e-16 \\
				%\cline{2-4} \cline{6-8}
				& $32\times32\times32$ & 3.57e-16 & 2.89e-15 &  & $32\times32\times32$  & 1.71e-16 & 6.66e-16 \\
				& $64\times64\times64$  & 5.12e-16 & 3.89e-15 &  & $64\times64\times64$  & 3.35e-16 & 1.11e-15 \\
				\hline
			\end{tabular}
		\end{table}
		
	}
\end{exa}

\begin{exa} {\em
		\label{WB_delta3D}
		({\bf{Pertubation to equilibrium state}}) In this example, we consider a small perturbation applied to the equilibrium state \eqref{eqstate3D}. The initial pressure is set as:
		\begin{equation*}
			p = p^e + \mu e^{-100(x^2 + y^2 + z^2)},
		\end{equation*}
		where $\mu = 10^{-3}$. The parameter $K$, $\rho_0$, and $G$ are the same as those in Example~\ref{WB3D}.
		\begin{figure}[htbp]
			\centering
			\begin{minipage}[b]{0.45\textwidth}
				\centering
				\includegraphics[scale=0.4]{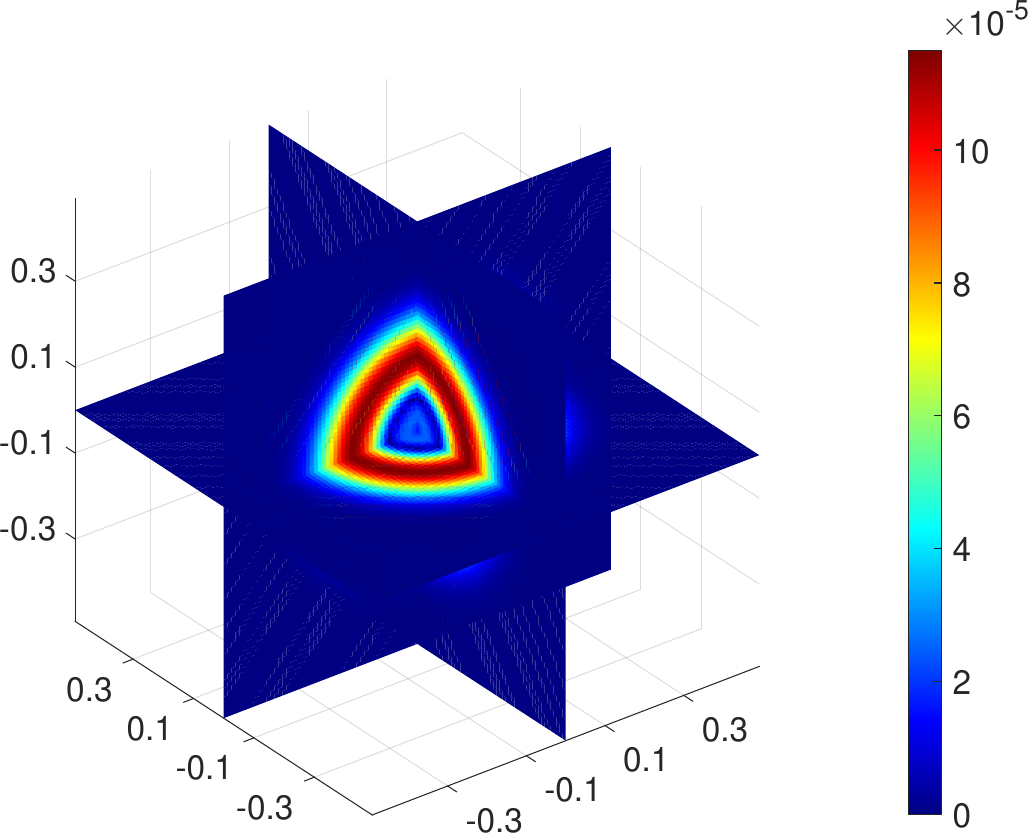}
				\caption*{(a) structure-preserving LDG scheme}
			\end{minipage}
			\hspace{0.05\textwidth} % 调整间距
			\begin{minipage}[b]{0.45\textwidth}
			\centering
			\includegraphics[scale=0.4]{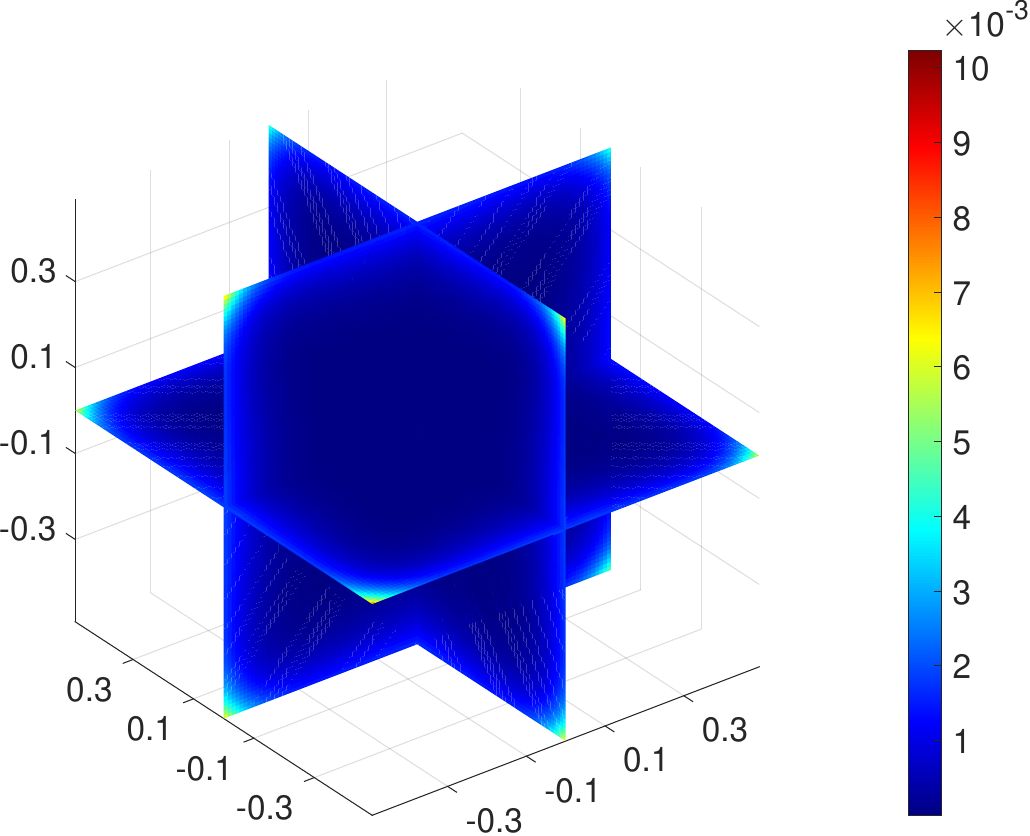}
			\caption*{(b) standard LDG scheme}
			\end{minipage}
			\caption{Example \ref{WB_delta3D}: The contour plots of the velocity magnitude by structure-preserving LDG scheme (a) and standard LDG scheme (b) at time $T = 0.1$ on a $60^3$ uniform mesh.}
			\label{fig:delta3D}
		\end{figure}
		The numerical results are illustrated in Figure~ \ref{fig:delta3D}. We can observe that our structure-preserving LDG scheme can hold the radial symmetry, while the standard LDG scheme can not.

	}
\end{exa}

\begin{exa} {\em
		\label{explosion3D}
		\begin{figure}[htbp]
			\centering
			\begin{minipage}[b]{0.45\textwidth}
				\centering
				\includegraphics[scale=0.4]{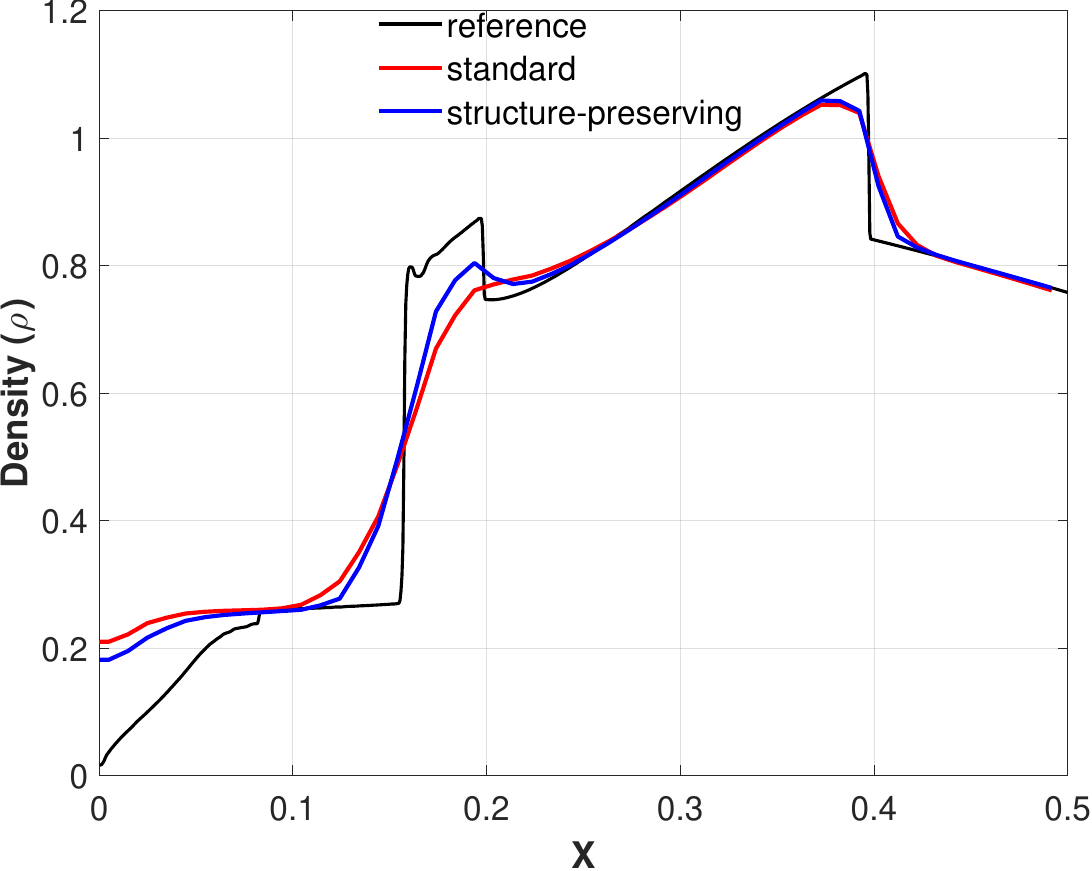}
				\caption*{(a) density $\rho$}
			\end{minipage}
			\hspace{0.05\textwidth} % 调整间距
			\begin{minipage}[b]{0.45\textwidth}
				\centering
				\includegraphics[scale=0.4]{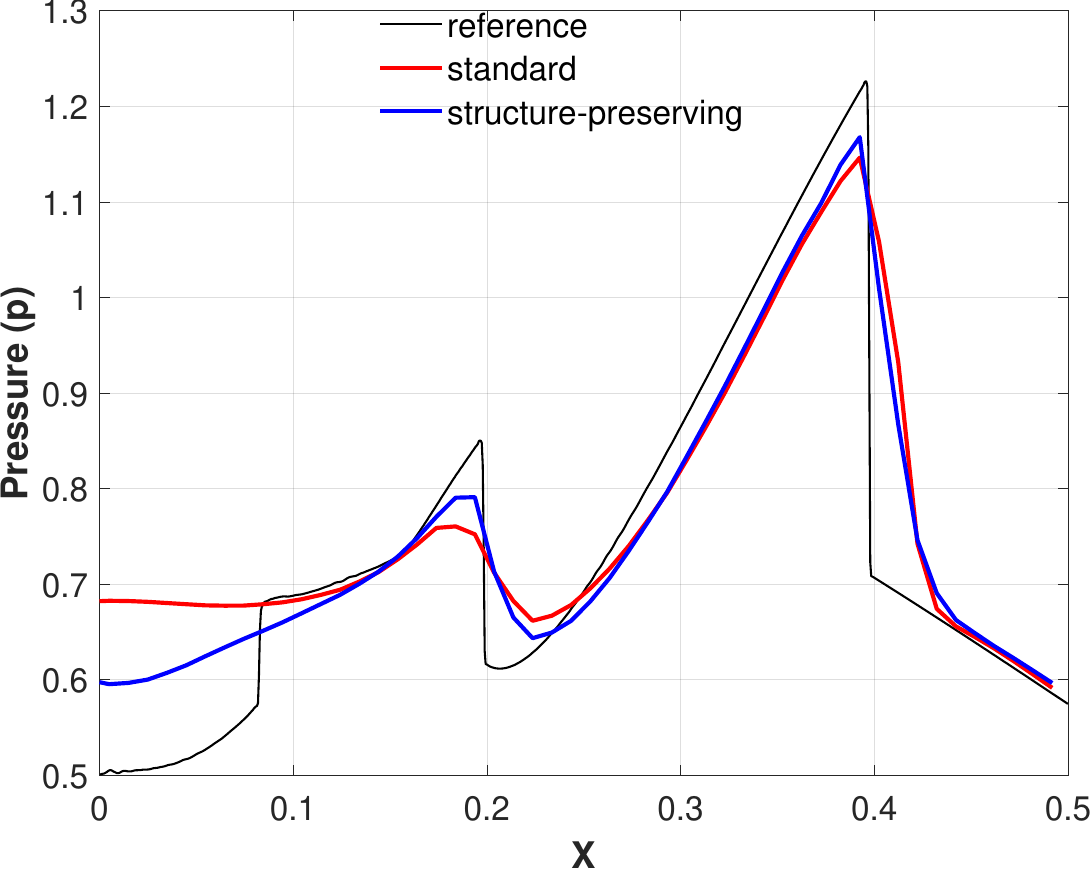}
				\caption*{(b) pressure $p$}
			\end{minipage}
			\caption{Example \ref{explosion3D}: The contour plots of the density $\rho$ (a) and the pressure $p$ (b) of the explosion by both standard and structure-preserving LDG scheme at time $T = 0.15$ on a $60^3$ uniform mesh along $y=z=0$.}
			\label{fig:explotion3D}
		\end{figure}
		\begin{figure}[htbp]
			\centering
			\begin{minipage}[b]{0.45\textwidth}
				\centering
				\includegraphics[scale=0.4]{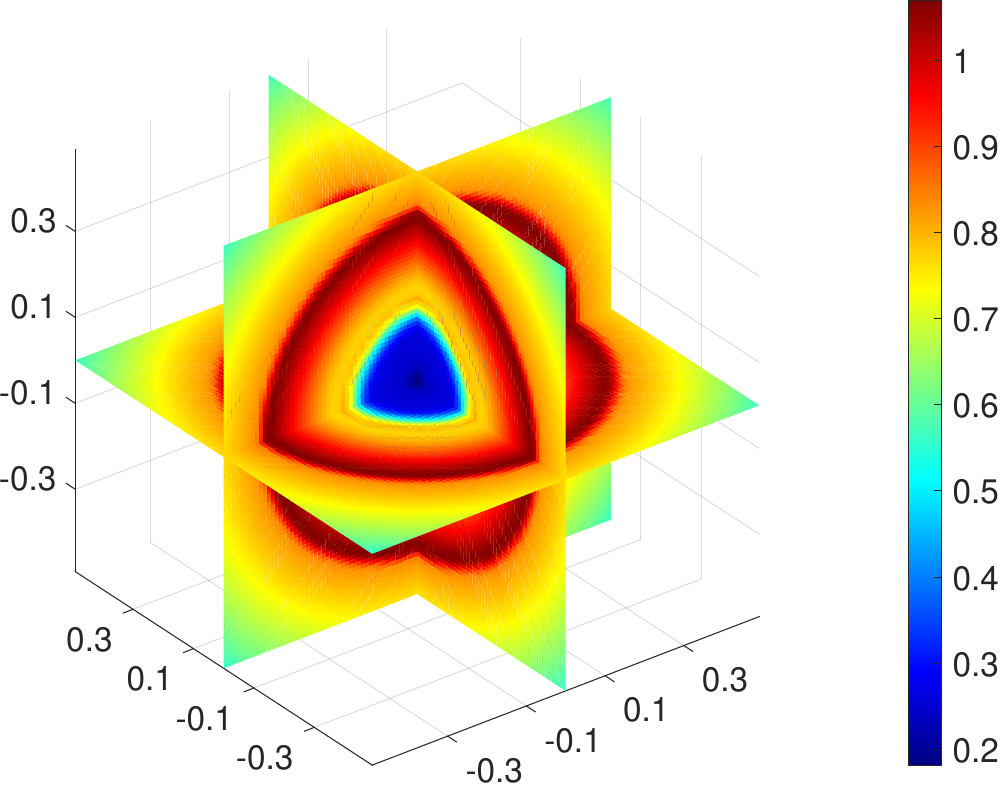}
				\caption*{(a) density $\rho$}
			\end{minipage}
			\hspace{0.05\textwidth} % 调整间距
			\begin{minipage}[b]{0.45\textwidth}
				\centering
				\includegraphics[scale=0.4]{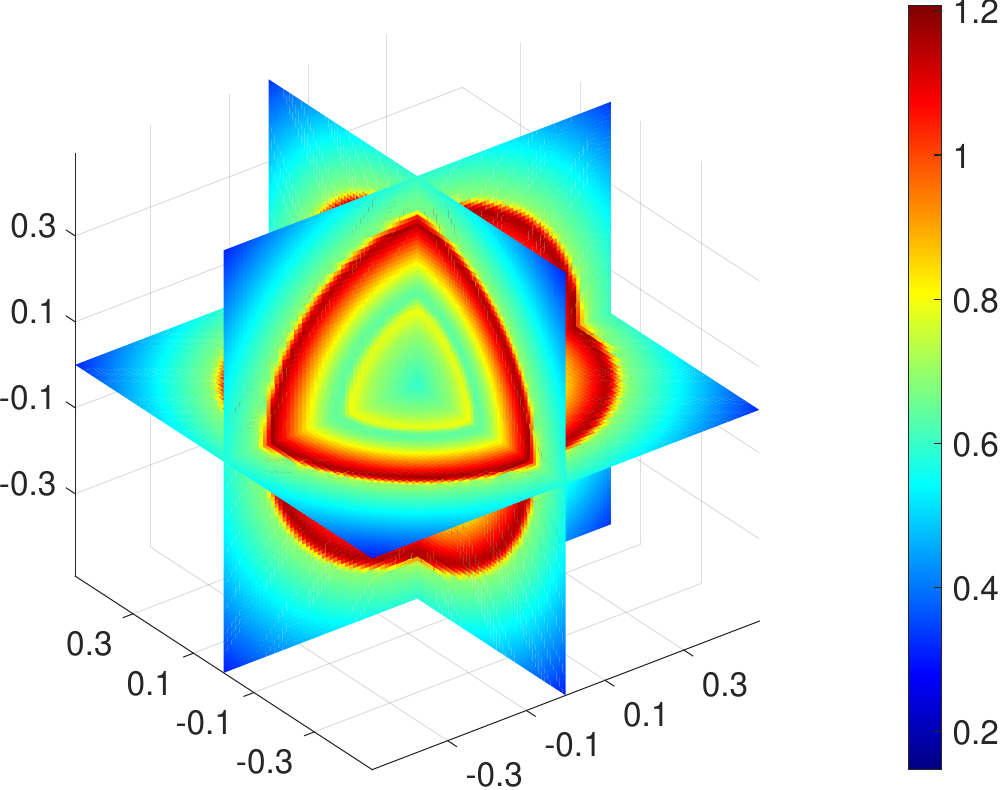}
				\caption*{(b) pressure $p$}
			\end{minipage}
			\caption{Example \ref{explosion3D}: The contour plots of the density $\rho$ (a) and the pressure $p$ (b) of the explosion by structure-preserving LDG scheme at time $T = 0.15$ on a $60^3$ uniform mesh.}
			\label{fig:explotion3D_slice}
		\end{figure}
		({\bf{Explosion}}) Furthermore, we consider an explosion at the center. We apply a jump around the center, similar in \cite{Zhang_2022,KAPPELI2014WB} to simulate the explosion:
		\begin{equation*}
			p_0(r) = \left\{
			\begin{aligned}
				&\alpha p^e(r), &{\rm for}& \; r < r_0, \\
				&p^e(r), &{\rm for}& \; r \geq r_0,
			\end{aligned} \right.
		\end{equation*}
		where we set $\alpha = 10$ and $r_0 = 0.1$.
		In this test, both OE technique and PP limiter is applied. In this numerical test, we set the boundary condition as transmissive boundary conditions. Besides, we set $K = \rho_0 = 1$ and $G = 1$.
		The numerical results on a $60^3$ uniform mesh at $T = 0.15$ are shown in Figure~\ref{fig:explotion3D} and Figure~\ref{fig:explotion3D_slice}. The reference solution is obtained by the solver in \cite{Zhang_2022} for SG Euler equations in spherical symmetry on a $800$ uniform mesh. The radial symmetric structure of density and pressure is preserved equally well, but our structure-preserving LDG scheme exhibits higher resolution in capturing shocks than the standard scheme.

	}
\end{exa}

\vspace{0.1cm}
%%%%%%%%%%%%%%%%%%%%%%%%%%%%%
%%%%%%%%%%%%%%%%%%%%%%%%%%%%%
%%%%%%%%%%%%%%%%%%%%%%%%%%%%%
%%%%%%%%%%%%%%%%%%%%%%%%%%%%%
%%%%%%%%%%%%%%%%%%%%%%%%%%%%%
%%%%%%%%%%%%%%%%%%%%%%%%%%%%%

\section{Conclusion}
\label{sec6}
\setcounter{equation}{0}
\setcounter{figure}{0}
\setcounter{table}{0}

In this paper, a high-order well-balanced and total-energy-conserving local discontinuous Galerkin scheme has been developed for the compressible self-gravitating Euler equations. High order accuracy in time and in space is obtained by strong-stability-preserving  Runge-Kutta methods and local discontinuous Galerkin framework, respectively. To address the difficulty from the time-dependent source term, we have split the gravitational potential to guarantee a correct preservation of steady states. With the energy conservative form, the total energy can be preserved and the scheme can easily extend to high order time discretizations. We have proven that our proposed scheme has well-balanced and total-energy-conserving properties. Extensive numerical results were provided to demonstrate the high order accuracy, shock-capturing ability, well-balanced property, and total-energy-conserving property. 

For future work, we aim to extend the current framework to unstructured meshes, enabling the simulation of more complex astrophysical phenomena. In addition, we plan to investigate the positivity-preserving property of the scheme to further enhance its robustness in the presence of strong discontinuities.

%\section*{Declaration}
%\subsection*{Credit}
%iang Pan: Conceptualization, Formal analysis, Investigation, Methodology, Software, Writing-original draft, Writing-review \& editing.
%Wei Chen: Conceptualization, Formal analysis, Methodology, Writing-review \& editing.
%Jianxian Qiu: Conceptualization, Formal analysis, Methodology, Writing-review \& editing.
%Tao Xiong: Conceptualization, Formal analysis, Methodology, Writing-review \& editing, Funding acquisition.

%\subsection*{Code and Data availability}
%The codes generated in this study are available from the corresponding author
%on reasonable request. No data sets were generated or analyzed during the current study.

%\subsection*{Conflict of interest}
%We declare that we have no financial and personal relationships with other people or organizations that can inappropriately influence our work. 

\section*{Acknowledgments}
We would like to thank Dr. Yulong Xing and Weijie Zhang for their helpful discussions and providing part of the data used in this work.

\bibliographystyle{abbrv}
\bibliography{refer}

\newcommand{\noop}[1]{}
\begin{thebibliography}{10}

\bibitem{AudusseSWE2004}
E.~Audusse, F.~Bouchut, M.-O. Bristeau, R.~Klein, and B.~Perthame.
\newblock A fast and stable well-balanced scheme with hydrostatic
  reconstruction for shallow water flows.
\newblock {\em SIAM Journal on Scientific Computing}, 25(6):2050--2065, 2004.

\bibitem{Batten1997multidimensionalHLLC}
P.~Batten, N.~Clarke, C.~Lambert, and D.~M. Causon.
\newblock {On the choice of wavespeeds for the HLLC Riemann solver}.
\newblock {\em SIAM Journal on Scientific Computing}, 18(6):1553--1570, 1997.

\bibitem{Berberich2021WBFV}
J.~P. Berberich, P.~Chandrashekar, and C.~Klingenberg.
\newblock {High order well-balanced finite volume methods for multi-dimensional
  systems of hyperbolic balance laws}.
\newblock {\em Computers $\&$ Fluids}, 219:104858, 2021.

\bibitem{Bremudez19941049}
A.~Bermudez and M.~E. Vazquez.
\newblock {Upwind methods for hyperbolic conservation laws with source terms}.
\newblock {\em Computers $\&$ Fluids}, 23(8):1049--1071, 1994.

\bibitem{BOTTA2004WB}
N.~Botta, R.~Klein, S.~Langenberg, and S.~Lützenkirchen.
\newblock Well balanced finite volume methods for nearly hydrostatic flows.
\newblock {\em Journal of Computational Physics}, 196(2):539--565, 2004.

\bibitem{Brenier2000VPtoEuler}
Y.~Brenier.
\newblock {Convergence of the Vlasov-Poisson system to the incompressible Euler
  equations}.
\newblock {\em Communications in Partial Differential Equations},
  25(3-4):737--754, 2000.

\bibitem{Castillo2002performance}
P.~Castillo.
\newblock Performance of discontinuous {G}alerkin methods for elliptic {PDE}s.
\newblock {\em SIAM Journal on Scientific Computing}, 24(2):524--547, 2002.

\bibitem{CASTILLO20061307}
P.~Castillo.
\newblock {A review of the local discontinuous Galerkin (LDG) method applied to
  elliptic problems}.
\newblock {\em Applied Numerical Mathematics}, 56(10):1307--1313, 2006.
\newblock Selected Papers from the First Chilean Workshop on Numerical Analysis
  of Partial Differential Equations (WONAPDE 2004).

\bibitem{Priori_analysis_LDG_elliptic}
P.~Castillo, B.~Cockburn, I.~Perugia, and D.~Sch\"{o}tzau.
\newblock An a priori error analysis of the local discontinuous {G}alerkin
  method for elliptic problems.
\newblock {\em SIAM Journal on Numerical Analysis}, 38(5):1676--1706, 2000.

\bibitem{2015WB2rd}
P.~Chandrashekar and C.~Klingenberg.
\newblock {A second order well-balanced finite volume scheme for Euler
  equations with gravity}.
\newblock {\em SIAM Journal on Scientific Computing}, 37(3):B382--B402, 2015.

\bibitem{CHEN2011EC_VA}
G.~Chen, L.~Chacón, and D.~Barnes.
\newblock An energy- and charge-conserving, implicit, electrostatic
  particle-in-cell algorithm.
\newblock {\em Journal of Computational Physics}, 230(18):7018--7036, 2011.

\bibitem{CHENG2014630}
Y.~Cheng, A.~J. Christlieb, and X.~Zhong.
\newblock {Energy-conserving discontinuous Galerkin methods for the
  Vlasov-Ampère system}.
\newblock {\em Journal of Computational Physics}, 256:630--655, 2014.

\bibitem{Cheng2014EC_DG_VM}
Y.~Cheng, A.~J. Christlieb, and X.~Zhong.
\newblock {Energy-conserving discontinuous Galerkin methods for the
  Vlasov-Maxwell system}.
\newblock {\em Journal of Computational Physics}, 279:145--173, 2014.

\bibitem{Superconvergence_LDG_elliptic}
B.~Cockburn, G.~Kanschat, I.~Perugia, and D.~Sch\"{o}tzau.
\newblock {Superconvergence of the local discontinuous Galerkin method for
  elliptic problems on cartesian grids}.
\newblock {\em SIAM Journal on Numerical Analysis}, 39(1):264--285, 2001.

\bibitem{2013MultipoleExpansion}
S.~M. Couch, C.~Graziani, and N.~Flocke.
\newblock {An improved multipole approximation for self-gravity and its
  importance for core-collapse supernova simulations}.
\newblock {\em The Astrophysical Journal}, 778(2):181, nov 2013.

\bibitem{Nicolas2016BGKVPlimits}
N.~Crouseilles, G.~Dimarco, and M.-H. Vignal.
\newblock {Multiscale schemes for the BGK--Vlasov--Poisson System in the
  quasi-neutral and fluid limits. Stability analysis and first order schemes}.
\newblock {\em Multiscale Modeling \& Simulation}, 14(1):65--95, 2016.

\bibitem{Dios3dimECDGforVP}
B.~A. de~Dios, J.~A. Carrillo, and C.-W. Shu.
\newblock {Discontinuous Galerkin methods for the multi-dimensional
  Vlasov-Poisson problem}.
\newblock {\em Mathematical Models and Methods in Applied Sciences}, 22:1250042
  [45 pages], 12 2012.

\bibitem{Dios2012HIGHOA}
B.~A. de~Dios and S.~Hajian.
\newblock {High order and energy preserving discontinuous Galerkin methods for
  the Vlasvo-Poisson system}.
\newblock {\em arXiv: Numerical Analysis}, 2012.

\bibitem{Dimarco2014APVPBtoPlasma}
G.~Dimarco, L.~Mieussens, and V.~Rispoli.
\newblock {An asymptotic preserving automatic domain decomposition method for
  the Vlasov-Poisson-BGK system with applications to plasmas}.
\newblock {\em Journal of Computational Physics}, 274:122--139, 2014.

\bibitem{DU2024WBPP}
J.~Du, Y.~Yang, and F.~Zhu.
\newblock {Well-balanced positivity-preserving high-order discontinuous
  Galerkin methods for Euler equations with gravitation}.
\newblock {\em Journal of Computational Physics}, 505:112877, 2024.

\bibitem{Filbet2022EC_DG_VP}
F.~Filbet and T.~Xiong.
\newblock {Conservative discontinuous Galerkin/Hermite spectral method for the
  Vlasov-Poisson system}.
\newblock {\em Communications on Applied Mathematics and Computation},
  4(1):34--59, 2022.

\bibitem{Greenberg1996}
J.~M. Greenberg and A.~Y. Leroux.
\newblock {A Well-Balanced Scheme for the Numerical Processing of Source Terms
  in Hyperbolic Equations}.
\newblock {\em SIAM Journal on Numerical Analysis}, 33(1):1--16, 1996.

\bibitem{Käppeli2019HWB}
L.~Grosheintz-Laval and R.~Käppeli.
\newblock {High-order well-balanced finite volume schemes for the Euler
  equations with gravitation}.
\newblock {\em Journal of Computational Physics}, 378:324--343, 2019.

\bibitem{2019ComparisonWBandNonwb}
M.~Han~Veiga, D.~A. Romero~Velasco, R.~Abgrall, and R.~Teyssier.
\newblock {Capturing near-equilibrium solutions: a comparison between
  high-order discontinuous Galerkin Methods and well-balanced schemes}.
\newblock {\em Communications in Computational Physics}, 26(1):1--34, 2019.

\bibitem{2006HubberSGSPH}
D.~Hubber, S.~P. Goodwin, and A.~Whitworth.
\newblock {Resolution requirements for simulating gravitational fragmentation
  using SPH}.
\newblock {\em Astronomy and Astrophysics}, 450:881--886, 2006.

\bibitem{1902JeansInstability}
J.~H. Jeans.
\newblock {The Stability of a Spherical Nebula}.
\newblock {\em Philosophical Transactions of the Royal Society of London.
  Series A, Containing Papers of a Mathematical or Physical Character},
  199:1--53, 1902.

\bibitem{JIANG201348}
Y.-F. Jiang, M.~Belyaev, J.~Goodman, and J.~M. Stone.
\newblock {A new way to conserve total energy for Eulerian hydrodynamic
  simulations with self-gravity}.
\newblock {\em New Astronomy}, 19:48--55, 2013.

\bibitem{KAPPELI2014WB}
R.~Käppeli and S.~Mishra.
\newblock {Well-balanced schemes for the Euler equations with gravitation}.
\newblock {\em Journal of Computational Physics}, 259:199--219, 2014.

\bibitem{LEVEQUE1998346}
R.~J. LeVeque.
\newblock {Balancing source terms and flux gradients in high-resolution Godunov
  methods: the quasi-steady wave-propagation algorithm}.
\newblock {\em Journal of Computational Physics}, 146(1):346--365, 1998.

\bibitem{Li2015WellBalancedDG}
G.~Li and Y.~Xing.
\newblock {Well-balanced discontinuous Galerkin methods for the Euler equations
  under gravitational fields}.
\newblock {\em Journal of Scientific Computing}, 67:493--513, 2016.

\bibitem{LI2018WB}
G.~Li and Y.~Xing.
\newblock {Well-balanced discontinuous Galerkin methods with hydrostatic
  reconstruction for the Euler equations with gravitation}.
\newblock {\em Journal of Computational Physics}, 352:445--462, 2018.

\bibitem{li2018FD_WB}
G.~Li and Y.~Xing.
\newblock {Well-balanced finite difference weighted essentially non-oscillatory
  schemes for the Euler equations with static gravitational fields}.
\newblock {\em Computers $\&$ Mathematics with Applications}, 75(6):2071--2085,
  2018.
\newblock 2nd Annual Meeting of SIAM Central States Section, September
  30-October 2, 2016.

\bibitem{LIU2023EC_SL_VA}
H.~Liu, X.~Cai, Y.~Cao, and G.~Lapenta.
\newblock {An efficient energy conserving semi-Lagrangian kinetic scheme for
  the Vlasov-Ampère system}.
\newblock {\em Journal of Computational Physics}, 492:112412, 2023.

\bibitem{Liu2015}
H.~Liu and N.~Ploymaklam.
\newblock {A local discontinuous Galerkin method for the Burgers--Poisson
  equation}.
\newblock {\em Numerische Mathematik}, 129(2):321--351, Feb. 2015.

\bibitem{Liu2010DriftDiffusion}
Y.~Liu and C.-W. Shu.
\newblock {Error analysis of the semi-discrete local discontinuous Galerkin
  method for semiconductor device simulation models}.
\newblock {\em Science China Mathematics}, 53(12):3255--3278, 2010.

\bibitem{Liu2016DriftDiffusionIMEX}
Y.~Liu and C.-W. Shu.
\newblock {Analysis of the local discontinuous Galerkin method for the
  drift-diffusion model of semiconductor devices}.
\newblock {\em Science China Mathematics}, 59(1):115--140, 2016.

\bibitem{maciel2015introduction}
W.~J. Maciel.
\newblock {\em Introduction to stellar structure}.
\newblock Springer, 2015.

\bibitem{MARKIDIS2011EC_PIC_VM}
S.~Markidis and G.~Lapenta.
\newblock The energy conserving particle-in-cell method.
\newblock {\em Journal of Computational Physics}, 230(18):7037--7052, 2011.

\bibitem{2023SPFEforEulerPoisson}
M.~Matthias, J.~Shadid, N., and I.~Tomas.
\newblock {Structure-preserving finite-element schemes for the Euler-Poisson
  equations}.
\newblock {\em Communications in Computational Physics}, 33(3):647--691, 2023.

\bibitem{mullen2021extension}
P.~Mullen, T.~Hanawa, and C.~Gammie.
\newblock {An extension of the Athena++ framework for fully conservative
  self-gravitating hydrodynamics}.
\newblock {\em Astrophysical Journal, Supplement Series}, 252(2), Feb. 2021.

\bibitem{1995MullerSGEuler}
E.~Müller and M.~Steinmetz.
\newblock Simulating self-gravitating hydrodynamic flows.
\newblock {\em Computer Physics Communications}, 89(1):45--58, 1995.
\newblock Numerical Methods in Astrophysical Hydrodynamics.

\bibitem{2024PengOEDG}
M.~Peng, Z.~Sun, and K.~Wu.
\newblock O{EDG}: oscillation-eliminating discontinuous {G}alerkin method for
  hyperbolic conservation laws.
\newblock {\em Math. Comp.}, 94(353):1147--1198, 2025.

\bibitem{REN2023FVPP}
Y.~Ren, K.~Wu, J.~Qiu, and Y.~Xing.
\newblock {On high order positivity-preserving well-balanced finite volume
  methods for the Euler equations with gravitation}.
\newblock {\em Journal of Computational Physics}, 492:112429, 2023.

\bibitem{PurelyDG_SG2021}
M.~Schlottke-Lakemper, A.~R. Winters, H.~Ranocha, and G.~J. Gassner.
\newblock {A purely hyperbolic discontinuous Galerkin approach for
  self-gravitating gas dynamics}.
\newblock {\em Journal of Computational Physics}, 442:110467, 2021.

\bibitem{Shu1988TVD_RK}
C.-W. Shu.
\newblock {Total-variation-diminishing time discretizations}.
\newblock {\em SIAM Journal on Scientific and Statistical Computing},
  9(6):1073--1084, 1988.

\bibitem{Simon_2016}
J.~B. Simon, P.~J. Armitage, R.~Li, and A.~N. Youdin.
\newblock {The mass and size distribution of planetesimals formed by the
  streaming instability. I. The role of self-gravity}.
\newblock {\em The Astrophysical Journal}, 822(1):55, may 2016.

\bibitem{SecondFVPPEulerKlingenberg}
A.~Thomann, M.~Zenk, and C.~Klingenberg.
\newblock {A second-order positivity-preserving well-balanced finite volume
  scheme for Euler equations with gravity for arbitrary hydrostatic
  equilibria}.
\newblock {\em International Journal for Numerical Methods in Fluids},
  89(11):465--482, 2019.

\bibitem{Toro2009RiemannSolver}
E.~Toro.
\newblock {\em {Riemann Solvers and Numerical Methods for Fluid Dynamics: A
  Practical Introduction}}.
\newblock Springer-Verlag, Berlin, Heidelberg, 2013.

\bibitem{Varma2019WB2nd}
D.~Varma and P.~Chandrashekar.
\newblock {A second-order, discretely well-balanced finite volume scheme for
  Euler equations with gravity}.
\newblock {\em Computers $\&$ Fluids}, 181:292--313, 2019.

\bibitem{2012PPTwodimension}
C.~Wang, X.~Zhang, C.-W. Shu, and J.~Ning.
\newblock {Robust high order discontinuous Galerkin schemes for two-dimensional
  gaseous detonations}.
\newblock {\em Journal of Computational Physics}, 231(2):653--665, 2012.

\bibitem{wu2019provably}
K.~Wu and C.-W. Shu.
\newblock {Provably positive high-order schemes for ideal magnetohydrodynamics:
  analysis on general meshes}.
\newblock {\em Numerische Mathematik}, 142(4):995--1047, 2019.

\bibitem{wu2021uniformly}
K.~Wu and Y.~Xing.
\newblock {Uniformly high-order structure-preserving discontinuous Galerkin
  methods for Euler equations with gravitation: positivity and
  well-balancedness}.
\newblock {\em SIAM Journal on Scientific Computing}, 43(1):A472--A510, 2021.

\bibitem{Xing2014WbDgMovingWater}
Y.~Xing.
\newblock {Exactly well-balanced discontinuous Galerkin methods for the shallow
  water equations with moving water equilibrium}.
\newblock {\em Journal of Computational Physics}, 257:536--553, 2014.

\bibitem{XING2005FdWbSwe}
Y.~Xing and C.-W. Shu.
\newblock {High order finite difference WENO schemes with the exact
  conservation property for the shallow water equations}.
\newblock {\em Journal of Computational Physics}, 208(1):206--227, 2005.

\bibitem{Xing2012HighOW}
Y.~Xing and C.-W. Shu.
\newblock {High order well-balanced WENO scheme for the gas dynamics equations
  under gravitational fields}.
\newblock {\em Journal of Scientific Computing}, 54:645 -- 662, 2012.

\bibitem{Xing2010PpWbDgSwe}
Y.~Xing, X.~Zhang, and C.-W. Shu.
\newblock {Positivity-preserving high order well-balanced discontinuous
  Galerkin methods for the shallow water equations}.
\newblock {\em Advances in Water Resources}, 33(12):1476--1493, 2010.

\bibitem{ECforVA2024}
B.~Ye, J.~Hu, C.-W. Shu, and X.~Zhong.
\newblock {Energy-conserving discontinuous Galerkin methods for the
  Vlasov-Ampère system with Dougherty-Fokker-Planck collision operator}.
\newblock {\em Journal of Computational Physics}, 514:113219, 2024.

\bibitem{YIN2023EV_Moment_VW}
T.~Yin, X.~Zhong, and Y.~Wang.
\newblock {Highly efficient energy-conserving moment method for the
  multi-dimensional Vlasov-Maxwell system}.
\newblock {\em Journal of Computational Physics}, 475:111863, 2023.

\bibitem{Zhang_2022}
W.~Zhang, Y.~Xing, and E.~Endeve.
\newblock {Energy conserving and well-balanced discontinuous Galerkin methods
  for the Euler-Poisson equations in spherical symmetry}.
\newblock {\em Monthly Notices of the Royal Astronomical Society},
  514(1):370--389, 05 2022.

\bibitem{Weijie2022WBDG}
W.~Zhang, Y.~Xing, Y.~Xia, and Y.~Xu.
\newblock {High-order positivity-preserving well-balanced discontinuous
  Galerkin methods for Euler equations with gravitation on unstructured
  meshes}.
\newblock {\em Communications in Computational Physics}, 31(3):771--815, 2022.

\bibitem{2020PPlimiterDG}
X.~Zhang and C.-W. Shu.
\newblock On positivity-preserving high order discontinuous {G}alerkin schemes
  for compressible {E}uler equations on rectangular meshes.
\newblock {\em Journal of Computational Physics}, 229(23):8918--8934, 2010.

\bibitem{MadauleNDG_EC_VP2014}
Éric Madaule, M.~Restelli, and E.~Sonnendrücker.
\newblock {Energy conserving discontinuous Galerkin spectral element method for
  the Vlasov–Poisson system}.
\newblock {\em Journal of Computational Physics}, 279:261--288, 2014.

\end{thebibliography}

\appendix
\section{Proof of Theorem \ref{thm1}}
\label{pro:WB}
\begin{proof}
	We proceed by mathematical induction. Based on initial condition for the explicitly known equilibrium state $ \mfu^0(\mfx)=\mathbf{0}$, $p^0(\mfx)=-\rho^0(\mfx)\nabla \phi^0$, $\Delta  \phi^0(\mfx) = 4\pi G\rho^0(\mfx)$, the initial condition for the structure-preserving LDG scheme is given by \eqref{S3:eq:31} and \eqref{S3:eq:32}. 	
	Suppose that at time $t^n$, the solution is equal to the discrete equilibrium state $\{\rho_h^e,\mfu_h^e = \mathbf{0},p_h^e,\phi_h^e\}$, i.e.
	\begin{equation}
		\rho_h^{n} = \rho_h^e,\; (\rho\mfu)_h^{n} = \mathbf{0},\;
		(E_{tot})_h^{n} = \dfrac{p_h^e}{\gamma-1} + \mfP(\dfrac{1}{2}\rho_h^e\phi_h^e).
	\end{equation}
	We then prove that 
	\begin{equation}
		\begin{split}
			&\rho_h^{n+1} = \rho_h^e,\; (\rho\mfu)_h^{n+1} = \mathbf{0},\; p_h^{n+1} = p_h^e,\;
			\phi_h^{n+1} = \phi_h^e,\\
			&E_h^{n+1} = \dfrac{p_h^e}{\gamma-1},\;
			(E_{tot})_h^{n+1} = \dfrac{p_h^e}{\gamma-1} + \mfP(\dfrac{1}{2}\rho_h^e\phi_h^e).
		\end{split}
	\end{equation}
	
	Following the for the next time level $t^{n+1}$, we have
	\begin{itemize}
		\item 
		First in Step 1, following the decomposition of the gravitational force and gravitational potential $(\mfg_h^n,\phi_h^n)$, we have that $(\bdp{\mfg_h^{\delta}}^n,\bdp{\phi_h^{\delta}}^n) = \mD_1(4\pi G(\rho^n_h - \rho_h^e))$. Since $\rho^n_h = \rho_h^e$, $(\bdp{\mfg_h^{\delta}}^n,\bdp{\phi_h^{\delta}}^n) = (\mathbf{0},0)$, it follows that $\phi_h^{n} = \phi_h^e$. We then have that
		\begin{equation}
			E_h^n = (E_{tot})_h^{n} - \mfP(\dfrac{1}{2}\rho_h^n\phi_h^n) = (E_{tot})_h^{n} - \mfP(\dfrac{1}{2}\rho_h^e\phi_h^e) = \dfrac{p_h^e}{\gamma-1}.
		\end{equation}
		Moreover, 
		\begin{equation}
			p_h^n = (\gamma-1)\bdp{E_h^n - \dfrac{1}{2}\dfrac{((\rho\mfu)_h^n)^2}{\rho_h^n}} = (\gamma-1)E_h^n = p_h^e.
		\end{equation}
        If $E_h^n = (E_{tot})_h^{n} - \dfrac{1}{2}\rho_h^n\phi_h^n$, the pressure $p_h^n\neq p_h^e$ and the WB property is destroyed.
		\item We notice that  $\mfU_h^{n,\text{int}(K)} = (\rho_h^{e,\text{int}(K)},\mathbf{0},p_h^{e,\text{int}(K)}/(\gamma-1))^T,\mfU_h^{n,\text{ext}(K)} = (\rho_h^{e,\text{ext}(K)},\mathbf{0},p_h^{e,\text{ext}(K)}/(\gamma-1))^T$. According to the modified HLLC flux in step 2 and lemma \ref{S2:L1}, we have that
		\begin{equation*}
			\widehat{\mfF}^n_{\mfn_{\mathcal{E}, K}} =\left(\widehat{f}^{n,[1]},\widehat{\bm{f}}^{n,[2],T},\widehat{f}^{n,[3]}\right)_{\mfn_{\mE, K}}^T = \bdp{0,p_h^{e,*}\mfn_{\mE, K}^T,0}^T.
		\end{equation*}
		In other words, $\widehat{f}^{n,[1]}_{\mfn_{\mE, K}} = 0$.
		\item In Step 3, from the Poisson equation for $(\dot{\mfg^n_h},\dot{\phi^n_h}) = \mD_2\bdp{4\pi G\nabla\cdot(\rho\mfu)^n_h} = \mD_2\bdp{\mathbf{0}}$, i.e.
		\begin{equation}
			\begin{split}
				\int_K\dot{\mfg^n_h}\cdot\mfw\dd{\mfx} &= \sum_{\mE\in \partial K}\int_{\mE}\widehat{\dot{\phi^n_h}}\mfw^{\text{int}(K)}\cdot\mfn_{\mE,K}\dd{s} - \int_{K}\dot{\phi^n_h}\nabla\cdot\mfw\dd{\mfx},\\
				\int_K\dot{\mfg^n_h}\cdot\nabla\psi\dd{\mfx} &= \int_{\mE}\psi^{\text{int}(K)}\widehat{\dot{\mfg^n_h}}\cdot\mfn_{\mE,K}\dd{s},
			\end{split}
		\end{equation}
		$(\dot{\mfg^n_h},\dot{\phi^n_h}) = (0,\mathbf{0})$ with the flux satisfying $\widehat{\dot{\phi^n_h}} = 0$, $\widehat{\dot{\mfg^n_h}} = \mathbf{0}$. Hence the energy flux $(\widehat{\mfF_g})^n_{\mfn_{\mathcal{E},K}} = 0$.
		\item Following the flowchart in the fully discrete version in step 4, we only discuss the momentum equation here, as the fluxes of mass equation and energy equation are both zero. For the momentum equation
		\begin{align*}
			&\qquad \sum_{q=1}^{Q}w_K^{(q)}(\rho\mfu)_h^{n+1}(\mfx_K^{(q)})v(\mfx_K^{(q)}) \\
		 	&= \sum_{q=1}^{Q}w_K^{(q)}(\rho\mfu)_h^{n}(\mfx_K^{(q)})v(\mfx_K^{(q)}) + \Delta t \left( \sum_{q=1}^{Q}w_K^{(q)}\mfF^{n,[2]}_h(\mfx_K^{(q)})\cdot\nabla v(\mfx_K^{(q)})\right. \\
			&\qquad  - \left.\sum_{\mE\in \partial K}\sum_{\mu=1}^{N}w_{\mE}^{(\mu)}\widehat{f}^{n,[2]}_{\mfn_{\mE,K}}(\mfx_{\mE}^{(\mu)})v^{\text{int}(K)}(\mfx_\mE^{(\mu)}) + \left<\mfS^{[2],n},v\right>_{K}\right) \\
			& = \sum_{q=1}^{Q}w_K^{(q)}(\rho\mfu)_h^{e}(\mfx_K^{(q)})v(\mfx_K^{(q)}) + \Delta t \left( \sum_{q=1}^{Q}w_K^{(q)}p_h^{e}(\mfx_K^{(q)})\nabla v(\mfx_K^{(q)})\right. \\
			&\qquad  - \left.\sum_{\mE\in \partial K}\sum_{\mu=1}^{N}w_{\mE}^{(\mu)}p_h^{e,*}(\mfx_{\mE}^{(\mu)})\mfn_{\mE, K}v^{\text{int}(K)}(\mfx_\mE^{(\mu)})\right.\\ 
			&\qquad + \sum_{q = 1}^{Q}w_{K}^{(q)} \left(\dfrac{\rho^n_h(\mfx_K^{(q)})}{\rho^e_h(\mfx_K^{(q)})} - \dfrac{\overline{(\rho^n_h)_K}}{\overline{(\rho^e_h)_K}}\right)\nabla p_h^e(\mfx_K^{(q)})v(\mfx_K^{(q)}) \\
			& \qquad + \dfrac{\overline{(\rho^n_h)_K}}{\overline{(\rho^e_h)_K}}\left(\sum_{\mE\in\partial K}\sum_{\mu = 1}^{N}w_{\mE}^{(\mu)}  p^e_h(\mfx_{\mE}^{(\mu)})v^{\text{int}(K)}\mfn_{\mE,K}(\mfx_{\mE}^{(\mu)})  - \sum_{q = 1}^{Q}w_K^{(q)} p_h^e(\mfx_K^{(q)})\nabla v(\mfx_K^{(q)})\right) \\
			&\qquad\left. + \sum_{q = 1}^{Q}-w_K^{(q)}\rho_h(x_K^{(q)})\mfg^{\delta}_h(\mfx_K^{(q)})v(\mfx_K^{(q)}) \right) \\
			& = \sum_{q=1}^{Q}w_K^{(q)}(\rho\mfu)_h^{e}(\mfx_K^{(q)})v(\mfx_K^{(q)}) + \mfI + \mfI\mfI + \mfI\mfI\mfI + \mfI\mfV,
%			&\qquad + \Delta t\left[ \sum_{q=1}^{Q}w_K^{(q)}p_h^{e}(\mfx_K^{(q)})\nabla v(\mfx_K^{(q)}) - \dfrac{\overline{(\rho^n_h)_K}}{\overline{(\rho^e_h)_K}}\sum_{q=1}^{Q}w_K^{(q)}p_h^{e}(\mfx_K^{(q)})\nabla v(\mfx_K^{(q)}) \right. \\
%			&\qquad +  \sum_{q = 1}^{Q}w_{K}^{(q)} \left(\dfrac{\rho^n_h(\mfx_K^{(q)})}{\rho^e_h(\mfx_K^{(q)})} - \dfrac{\overline{(\rho^n_h)_K}}{\overline{(\rho^e_h)_K}}\right)\nabla p_h^e(\mfx_K^{(q)})v(\mfx_K^{(q)}) \\
%			&\qquad - \sum_{\mE\in \partial K}\sum_{\mu=1}^{N}w_{\mE}^{(\mu)}p_h^{e,*}(\mfx_{\mE}^{(\mu)})\mfn_{\mE, K}v^{\text{int}(K)}(\mfx_\mE^{(\mu)}) + \dfrac{\overline{(\rho^n_h)_K}}{\overline{(\rho^e_h)_K}}\sum_{\mE\in \partial K}\sum_{\mu=1}^{N}w_{\mE}^{(\mu)}p_h^{e,*}(\mfx_{\mE}^{(\mu)})\mfn_{\mE, K}v^{\text{int}(K)}(\mfx_\mE^{(\mu)}) \\
%			&\qquad\left. - \sum_{q = 1}^{Q}w_K^{(q)}\rho_h(x_K^{(q)})\mfg^{\delta}_h(\mfx_K^{(q)})v(\mfx_K^{(q)}) \right] \\
%			& = \sum_{q=1}^{Q}w_K^{(q)}(\rho\mfu)_h^{e}(\mfx_K^{(q)})v(\mfx_K^{(q)}) = \mathbf{0}.
		\end{align*}
		where 
		\begin{equation*}
			\mfI = \Delta t\left[ \sum_{q=1}^{Q}w_K^{(q)}p_h^{e}(\mfx_K^{(q)})\nabla v(\mfx_K^{(q)}) - \dfrac{\overline{(\rho^n_h)_K}}{\overline{(\rho^e_h)_K}}\sum_{q=1}^{Q}w_K^{(q)}p_h^{e}(\mfx_K^{(q)})\nabla v(\mfx_K^{(q)}) \right] = \mathbf{0},
		\end{equation*}
		\begin{equation*}
			\mfI\mfI = \Delta t\left[\sum_{q = 1}^{Q}w_{K}^{(q)} \left(\dfrac{\rho^n_h(\mfx_K^{(q)})}{\rho^e_h(\mfx_K^{(q)})} - \dfrac{\overline{(\rho^n_h)_K}}{\overline{(\rho^e_h)_K}}\right)\nabla p_h^e(\mfx_K^{(q)})v(\mfx_K^{(q)})\right] = \mathbf{0},
		\end{equation*}
		\begin{equation*}
			\mfI\mfI\mfI = \Delta t\left[- \sum_{\mE\in \partial K}\sum_{\mu=1}^{N}w_{\mE}^{(\mu)}p_h^{e,*}(\mfx_{\mE}^{(\mu)})\mfn_{\mE, K}v^{\text{int}(K)}(\mfx_\mE^{(\mu)}) + \dfrac{\overline{(\rho^n_h)_K}}{\overline{(\rho^e_h)_K}}\sum_{\mE\in \partial K}\sum_{\mu=1}^{N}w_{\mE}^{(\mu)}p_h^{e,*}(\mfx_{\mE}^{(\mu)})\mfn_{\mE, K}v^{\text{int}(K)}(\mfx_\mE^{(\mu)})\right] = \mathbf{0},
		\end{equation*}
		\begin{equation*}
			\mfI\mfV = \Delta t\left[- \sum_{q = 1}^{Q}w_K^{(q)}\rho_h(x_K^{(q)})\mfg^{\delta}_h(\mfx_K^{(q)})v(\mfx_K^{(q)}) \right] = \mathbf{0}.
		\end{equation*}
		Therefore, we have that
		\begin{equation*}
			(\rho\mfu)_h^{n+1} = (\rho\mfu)_h^{n} = \mathbf{0},
		\end{equation*}
		In conclusion, we summarize that
		\begin{equation}
			\rho_h^{n+1}= \rho_h^e,\; (\rho\mfu)_h^{n+1} = \mathbf{0},\; 
			(E_{tot})_h^{n+1} = \dfrac{p_h^e}{\gamma-1} + \mfP(\dfrac{1}{2}\rho_h^e\phi_h^e).
		\end{equation}
		Namely, the WB property is preserved at the next time level $t^{n+1}$. The proof is completed.
	\end{itemize}
	
\end{proof}

%%%%%%%%%%%%%%%%%%%%%%%%%%%%%%%%%%
%
%%%%%%%%%%%%%%%%%%%%%%%%%%%%%%%%%%

\end{document}